\documentclass[11pt]{article}

\let\oldemptyset\emptyset

\usepackage{amsmath,amssymb,amsfonts,mathabx,enumitem,amsthm,accents,pstricks,pstricks-add,mathrsfs,xy}
\xyoption{all}

\numberwithin{equation}{section}

\newtheorem{numthm}{Theorem}

\newtheorem{thm}{Theorem}[section]
\newtheorem{lmm}[thm]{Lemma}
\newtheorem{prp}[thm]{Proposition}
\newtheorem{crl}[thm]{Corollary}
\theoremstyle{definition}
\newtheorem{dfn}[thm]{Definition}
\newtheorem{eg}[thm]{Example}

\def\BE#1{\begin{equation}\label{#1}}
\def\EE{\end{equation}}
\def\eref#1{(\ref{#1})}

\def\lra{\longrightarrow}

\def\lr#1{\langle{#1}\rangle}
\def\blr#1{\big\langle{#1}\big\rangle}
\def\ov#1{\overline{#1}}

\def\wt#1{\widetilde{#1}}
\def\sf#1{\textsf{#1}}

\def\sm#1{\begin{small}{#1}\end{small}}
\def\tn#1{\textnormal{#1}}
\def\wh#1{\widehat{#1}}
\def\wch#1{\widecheck{#1}}

\def\de{\delta}
\def\ep{\epsilon}

\def\io{\iota}

\def\la{\lambda}

\def\om{\omega}

\def\th{\theta}
\def\ve{\varepsilon}
\def\vp{\varpi}
\def\vr{\varrho}
\def\vph{\varphi}

\def\ze{\zeta}

\def\De{\Delta}
\def\La{\Lambda}

\def\Th{\Theta}

\def\bI{\mathbb I}

\def\C{\mathbb C}

\def\sC{\mathscr C}

\def\E{\textnormal{E}}
\def\bF{\mathbb F}

\def\bI{\mathbb I}
\def\cK{\mathcal K}

\def\cN{\mathcal N}
\def\cO{\mathcal O}
\def\P{\mathbb P}
\def\cP{\mathcal P}

\def\R{\mathbb R}

\def\T{\mathbb T}
\def\cU{\mathcal U}
\def\cW{\mathcal W}
\def\X{\mathbf X}

\def\fX{\mathfrak X}

\def\cZ{\mathcal Z}
\def\Z{\mathbb Z}
\def\sZ{\mathscr Z}

\def\fd{\mathfrak d}

\def\fh{\mathfrak h}

\def\fs{\mathfrak s}
\def\ft{\mathfrak t}

\def\ne{\textnormal{e}}
\def\fI{\mathfrak i}

\def\Ann{\textnormal{Ann}}
\def\id{\textnormal{id}}
\def\Dom{\textnormal{Dom}}
\def\Id{\tn{Id}}
\def\Im{\tn{Im}}

\def\nd{\textnormal{d}}

\def\Span{\textnormal{Span}}
\def\Symp{\textnormal{Symp}}

\def\i{\infty}
\def\w{\wedge}

\def\eset{\oldemptyset}
\def\prt{\partial}
\def\1{\mathbf 1}
\def\hb{\hbar}

\def\bu{\bullet}

\begin{document}

\title{Normal Crossings Degenerations of Symplectic Manifolds}
\author{Mohammad F.~Tehrani and 
Aleksey Zinger\thanks{Partially supported by NSF grant 1500875 and MPIM}}

\date{\today}
\maketitle

\begin{abstract}
\noindent
We use local Hamiltonian torus actions to degenerate 
a symplectic manifold to a normal crossings symplectic variety 
in a smooth one-parameter family.
This construction, motivated in part by the Gross-Siebert and B.~Parker's programs, 
contains a multifold version of the usual (two-fold) symplectic cut construction
and in particular splits a symplectic manifold into several symplectic manifolds
containing normal crossings symplectic divisors with shared irreducible components
in one~step.
\end{abstract}
\tableofcontents

\section{Introduction}
\label{intro_sec}

\noindent
Flat one-parameter families of degenerations are an important tool in algebraic geometry 
and raise considerable interest in related areas of symplectic topology.
The Gross-Siebert program~\cite{GS0} for a direct proof of mirror symmetry 
has highlighted in particular the significance of log smooth degenerations 
to log smooth algebraic varieties. 
A central part of this program is the study of Gromov-Witten invariants 
(which are fundamentally symplectic topology invariants) under such degenerations.
It is undertaken from an algebro-geometric perspective in \cite{AC,Ch,GS}.
The almost complex analogue of the log smooth category provided by 
the exploded manifold category of~\cite{Brett0} underlines 
a similar study of Gromov-Witten invariants in~\cite{Brett}.
Log smooth varieties include varieties with \sf{normal crossings} (or \sf{NC}) singularities,
i.e.~singularities of the form $z_1\!\ldots\!z_N\!=\!0$ in complex coordinates.
Purely symplectic topology notions of NC symplectic variety and of smooth one-parameter
family of degenerations to such a variety are introduced in \cite{SympDivConf,SympSumMulti}. 
The main construction of this paper uses a collection of compatible Hamiltonian torus actions,
i.e.~an \sf{$N$-fold cutting configuration}  in the sense of Definition~\ref{SympCut_dfn1},
to degenerate a symplectic manifold to an NC~symplectic variety 
in a smooth one-parameter family.\\

\noindent
The now classical symplectic cut construction of~\cite{L} decomposes 
a symplectic manifold~$(X,\om)$ into two symplectic manifolds,
$(X_-,\om_-)$ and~$(X_+,\om_+)$, using an  $S^1$-action with a Hamiltonian~$h$
on an~open subset~$W$ of~$X$.
This construction cuts~$X$ into closed subsets~$U^{\le}$ and~$U^{\ge}$ 
along a separating real hypersurface $\wt{V}\!=\!h^{-1}(a)$
and collapses their boundary~$\wt{V}$
to a smooth symplectic divisor $V\!=\!\wt{V}/S^1$ inside~$(X_-,\om_-)$ and~$(X_+,\om_+)$.
The associated ``wedge"  
$$X_{\eset}\equiv X_-\!\!\cup_{V}\!\!X_+$$
is  an \sf{SC~symplectic variety} associated with 
a 2-fold \sf{simple crossings} (or \sf{SC}) \sf{symplectic configuration} 
in the sense of Definition~\ref{SCC_dfn}.\\

\noindent
The main construction of this paper, described in Section~\ref{Thm12Pf_sec},
contains a multifold version of the construction of~\cite{L}.
We use an $N$-fold cutting configuration~$\sC$ in particular
to decompose a symplectic manifold~$(X,\om)$ into $N$~symplectic manifolds
$(X_i,\om_i)$ at~once.
This construction cuts~$X$ into closed subsets~$U_i^{\le}$ along separating real hypersurfaces 
$U_{ij}^{\le}\!=\!h_{ij}^{-1}(0)$ for  Hamiltonians~$h_{ij}$ generating 
$S^1$-actions on open subsets~$W_{ij}$ of~$X$.
These subsets~$U_i^{\le}$ have boundary and corners
$$ U_I^{\le} \equiv \bigcap_{j\in I-i}\!\!\!U_{ij}^{\le}\subset U_i^{\le}\,,
\qquad \{i\}\subsetneq I \subset \big\{1,\ldots,N\big\}.$$
We collapse these~$U_I^{\le}$ to 
symplectic submanifolds $X_I\!\subset\!X_i$ with $I\!\ni\!i$
which form an \sf{SC~symplectic divisor} in~$X_i$
in the sense of  Definition~\ref{SCD_dfn}. 
The entire collection $\{X_I\}$ determines an SC~symplectic variety~$X_{\eset}$.
This output of our main construction corresponds
to the two middle statements in Theorem~\ref{SympCut_thm12} in Section~\ref{SympCutThm_subs2}.
The output described by the remainder of Theorem~\ref{SympCut_thm12} and
by Theorem~\ref{SympCut_thm3} 
endows the basic output of Theorem~\ref{SympCut_thm12}
with a rich geometric structure desirable for a range of applications;
this structure is summarized~below.\\

\noindent
The normal bundles of the symplectic divisor~$V$ in the symplectic manifolds
$(X_-,\om_-)$ and~$(X_+,\om_+)$ arising from a symplectic cut construction of~\cite{L}
are canonically dual.
The symplectic sum construction of~\cite{Gf} thus determines a deformation equivalence class
of \sf{nearly regular symplectic fibrations} 
\BE{picZtoC_e} \big(\cZ,\om_{\cZ},\pi\!:\cZ\!\lra\!\C\big) 
\qquad\hbox{s.t.}\quad  \pi^{-1}(0)=X_{\eset}\subset\cZ\,;\EE
see Definition~\ref{SimpFibr_dfn}.
If~$\wt{V}$ is compact, the symplectic deformation equivalence class of 
a fiber \hbox{$\cZ_{\la}\!\equiv\!\pi^{-1}(\la)$ of~$\pi$} is independent of
the choice of $\la\!\in\!\C^*$ sufficiently small.
It is then called a \sf{symplectic sum} $X_-\!\#_V\!X_+$ of~$X_-$ and~$X_+$ and is 
symplectically deformation equivalent to~$(X,\om)$.
The constructions of~\cite{L} and~\cite{Gf} together thus provide 
a symplectic topology analogue of the algebro-geometric notion of 
smooth one-parameter family of degenerations of a smooth algebraic variety
to two smooth algebraic varieties joined along a smooth algebraic divisor.\\

\noindent
The main construction of this paper uses an $N$-fold cutting configuration~$\sC$ 
for~$(X,\om)$ to produce 
a symplectic manifold~$(\cZ,\om_{\cZ})$ which contains the tuple $(X_i,\om_i)_{i=1}^N$
of the symplectic manifolds cut out from~$X$ by~$\sC$ as an~SC symplectic divisor.
The cutting configuration~$\sC$ also determines an $N$-fold Hamiltonian configuration
for~$(\cZ,\om_{\cZ})$ in the sense of Definition~\ref{SympCut_dfn0} and 
a deformation equivalence class of maps as in~\eref{picZtoC_e}
that restrict to nearly regular symplectic fibrations on sufficiently small 
neighborhoods~$\cZ'$ of~$X_{\eset}$.
One implication for the well-known $N\!=\!2$ case is that the domain of 
the map~$\pi$ in~\eref{picZtoC_e} can be taken inside of a symplectic manifold
{\it completely} determined by the data $(W,\phi,h,a)$ of the symplectic cut construction
of~\cite{L}.
If~$X$ is compact, then the fiber of~$\pi|_{\cZ'}$ over every sufficiently small
value $\la\!\in\!\R^+$ is canonically isomorphic  to the original symplectic 
manifold~$(X,\om)$ with the cutting configuration~$\sC$; see Theorem~\ref{SympCut_thm3}.
The full output of the constructions of Sections~\ref{Thm12Pf_sec} and~\ref{SumCutSm_sec},
described in Theorems~\ref{SympCut_thm12} and~\ref{SympCut_thm3},
thus provides 
a symplectic topology analogue of the algebro-geometric notion of 
smooth one-parameter family of degenerations of a smooth algebraic variety
to an NC algebraic variety.\\

\noindent
By \cite[Proposition~5.1]{SympSumMulti}, a fibration $\pi\!:\cZ'\!\lra\!\C$
as above determines a homotopy class of trivializations of 
the \sf{normal bundle~$\cO_{X_{\prt}}(X_{\eset})$} of the singular locus~$X_{\prt}$ of~$X_{\eset}$.
We show in~\cite{SympSumCut} that 
the multifold symplectic sum/smoothing construction of~\cite{SympSumMulti} and
the multifold symplectic cut/degeneration construction 
of Sections~\ref{Thm12Pf_sec} and~\ref{SumCutSm_sec}
are mutual inverses as operations between  
the deformation equivalences classes  of compact SC~symplectic varieties with  
trivializations of the normal bundle of the singular locus
and of compact symplectic manifolds with  cutting configurations.
This can be seen explicitly in the basic local setting of Section~\ref{BasicEg_subs}.\\

\noindent
For $N\!\in\!\Z^+$, we define
$$[N]=\big\{1,\ldots,N\!\big\}, \quad 
\cP^*(N)=\big\{I\!\subset\![N]\!:I\!\neq\!\eset\big\},\quad
(S^1)^N_{\bu}=\bigg\{\!(\ne^{\fI\th_i})_{i\in[N]}\!\in\!(S^1)^N\!\!:
\prod_{i\in[N]}\!\!\!\ne^{\fI\th_i}\!=\!1\!\bigg\}.$$
The Lie algebra of the  codimension~1 subtorus $(S^1)^N_{\bu}\!\subset\!(S^1)^N$
and its dual are given~by
\BE{ftNnu_e}\ft_{N;\bu}=
\big\{(r_i)_{i\in[N]}\!\in\!\R^N\!:\sum_{i\in[N]}\!\!\!r_i\!=\!0\big\}
\quad\hbox{and}\quad
\ft_{N;\bu}^*= \R^N\big/\big\{(a,\ldots,a)\!\in\!\R^N\!:\,a\!\in\!\R\big\},\EE
respectively.
For $I\!\subset\![N]$, we identify $(S^1)^I$ with the subgroup
$$\big\{(\ne^{\fI\th_i})_{i\in[N]}\!\in\!(S^1)^N\!: \ne^{\fI\th_i}\!=\!1~\forall\,
i\!\in\![N]\!-\!I\big\}$$
of $(S^1)^{N}$ in the natural way and let
$$(S^1)^I_{\bu}\!\equiv\!(S^1)^N_{\bu}\!\cap\!(S^1)^I\,.$$
Denote by $\ft_{I;\bu}\!\subset\!\ft_{N;\bu}$ the Lie algebra of $(S^1)^I_{\bu}$
and by~$\ft_{I;\bu}^*$ its dual.
For $i,j\!\in\!I\!\subset\![N]$, the homomorphism
$$\ft_{I;\bu}^*\!=\!\R^I\big/\big\{(a,\ldots,a)\!\in\!\R^I\!:\,a\!\in\!\R\big\}\lra\R,
\qquad \eta\!\equiv\!\big[(a_k)_{k\in I}\big]\lra \eta_{ij}\!\equiv\!a_j\!-\!a_i\,,$$
is well-defined.
We write $(\eta)_i\!<\!(\eta)_j$ (resp.~$(\eta)_i\!\le\!(\eta)_j$, $(\eta)_i\!=\!(\eta)_j$) 
if $0\!<\!\eta_{ij}$  (resp.~$0\!\le\!\eta_{ij}$, $0\!=\!\eta_{ij}$).

\begin{dfn}\label{SympCut_dfn0}
Let $N\!\in\!\Z^+$ and $(X,\om)$ be a  symplectic manifold.
An \sf{$N$-fold Hamiltonian configuration for~$(X,\om)$} is a tuple 
\BE{SympCutDfn_e}
\sC\equiv\big(U_I,\phi_I,\mu_I\big)_{I\in\cP^*(N)},\EE
where $(U_I)_{I\in\cP^*(N)}$ is an open cover of~$X$ and $\phi_I$  
is a Hamiltonian $(S^1)^I_{\bu}$-action on~$U_I$ with moment map~$\mu_I$, such~that 
\begin{enumerate}[label=$(\alph*)$,leftmargin=*]

\item\label{IJinter_it} $U_I\!\cap\!U_J\!=\!\eset$ unless $I\!\subset\!J$ or $J\!\subset\!I$;

\item\label{phiIUJ_it} $\mu_J(x)|_{\ft_{I;\bu}}\!=\!\mu_I(x)$
for all $x\!\in\!U_I\!\cap\!U_J$ and $I\!\subset\!J\!\subset\![N]$;

\item\label{UIJpos_it} $(\mu_J(x))_i\!<\!(\mu_J(x))_j$ 
for all $x\!\in\!U_I\!\cap\!U_J$, $i\!\in\!I\!\subset\!J\!\subset\![N]$, 
and $j\!\in\!J\!-\!I$.
\end{enumerate}
\end{dfn}

\begin{dfn}\label{SympCut_dfn1}
An \sf{$N$-fold cutting configuration for~$(X,\om)$} is 
an  $N$-fold Hamiltonian configuration  as in~\eref{SympCutDfn_e}
such~that the restriction of the $(S^1)^I_{\bu}$-action~$\phi_I$ to $(S^1)^{I_0}_{\bu}$ 
is free on the preimage of $0\!\in\!\ft_{I_0;\bu}^*$ under the moment~map
\BE{hdfn_e} 
\mu_{I_0;I}\!: \big\{x\!\in\!U_I\!\!:
\big(\mu_I(x)\!\big)_{\!i}\!<\!\big(\mu_I(x)\!\big)_{\!j}
~\forall\,i\!\in\!I_0,\,j\!\in\!I\!-\!I_0\big\}\lra \ft_{I_0;\bu}^*,~~
\mu_{I_0;I}(x)\!=\!\mu_I(x)\big|_{\ft_{I_0;\bu}}\,,\EE
for all \hbox{$I_0\!\subset\!I\!\subset\![N]$} with $I_0\!\neq\!\eset$.
\end{dfn}

\noindent
We specify our conventions concerning moment maps for Hamiltonian actions on symplectic manifolds
and identifications of Lie algebras in Section~\ref{TorusAct_subs}. 
An $N$-fold Hamiltonian configuration is determined by the $2^N\!-\!N\!\!-\!1$ actions~$\phi_I$
by the non-trivial subtori $(S^1)^I_{\bu}\!\subset\!(S^1)^N$ and their moment maps~$\mu_I$
on open subsets~$U_I$ of~$X$.
As described in Section~\ref{SympCutThm_subs1}, such a collection can alternatively be
specified by $\binom{N}{2}$  Hamiltonian $S^1$-actions~$\phi_{ij}$ and their Hamiltonians~$h_{ij}$
on (generally) larger open subsets of~$X$.
The usual symplectic cut construction of~\cite{L} is the $N\!=\!2$ case of this alternative description,
which identifies $(S^1)^2_{\bu}$  with~$S^1$ by projection to one of the components of~$(S^1)^2$.
Simple examples of $N$-fold Hamiltonian and cutting configurations are described 
in Section~\ref{BasicEg_subs}.
The output of the constructions of Sections~\ref{Thm12Pf_sec} and~\ref{SumCutSm_sec}
for the cutting configurations of Section~\ref{BasicEg_subs} can be readily identified;
see Proposition~\ref{BasicEg_prp}.\\

\noindent
If the domain~$U_{[N]}$ of the action~$\phi_{[N]}$ by the largest subtorus
$(S^1)^N_{\bu}\!\subset\!(S^1)^N$ is the entire manifold~$X$,
the remaining actions~$\phi_I$ and moment maps~$\mu_I$ are the restrictions 
of~$\phi_{[N]}$ and~$\mu_{[N]}$, respectively, to~$U_I$.
As described in Section~\ref{HamSpaGen_subs},
the symplectic manifolds~$(\cZ,\om_{\cZ})$ and~$(X_i,\om_i)$ are then obtained
through a single application of the symplectic reduction of \cite{Meyer,MaW}.
There is also a natural  nearly regular fibration~$\pi$ as in~\eref{picZtoC_e} 
defined on the entire manifold~$\cZ$.
If in addition $X$ is compact, $(\cZ,\om_{\cZ},\pi)$ can be replaced by 
a compact symplectic manifold $(\wh\cZ_a,\om_{\wh\cZ;a})$ for each $a\!\in\!\R^+$ 
sufficiently large and an $S^1$-equivariant nearly regular symplectic fibration
\BE{piwhcZtoP1_e}\wh\pi\!:\wh\cZ_a\lra\P^1
\qquad\hbox{s.t.}\quad  \wh\pi^{-1}(0)=X_{\eset}\subset\wh\cZ_a\,;\EE
see Section~\ref{CmptcZ_subs}.
In Sections~\ref{MomPolytThm_subs} and~\ref{AdmissDecomp_subs}, 
we relate on the setup for and the output of the main constructions of this paper
to the moment polytopes arising from the Atiyah-Guillemin-Sternberg Convexity Theorem;
see Theorem~\ref{toriccut_thm} and Proposition~\ref{CombCond_prp}.\\

\noindent
The notions of normal crossings symplectic singularities
and of smoothings of such singularities introduced in~\cite{SympDivConf}
and~\cite{SympSumMulti}, respectively,
are recalled in Section~\ref{review_sec}.
Section~\ref{SympCutThm_subs2} contains the main statements of this paper,
Theorems~\ref{SympCut_thm12} and~\ref{SympCut_thm3};
they describe the output of our multifold symplectic cut/degeneration construction.
The SC symplectic configuration determined by a cutting configuration
and the symplectic manifold containing the associated SC~symplectic 
variety are constructed in Section~\ref{Thm12Pf_sec};
this establishes Theorem~\ref{SympCut_thm12}.
Section~\ref{SumCutSm_sec} endows a neighborhood of this symplectic variety
with the structure of a one-parameter family of smoothings
and establishes Theorem~\ref{SympCut_thm3}.\\

\noindent
We would like to thank  A.~Cannas da Silva, M.~McLean,
B.~Parker, and A.~Pires for related discussions
and E.~Lerman for pointing out related literature.

\section{Notation and terminology}
\label{review_sec}

\noindent
For $I\!\subset\![N]$, define 
$$\cP_I(N)=\big\{J\!\in\!\cP^*(N)\!:\,I\!\subset\!J\big\}.$$ 
For $i\!\in\![N]$, we write $\cP_{\{i\}}(N)$ as~$\cP_i(N)$.
For $i,j\!\in\![N]$ distinct, we  write~$\cP_{\{i,j\}}(N)$ as~$\cP_{ij}(N)$.

\subsection{Normal crossings symplectic varieties}
\label{SCdivconf_subs}

\noindent
We now recall the notions of simple crossings (or \sf{SC}) symplectic divisor
and variety introduced, described in more detail, and 
illustrated with examples in \cite[Section~2.1]{SympDivConf}.\\

\noindent
Let $X$ be a (smooth) manifold. 
For any submanifold $V\!\subset\!X$, let
$$\cN_XV\equiv \frac{TX|_V}{TV}\lra V$$
denote the normal bundle of~$V$ in~$X$.
For a collection $\{V_i\}_{i\in S}$ of submanifolds of~$X$ and $I\!\subset\!S$, let
$$V_I\equiv \bigcap_{i\in I}\!V_i\subset X\,.$$
Such a collection  is called \sf{transverse} if any subcollection $\{V_i\}_{i\in I}$ 
of these submanifolds intersects transversely, i.e.~the homomorphism
\BE{TransVerHom_e}
T_xX\oplus\bigoplus_{i\in I}T_xV_i\lra \bigoplus_{i\in I}T_xX, \qquad
\big(v,(v_i)_{i\in I}\big)\lra (v\!+\!v_i)_{i\in I}\,,\EE
is surjective for every $x\!\in\!V_I$. 
Each subspace $V_I\!\subset\!X$ is then a submanifold of~$X$.\\

\noindent
If $X$ is an oriented manifold,
a transverse collection $\{V_i\}_{i\in S}$ of oriented submanifolds of~$X$
of even codimensions  induces an orientation on each submanifold $V_I\!\subset\!X$
with $|I|\!\ge\!2$, which we call \sf{the intersection orientation of~$V_I$}.
If $V_I$ is zero-dimensional, it is a discrete collection of points in~$X$
and the homomorphism~\eref{TransVerHom_e} is an isomorphism at each point $x\!\in\!V_I$;
the intersection orientation of~$V_I$ at $x\!\in\!V_I$
then corresponds to a plus or minus sign, depending on whether this isomorphism
is orientation-preserving or orientation-reversing.
For convenience, we call the original orientations of 
$X\!=\!V_{\eset}$ and $V_i\!=\!V_{\{i\}}$ \sf{the intersection orientations}
of these submanifolds~$V_I$ of~$X$ with $|I|\!<\!2$.\\

\noindent
Suppose $(X,\om)$ is a symplectic manifold and $\{V_i\}_{i\in S}$ is a transverse collection 
of submanifolds of~$X$ such that each $V_I$ is a symplectic submanifold of~$(X,\om)$.
Each $V_I$ then carries an orientation induced by $\om|_{V_{I}}$,
which we  call the \sf{$\om$-orientation}.
If $V_I$ is zero-dimensional, it is automatically a symplectic submanifold of~$(X,\om)$;
the $\om$-orientation of~$V_I$ at each point $x\!\in\!V_I$ corresponds to the plus sign 
by definition.
By the previous paragraph, the $\om$-orientations of~$X$ and~$V_i$ with $i\!\in\!I$
also induce intersection orientations on all~$V_I$.
By definition, the intersection and symplectic orientations of~$V_I$ agree if $|I|\!<\!2$.

\begin{dfn}\label{SCD_dfn}
Let $(X,\om)$ be a symplectic manifold.
An  \sf{SC symplectic divisor} 
in~$(X,\om)$ is a finite transverse collection 
$\{V_i\}_{i\in S}$ of closed submanifolds of~$X$ of codimension~2 such that 
$V_I$ is a symplectic submanifold of~$(X,\om)$ for every $I\!\subset\!S$
and the intersection and $\om$-orientations of~$V_I$~agree.
\end{dfn}

\begin{dfn}\label{SCdivstr_dfn}
Let $X$ be a manifold and $\{V_i\}_{i\in S}$ be a finite transverse collection of 
closed submanifolds of~$X$ of codimension~2.
A \sf{symplectic structure on $\{V_i\}_{i\in S}$ in~$X$} is a symplectic form~$\om$ 
on~$X$ such that $V_I$ is a symplectic submanifold of $(X,\om)$ for all $I\!\subset\!S$.
\end{dfn}

\noindent
For $X$ and $\{V_i\}_{i\in S}$ as in Definition~\ref{SCdivstr_dfn}, 
we denote by $\Symp^+(X,\{V_i\}_{i\in S})$ the space of all symplectic structures~$\om$ 
on $\{V_i\}_{i\in S}$ in~$X$ such that $\{V_i\}_{i\in S}$
is an SC symplectic divisor in~$(X,\om)$.

\begin{dfn}\label{TransConf_dfn1}
Let $N\!\in\!\Z^+$.
An \sf{$N$-fold transverse configuration} is a tuple $\{X_I\}_{I\in\cP^*(N)}$
of manifolds such that $\{X_{ij}\}_{j\in[N]-i}$ is a transverse collection 
of submanifolds of~$X_i$ for each $i\!\in\![N]$ and
$$X_{\{ij_1,\ldots,ij_k\}}\equiv \bigcap_{m=1}^k\!\!X_{ij_m}
=X_{ij_1\ldots j_k}\qquad\forall~j_1,\ldots,j_k\in[N]\!-\!i.$$
\end{dfn}

\begin{dfn}\label{TransConf_dfn2}
Let $N\!\in\!\Z^+$ and $\X\!\equiv\!\{X_I\}_{I\in\cP^*(N)}$ be an $N$-fold transverse configuration
such that $X_{ij}$ is a closed submanifold of~$X_i$ of codimension~2
for all $i,j\!\in\![N]$ distinct.
A \sf{symplectic structure on~$\X$} is a~tuple 
$$(\om_i)_{i\in[N]}\in 
\prod_{i=1}^N\Symp\big(X_i,\{X_{ij}\}_{j\in[N]-i}\big)$$
such that $\om_{i_1}|_{X_{i_1i_2}}\!=\!\om_{i_2}|_{X_{i_1i_2}}$ for all $i_1,i_2\!\in\![N]$.
\end{dfn}

\noindent 
For $\X\!\equiv\!\{X_I\}_{I\in\cP^*(N)}$ as in Definition~\ref{TransConf_dfn1}, 
define
\begin{gather}
\label{Xesetdfn_e}
X_{\eset}=\bigg(\bigsqcup_{i=1}^NX_i\bigg)\bigg/\!\!\sim, \quad
X_i\ni x\sim x\in X_j~~\forall~x\in X_{ij}\subset X_i,X_j,~i\neq j\,,\\
\label{Xprtdfn_e}
X_{\prt}\equiv\bigcup_{\begin{subarray}{c}I\subset[N]\\ |I|=2\end{subarray}}
\!\!\!X_I\subset X_{\eset}\,.
\end{gather}
For $\X$ as in Definition~\ref{TransConf_dfn2}, 
denote by $\Symp^+(\X)$ the space of all symplectic structures $(\om_i)_{i\in[N]}$ 
on $\X$ such that 
$\{X_{ij}\}_{j\in[N]-i}$ is an SC symplectic divisor in $(X_i,\om_i)$
for each $i\!\in\![N]$.

\begin{dfn}\label{SCC_dfn}
Let $N\!\in\!\Z^+$.
An \sf{$N$-fold simple crossings} (or \sf{SC}) \sf{symplectic configuration} 
is a~tuple 
\BE{SCCdfn_e}\X=\big((X_I)_{I\in\cP^*(N)},(\om_i)_{i\in[N]}\big)\EE
such that $\{X_I\}_{I\in\cP^*(N)}$ is an $N$-fold transverse configuration,
$X_{ij}$ is a closed submanifold of~$X_i$ of codimension~2
for all $i,j\!\in\![N]$ distinct, and
$(\om_i)_{i\in[N]}\in\Symp^+(\X)$.
The \sf{SC symplectic variety associated~to} such a tuple~$\X$ 
is the pair~$(X_{\eset},(\om_i)_{i\in[N]})$.
\end{dfn}

\noindent
The basic local example of an SC symplectic variety is the union of the $N$~coordinate 
hyperplanes in~$\C^N$ with the restrictions of the standard symplectic form on~$\C^N$,
i.e.~the $\la\!=\!0$ case of the hypersurface~$X$ in~\eref{Xladfn_e}.
This is the SC symplectic variety associated to the SC symplectic configuration~\eref{BasicSCC_e}
and is central fiber of the  one-parameter family~\eref{BasicSCCsm_e}
 of degenerations
of the smooth symplectic manifolds~$X$ in~\eref{Xladfn_e} with $\la\!\neq\!0$ and 
arises from the basic local symplectic cutting configuration of Section~\ref{BasicEg_subs}.

\subsection{Divisors, line bundles, and smoothability}
\label{SympSmooth_subs}

\noindent
Suppose $(X,\om)$ is a symplectic manifold 
and $V\!\subset\!X$ is a smooth symplectic divisor, 
i.e.~$|S|\!=\!1$ in the notation of Definition~\ref{SCD_dfn}.
The normal bundle of~$V$ in~$X$,
\BE{cNXVsymp_e}
\cN_XV\equiv \frac{TX|_V}{TV}\approx TV^{\om}
\equiv \big\{v\!\in\!T_xX\!:\,x\!\in\!V,\,\om(v,w)\!=\!0~\forall\,w\!\in\!T_xV\big\}
\lra V,\EE
then inherits a fiberwise symplectic form~$\om|_{\cN_XV}$ from~$\om$.
The space of complex structures on the fibers of~\eref{cNXVsymp_e}
compatible with (resp.~tamed~by)  $\om|_{\cN_XV}$ is non-empty and contractible;
we call such complex structures \sf{$\om$-compatible} (resp.~\sf{$\om$-tame}).
Fix an identification~$\Psi$ 
of a tubular neighborhood~$D_X^{\ep}V$ of~$V$ in~$\cN_XV$ 
with a tubular neighborhood of~$V$ in~$X$ 
(i.e.~a \sf{regularization} of~$V$ in~$X$ in the sense of 
\cite[Definition~2.8]{SympDivConf}) and an
 $\om$-tame complex structure~$\fI$ on~$\cN_XV$.
Let
\begin{gather}
\label{cOXVdfn_e}
\cO_X(V)= \big(\Psi^{-1\,*}\pi_{\cN_XV}^*\cN_XV|_{\Psi(D_X^{\ep}V)} \sqcup 
(X\!-\!V)\!\times\!\C\big)\big/\!\!\sim,\\
\notag
\Psi^{-1\,*}\pi_{\cN_XV}^*\cN_XV|_{\Psi(D_X^{\ep}V)} 
  \ni\big(\Psi(v),v,cv\big)\sim\big(\Psi(v),c\big)\in (X\!-\!V)\!\times\!\C.
\end{gather}
This is a complex line bundle over~$X$.
The space of pairs~$(\Psi,\fI)$ involved in explicitly constructing this line bundle
is contractible.\\

\noindent
Suppose $\X$ is an SC symplectic configuration as in~\eref{SCCdfn_e}. 
If $i,j,k\!\in\![N]$ are distinct, the inclusion
$(X_{jk},X_{ijk})\!\lra\!(X_j,X_{ij})$ induces an isomorphism
$$\cN_{X_{jk}}X_{ijk}\equiv \frac{TX_{jk}|_{X_{ijk}}}{TX_{ijk}}
\lra \frac{TX_j|_{X_{ijk}}}{TX_{ij}|_{X_{ijk}}}
\equiv \cN_{X_j}X_{ij}\big|_{X_{ijk}}$$
of rank~2 oriented real vector bundles over~$X_{ijk}$;
see the last third of \cite[Section~2.1]{SympSumMulti}.
In particular, the rank~2 oriented real vector bundles $\cN_{X_j}X_{ij}|_{X_{ijk}}$ 
and $\cN_{X_k}X_{ik}|_{X_{ijk}}$ are canonically identified with~$\cN_{X_{jk}}X_{ijk}$.
We can thus choose a collection 
$$\Psi_{ij;j}\!: \cN_{ij;j}'\lra X_j, \qquad i,j\!\in\![N],\,i\!\neq\!j,$$
of identifications of tubular neighborhoods of~$X_{ij}$ in~$\cN_{X_j}X_{ij}$
and in~$X_j$ and a collection of $\om_j$-tame complex structures~$\fI_{ij;j}$ on
the vector bundles~$\cN_{ij;j}$ so~that
\BE{restrisom_e3}\Psi_{ij;j}\big|_{\cN_{ij;j}'\cap\cN_{X_{jk}}X_{ijk}}= 
\Psi_{ik;k}\big|_{\cN_{ik;k}'\cap\cN_{X_{jk}}X_{ijk}}
\quad\hbox{and}\quad 
\fI_{ij;j}\big|_{\cN_{X_{jk}}X_{ijk}}=\fI_{ik;k}\big|_{\cN_{X_{jk}}X_{ijk}}\EE
for all $i,j,k\!\in\![N]$ with $k,j\!\neq\!i$.\\

\noindent
For $i,j\!\in\![N]$ distinct, let $\cO_{X_j}(X_{ij})$ be the complex line bundle
over~$X_j$ constructed as in~\eref{cOXVdfn_e} using the identification~$\Psi_{ij;j}$
and the complex structure~$\fI_{ij;j}$.
By~\eref{restrisom_e3}, there are canonical identifications
$$\cO_{X_j}(X_{ij})\big|_{X_{ijk}}=\cO_{X_{jk}}(X_{ijk})=
\cO_{X_k}(X_{ik})\big|_{X_{ijk}}$$
for all $i,j,k\!\in\![N]$ with $j,k\!\neq\!i$.
For each $i\!\in\![N]$,
\begin{gather}
\label{cOXXidfn_e}
\cO_{X_i^c}(X_i)\equiv\bigg(\bigsqcup_{j\in[N]-\{i\}}\hspace{-.2in}\cO_{X_j}(X_{ij})\bigg)
\!\!\bigg/\!\!\!\sim\,\lra 
X_i^c\!\equiv\bigcup_{j\in[N]-\{i\}}\hspace{-.2in}X_j\subset X_{\eset}\,,\\
\notag
\cO_{X_j}(X_{ij})\big|_{X_{ijk}} \ni u\sim u\in  \cO_{X_k}(X_{ik})\big|_{X_{ijk}}
\quad\forall\,i,j,k\!\in\![N],\,j,k\!\neq\!i,
\end{gather}
is thus a well-defined complex line bundle.
Let $\cO_{X_{\prt}}(X_i)\!=\!\cO_{X_i^c}(X_i)|_{X_{\prt}}$.
We call the complex line bundle
\BE{PsiDfn_e}
\cO_{X_{\prt}}(X_{\eset})\equiv\bigotimes_{i=1}^N\cO_{X_{\prt}}(X_i)\EE
the \sf{normal bundle} of the singular locus~$X_{\prt}$ in~$X_{\eset}$.
The space of the collections of pairs~$(\Psi_{ij;j},\fI_{ij;j})$ 
involved in explicitly constructing this line bundle is contractible.\\

\noindent
The notions of smooth families of varieties and of smoothings of singular varieties
play important roles in algebraic geometry.
Analogues of these notions in the category of~SC symplectic varieties are introduced
in~\cite{SympSumMulti}.

\begin{dfn}[{\cite[Definition~2.6]{SympSumMulti}}]\label{SimpFibr_dfn}
If $(\cZ,\om_{\cZ})$ is a symplectic manifold and $\De\!\subset\!\C$
is a disk centered at the origin,
a smooth surjective map $\pi\!:\cZ\!\lra\!\De$ is a \textsf{nearly regular symplectic fibration}~if
\begin{enumerate}[label=$\bullet$,leftmargin=*]

\item $\cZ_0\!\equiv\!\pi^{-1}(0)=X_1\!\cup\!\ldots\!\cup\!X_N$ 
for some SC symplectic divisor $\{X_i\}_{i\in[N]}$ in~$(\cZ,\om_{\cZ})$, 

\item  $\pi$ is a submersion outside of the submanifolds $X_I$ with $|I|\!\ge\!2$,

\item for every $\la\!\in\!\De\!-\!\{0\}$,
the restriction~$\om_{\la}$ of~$\om_{\cZ}$ to $\cZ_{\la}\equiv\pi^{-1}(\la)$ is nondegenerate.

\end{enumerate} 
\end{dfn}

\vspace{.1in}

\noindent
We call a nearly regular symplectic fibration as in Definition~\ref{SimpFibr_dfn} 
a \sf{one-parameter family of smoothings} of the SC variety $(X_{\eset},(\om_i)_{i\in[N]})$
associated to the SC symplectic configuration~\eref{SCCdfn_e} with
$$X_I=\bigcap_{i\in I}X_i\subset\cZ_0\subset\cZ
\qquad \forall~I\!\in\!\cP^*(N)\,.$$
By \cite[Theorem~2.7]{SympSumMulti}, {\it some} SC~symplectic variety
$(X_{\eset},(\om_i')_{i\in[N]})$
deformation equivalent to\linebreak $(X_{\eset},(\om_i)_{i\in[N]})$ 
admits a one-parameter family of smoothings
if and only if the line bundle~\eref{PsiDfn_e} admits a trivialization
(or equivalently its Chern class vanishes).
Furthermore, the germ of the deformation equivalence class of such a smoothing 
is determined by a homotopy class of trivializations of~\eref{PsiDfn_e}.\\

\noindent
If the SC symplectic variety $X_{\eset}$ associated with
an SC symplectic configuration~$\X$ is compact,
we call the fiber of a one-parameter family of smoothings of~$X_{\eset}$ 
over a point in a small punctured disk~$\De^*$ around the origin a \sf{symplectic sum} for~$\X$.
Each such fiber comes with a natural cutting configuration.
The concern of Corollary~\ref{SympCut_crl} with the trivializations of the line 
bundle~\eref{PsiDfn_e} is in preparation for showing in~\cite{SympSumCut}
that the symplectic cut/degeneration construction of this paper and
the symplectic sum/smoothing construction of~\cite{SympSumMulti} are mutual inverses 
as morphisms between appropriate categories.

\subsection{Families of symplectic varieties}
\label{SympFamil_subs}

\noindent
We also need family versions of Definitions~\ref{TransConf_dfn2} and~\ref{SimpFibr_dfn}, 
especially over the interval $\bI\!\equiv\![0,1]$.

\begin{dfn}\label{Manfam_dfn}
Let $B$ be a manifold, possibly with boundary.
\begin{enumerate}[label=(\arabic*),leftmargin=*]

\item\label{Manfam_it} A smooth map $\pi_{\fX}\!:\fX\!\lra\!B$ is
a \sf{family of manifolds over~$B$} if $\prt\fX\!=\!\pi_{\fX}^{-1}(\prt B)$ and
 $\pi_{\fX}$ is a submersion.

\item\label{ManSymp_it} A \sf{family of symplectic manifolds over~$B$} is a tuple 
$(\fX,\om_{\fX},\pi_{\fX})$, where $\pi_{\fX}\!:\fX\!\lra\!B$ is
a family of manifolds over~$B$ and  $\om_{\fX}$ is a 2-form on~$\fX$ such~that
$$\big(\fX_t,\om_t\big)\equiv \big(\pi_{\fX}^{-1}(t),\om|_{\fX_t}\big)$$
is a symplectic manifold for every $t\!\in\!B$.

\end{enumerate}
\end{dfn}

\begin{dfn}\label{SCCfam_dfn}
Let $N\!\in\!\Z^+$ and $B$ be a manifold, possibly with boundary.
\begin{enumerate}[label=(\arabic*),leftmargin=*]

\item\label{TCfam_it} A \sf{family of $N$-fold transverse configurations over~$B$} is 
a tuple $\{\pi_I\!:\fX_I\!\lra\!B\}_{I\in\cP^*(N)}$,
where  $\{\fX_I\}_{I\in\cP^*(N)}$ is an $N$-fold transverse configuration
and $\pi_I\!:\fX_I\!\lra\!B$ is a family of manifolds over~$B$ for each $I\!\in\!\cP^*(N)$
such that $\pi_I\!=\!\pi_i|_{\fX_I}$ for all $i\!\in\!I\!\subset\![N]$.

\item\label{SCCfam_it} A \sf{family of $N$-fold SC symplectic configurations over~$B$} is 
a tuple 
$$\big((\pi_I\!:\fX_I\!\lra\!B)_{I\in\cP^*(N)},(\om_i)_{i\in[N]}\big),$$
where $\{\pi_I\!:\fX_I\!\lra\!B\}_{I\in\cP^*(N)}$ is a 
family of $N$-fold transverse configurations over~$B$ 
and $\om_i$ is a 2-form on~$\fX_i$ for each $i\!\in\![N]$ such that
$$\big((\fX_{t;I})_{I\in\cP^*(N)},(\om_{t;i})_{i\in[N]}\big)
\equiv \big( (\pi_I^{-1}(t))_{I\in\cP^*(N)},(\om_i)_{i\in[N]}\big)$$ 
is an  $N$-fold SC symplectic configuration for every $t\!\in\!B$.

\end{enumerate}
\end{dfn}

\noindent
We call a family $(\om_t)_{t\in B}$ of 2-forms on a manifold~$X$ \sf{smooth} if 
the \sf{induced 2-form on $B\!\times\!X$} given~by
$$\om_{t,x}(v,w)=\begin{cases}
\om_t|_x(v,w),&\hbox{if}~v,w\!\in\!T_xX;\\
0,&\hbox{if}~v\!\in\!T_tB;
\end{cases}$$
is smooth.
Let $\X\!\equiv\!\{X_I\}_{I\in\cP^*(N)}$ be an $N$-fold transverse configuration,
$$\big(\om_{t;i}\big)_{i\in[N]}\in\Symp^+(\X) \qquad\forall~t\!\in\!B,$$
and  $\om_i$ be the 2-form on $B\!\times\!X_i$ induced by the family
$\big(\om_{t;i}\big)_{t\in B}$.
If the family $(\om_{t;i})_{t\in B}$ of 2-forms on~$X_i$ is smooth for every $i\!\in\![N]$
and $\pi_I\!:B\!\times\!X_I\!\lra\!X_I$ is the component projection map, then 
the~tuple 
$$\big((\pi_I\!:\fX_I\!\equiv\!B\!\times\!X_I\lra B)_{I\in\cP^*(N)},(\om_i)_{i\in[N]}\big)$$
is a family of $N$-fold SC symplectic configurations over~$B$.\\

\noindent 
Suppose the fibers of the submersions $\pi_I$ in Definition~\ref{SCCfam_dfn}\ref{TCfam_it}
are compact.
These projections are then locally trivial, i.e.~for every $t_0\!\in\!B$
there exist a neighborhood~$U$ of~$t_0$ in~$B$ and a diffeomorphism
$$\Xi_I\!:\fX_I|_U\!\equiv\!\pi_I^{-1}(U)\lra U\!\times\!\pi_I^{-1}(t_0)$$
such that $\pi_I|_{\fX_I|_U}\!=\!\pi_U\!\circ\!\Xi_I$,
where 
$$\pi_U\!:U\!\times\!\pi_I^{-1}(t_0)\lra U$$
is the component projection map.
By the same reasoning as in the proof of \cite[Proposition~4.2]{SympDivConf},
the diffeomorphisms~$\Xi_I$ can be chosen so~that
$$ \Xi_I=\Xi_i|_{\pi_I^{-1}(U)} \qquad\forall~i\!\in\!I\!\subset\![N].$$
This implies that any family of compact $N$-fold transverse configurations is locally trivial.
Thus, we can  identify a family of $N$-fold SC symplectic configurations over~$\bI$
with a smooth path $(\om_{t;i})_{i\in[N]}$ in $\Symp^+(\X)$ for some 
$N$-fold transverse configuration~$\X$.\\

\noindent
Given a family  $\{\pi_I\!:\fX_I\!\lra\!B\}_{I\in\cP^*(N)}$ of transverse configurations,
we define fiber bundles
$$\pi_{\eset}\!:\fX_{\eset}\lra B \qquad\hbox{and}\qquad
\pi_{\prt}\!:\fX_{\prt}\lra B$$
similarly to~\eref{Xesetdfn_e} and~\eref{Xprtdfn_e}.
For each $t\!\in\!B$, let 
$$\fX_{t;\eset}=\pi_{\eset}^{-1}(t) \qquad\hbox{and}\qquad
\fX_{t;\prt}=\pi_{\prt}^{-1}(t).$$
The precise definition of the total space of the complex line bundle
$\cO_{X_{\prt}}(X_{\eset})$ in~\eref{PsiDfn_e} depends
on the choices of identifications~$\Psi_{ij;j}$ of neighborhoods of~$X_{ij}$
in~$X_j$ and in~$X_j$ and of the $\om_j$-tame complex structures~$\fI_{ij;j}$
on (the fibers of) $\cN_{X_j}X_{ij}$ that satisfy~\eref{restrisom_e3}, respectively. 
For a family of SC symplectic configurations as in Definition~\ref{SCCfam_dfn}\ref{SCCfam_it},
such choices can be made continuously with respect to $t\!\in\!B$.
We then obtain a complex line~bundle 
\BE{BprtO_e}\cO_{\fX_{\prt}}(\fX_{\eset})\lra\fX_{\prt} \qquad\hbox{s.t.}\qquad
\cO_{\fX_{\prt}}(\fX_{\eset})\big|_{\fX_{t;\prt}}= \cO_{\fX_{t;\prt}}(\fX_{t;\eset})
\quad\forall~t\!\in\!B\,. \EE
We call a family $(\hb_t)_{t\in B}$ of homotopy classes of trivializations of 
$\cO_{\fX_{t;\prt}}(\fX_{t;\eset})$ \sf{continuous} if for each $t_0\!\in\!B$ 
there exist a neighborhood~$U$ of~$t_0$ in~$B$ and a trivialization~$\Phi$ of 
$\cO_{\fX_{\prt}}(\fX_{\eset})|_{\pi_{\prt}^{-1}(U)}$ such that 
$\Phi|_{\fX_{t;\prt}}\!\in\!\hb_t$ for every $t\!\in\!U$.\\

\noindent
Suppose $\X$ is an $N$-fold transverse configuration as in 
Definition~\ref{TransConf_dfn2},
$$\big(\om_{0;i}\big)_{i\in[N]},\big(\om_{1;i}\big)_{i\in[N]} \in \Symp^+(\X)$$
lie in the same topological component of $\Symp^+(\X)$, 
and the associated line bundle~\eref{PsiDfn_e} admits a trivialization.
Let $\hb_0$ and $\hb_1$ be homotopy classes of trivializations of 
 the line bundles~$\cO_{X_{\prt}}(X_{\eset})$ corresponding to 
the SC symplectic configurations determined by 
$(\om_{0;i})_{i\in[N]}$ and $(\om_{1;i})_{i\in[N]}$, respectively.
We define $\hb_0\!=\!\hb_1$ if there exist a smooth
path $(\om_{t;i})_{i\in[N]}$ in $\Symp^+(\X)$
and a continuous family $(\hb_t)_{t\in B}$ of homotopy classes of trivializations of 
the corresponding line bundles $\cO_{\fX_{t;\prt}}(\fX_{t;\eset})$.

\begin{dfn}\label{SympFibrfam_dfn}
Let $B$ be a manifold, possibly with boundary.
\begin{enumerate}[label=(\arabic*),leftmargin=*]

\item\label{SympFibrFam_it} 
A \sf{family of nearly regular symplectic fibrations over~$B$} 
is a tuple $(\sZ,\om_{\sZ},\pi,\pi_{\sZ})$, where\linebreak $(\sZ,\om_{\sZ},\pi_{\sZ})$ is a family
of symplectic manifolds over~$B$ and $\pi\!:\sZ\!\lra\!\C$ is a smooth map such~that 
\BE{SympFibrFam_e}\big(\sZ_t,\om_{\sZ;t},\pi_t\big)\equiv
\big(\pi_{\sZ}^{-1}(t),\om_{\sZ}|_{\sZ_t},\pi|_{\sZ_t}\big)\EE
is a nearly regular symplectic fibration for every $t\!\in\!B$.

\item\label{SympSmoothFam_it}  Let $(\fX_{t;\eset})_{t\in B}$ be the family  of 
SC symplectic varieties associated with a family of $N$-fold SC symplectic configurations over~$B$
as in Definition~\ref{SCCfam_dfn}\ref{SCCfam_it}.
A \sf{family of one-parameter families of smoothings} of $(\fX_{t;\eset})_{t\in B}$
is a family $(\sZ,\om_{\sZ},\pi,\pi_{\sZ})$
of nearly regular symplectic fibrations over~$B$ such~that
\eref{SympFibrFam_e} is a one-parameter family of smoothings of~$\fX_{t;\eset}$ 
for every $t\!\in\!B$.\\

\end{enumerate}
\end{dfn}

\noindent
Suppose $(\sZ,\om_{\sZ},\pi_{\sZ})$ is a family of symplectic manifolds over~$B$, 
$(\fX_{t;\eset})_{t\in B}$ is as in Definition~\ref{SympFibrfam_dfn}\ref{SympSmoothFam_it},
and $\fX_{t;\eset}\!\subset\!\sZ_t$ is an SC symplectic divisor for every $t\!\in\!B$.
We  call a family $(\fh_t)_{t\in B}$ of germs of homotopy classes of 
one-parameter families $(\sZ'_t,\om_{\sZ}|_{\sZ'_t},\pi_t)$ 
of smoothings of~$\fX_{t;\eset}$ with $\sZ_t'\!\subset\!\sZ_t$ 
\sf{continuous} if for each $t_0\!\in\!B$ 
there exist a neighborhood~$U$ of~$t_0$ in~$B$ and 
a family $(\sZ',\om_{\sZ}|_{\sZ'},\pi,\pi_{\sZ}|_{\sZ'})$ of one-parameter 
families of smoothings  of $(\fX_{t;\eset})_{t\in U}$ with $\sZ'\!\subset\!\sZ$
such~that  $(\sZ_t',\om_{\sZ}|_{\cZ'_t},\pi_t)\!\in\!\fh_t$
for every $t\!\in\!B$.

\section{Main setup and output}
\label{MainStat_sec}

\noindent
We review basic notions from Hamiltonian symplectic geometry,
specifying our conventions, in Section~\ref{TorusAct_subs}.
In Section~\ref{SympCutThm_subs1}, we study the notions 
of Hamiltonian and cutting configurations of Definitions~\ref{SympCut_dfn0} 
and~\ref{SympCut_dfn1}.
Simple local examples of such configurations  are described 
in Section~\ref{BasicEg_subs}.
Section~\ref{MomPolytThm_subs} provides a plethora of 
Hamiltonian and cutting configurations for
symplectic manifolds with Hamiltonian torus actions.
Theorems~\ref{SympCut_thm12} and~\ref{SympCut_thm3}, 
stated in Section~\ref{SympCutThm_subs2} and proved in Sections~\ref{Thm12Pf_sec} 
and~\ref{SumCutSm_sec}, 
describe the output determined by a cutting configuration via 
the multifold symplectic cut/degeneration construction of this paper.

\subsection{Torus actions and moment maps}
\label{TorusAct_subs}

\noindent
\sf{The characteristic vector field of} a (smooth) $S^1$-action 
$$\phi\!:S^1\!\times\!X\lra X$$
on a manifold~$X$ is the vector field~$\ze_{\phi}$ on~$X$ given~by 
$$\ze_{\phi}(x)=\frac{\nd}{\nd\th}\phi(\ne^{\fI\th};x\big)\bigg|_{\th=0}\in T_xX
\qquad\forall~x\!\in\!X\,.$$
If $(X,\om)$ is a symplectic manifold,
a \sf{Hamiltonian for}  an $S^1$-action~$\phi$ on~$X$ as above
is a smooth map \hbox{$h\!:X\!\lra\!\R$} such~that 
$$-\nd h= \io_{\ze_{\phi}}\om\equiv \om(\ze_{\phi},\cdot);$$
such a function~$h$ is $S^1$-invariant.
For example, a Hamiltonian for the $S^1$-action
$$\phi\!:S^1\!\times\!\C^N\lra \C^N, \qquad
\phi(\ne^{\fI\th};z_1,\ldots,z_N\big)
=\big(z_1,\ldots,z_{i-1},\ne^{\fI\th}z_i,z_{i+1},\ldots,z_N\big),$$
on~$\C^N$ with the standard symplectic form 
$$\om_{\C^N}\equiv\nd x_1\!\w\!\nd y_1+\ldots+\nd x_N\!\w\!\nd y_N$$
is given~by 
$$h\!:\C^N\lra\R, \qquad h(z_1,\ldots,z_N)\!=\!\frac12|z_i|^2\,.$$
An $S^1$-action~$\phi$ on $(X,\om)$ is called \sf{Hamiltonian} if 
a Hamiltonian~$h$ for~$\phi$ exists.
In such a case, $h$ is well-defined up to a constant (on each connected component of~$X$).\\

\noindent
For a $k$-torus $\T\!\approx\!(S^1)^k$, we denote by $\ft$ its Lie algebra 
and by~$\ft^*$ the dual of~$\ft$.
A smooth~map
\BE{phiTdfn_e}\phi\!:\T\!\times\!X\lra X\EE
and an element $\xi\!\in\!\ft$ determine a vector field~$\xi_{\phi}$ on~$X$ by
$$\xi_{\phi}(x)=\nd_{(\id,x)}\phi(\xi,0)\in T_xX \qquad\forall~x\!\in\!X.$$
If $(X,\om)$ is a symplectic manifold,
a \sf{moment map} for a $\T$-action~$\phi$ on~$X$ as in~\eref{phiTdfn_e}  
is a $\T$-invariant smooth~map 
$$\mu\!:X\lra\ft^* \qquad\hbox{s.t.}\quad
-\nd\blr{\mu(\cdot),\xi}=\io_{\xi_{\phi}}\om~~\forall\,\xi\!\in\!\ft\,.$$
An $\T$-action~$\phi$ on $(X,\om)$ is called \sf{Hamiltonian} if 
a moment map~$\mu$ for~$\phi$ exists.
In such a case, $\mu$ is well-defined up to a constant (on each connected component of~$X$).
A \sf{Hamiltonian $\T$-pair for a symplectic manifold}~$(X,\om)$ is a pair~$(\phi,\mu)$ 
consisting of a Hamiltonian action~$\phi$ of a torus~$\T$
on~$(X,\om)$ and a~moment map~$\mu$ for
this action.\\

\noindent
We identify the Lie algebra~$\ft_1$ of~$S^1$ and its dual~$\ft_1^*$ with~$\R$ by 
the dual homomorphisms
\BE{Rt1iden_e}\R\lra\ft_1, \quad r\lra 
 \xi_r\!\equiv\!\frac{\nd}{\nd\th}\big(\ne^{\fI r\th}\big)\bigg|_{\th=0}\in T_{\id}S^1,
\qquad \ft_1^*\lra\R, \quad \eta\lra\blr{\eta,\xi_1}\in\R.\EE 
In particular, $\ze_{\phi}\!=\!(\xi_1)_{\phi}$.
The composition of the second homomorphism above with a moment map~$\mu$  
for an~$S^1$-action~$\phi$ on~$(X,\om)$ is a Hamiltonian for this action.
An isomorphism $\T\!\approx\!(S^1)^k$ and the identifications~\eref{Rt1iden_e}
determine identifications
$$ \ft\approx\R^k \qquad\hbox{and}\qquad \ft^*\approx\R^k \,.$$
A $\T$-action~$\phi$ on~$X$ then corresponds to a tuple $(\phi_i)_{i\in[k]}$
of commuting $S^1$-actions on~$X$.
A~moment map~$\mu$ for a $\T$-action~$\phi$ on~$(X,\om)$ corresponds
to a tuple $(h_i)_{i\in[k]}$ of Hamiltonians for the \hbox{$S^1$-actions~$\phi_i$}
 preserved by all \hbox{$S^1$-actions}, i.e.
$$h_i\big(\phi_j(\ne^{\fI\th};x)\big)=h_i(x)
\qquad\forall~(\ne^{\fI\th};x)\in S^1\!\times\!X,~i,j\!\in\![k].$$

\subsection{Hamiltonian and cutting configurations}
\label{SympCutThm_subs1}

\noindent
A Hamiltonian configuration~\eref{SympCutDfn_e} can alternatively be described
in terms of compatible $S^1$-actions as follows.
For $i,j\!\in\!\Z$, define
$$
\fs_{ij}=\begin{cases}1,&\hbox{if}~i\!>\!j;\\
0,&\hbox{if}~i\!=\!j;\\
-1,&\hbox{if}~i\!<\!j.\end{cases}$$
For $i\!\in\!I\!\subset\![N]$, the homomorphism
\BE{vrIidfn_e}
\vr_{I;i}\!:(S^1)^{I-i}\lra(S^1)^I, \qquad
\big(\vr_{I;i}\big((\ne^{\fI\th_j})_{j\in I-i}\big)\!\big)_k=
\begin{cases}
\prod\limits_{j\in I-i}\!\!\!\!\ne^{-\fI\fs_{ij}\th_j},&\hbox{if}~k\!=\!i;\\
\ne^{\fI\fs_{ik}\th_k},&\hbox{if}~k\!\in\!I\!-\!i;
\end{cases}\EE
is an isomorphism onto $(S^1)^I_{\bu}$.
The induced homomorphism on the duals of the Lie algebras,
\BE{vrIidfn_e2} 
\big\{\nd_{\id}\vr_{I;i}\big\}^*\!\!: \R^I\lra \R^{I-i}, \qquad 
(a_j)_{j\in I}\lra \big(\fs_{ij}(a_j\!-\!a_i)\big)_{j\in I},\EE
descends to an isomorphism from $\ft_{I;\bu}^*$ to~$\R^{I-i}$.
Via the isomorphisms~\eref{vrIidfn_e} and~\eref{vrIidfn_e2},
the tuple $(\phi_I,\mu_I)_{I\in\cP^*(N)}$ in~\eref{SympCutDfn_e}
corresponds to smooth~maps
\BE{phiijhij_e}
\phi_{ij}\!=\!\phi_{ji}\!:S^1\!\times\!
\bigcup_{I\in\cP_{ij}(N)}\!\!\!\!\!\!\!U_I\lra \bigcup_{I\in\cP_{ij}(N)}\!\!\!\!\!\!U_I
\quad\hbox{and}\quad
h_{ij}\!=\!h_{ji}\!: \bigcup_{I\in\cP_{ij}(N)}\!\!\!\!\!\!U_I\lra\R\EE
with $i,j\!\in\![N]$ distinct so that $\phi_{ij}$ are commuting $S^1$-actions
with Hamiltonians~$h_{ij}$ preserved by all these actions and satisfying
$$\fs_{ij}h_{ij}|_{U_I}+\fs_{jk}h_{jk}|_{U_I}=\fs_{ik}h_{ik}|_{U_I}$$
for all $i,j,k\!\in\!I\!\subset\![N]$.\\

\noindent
Since $(S^1)^I_{\bu}$ is a torus of dimension $|I|\!-\!1$,
the ``actions" $\phi_{\{i\}}$ with $i\!\in\![N]$ in~\eref{SympCutDfn_e}
are trivial. 
In the $N\!=\!1$ case of Definition~\ref{SympCut_dfn1}, 
the open cover consists of the single set $U_1\!=\!X$.
There are no torus actions then;
the output of Sections~\ref{Thm12Pf_sec} and~\ref{SumCutSm_sec} is then simply
$$(X_1,\om_1)=(X,\om), \quad 
(\cZ,\om_{\cZ})=\big(X\!\times\!\C,\pi_1^*\om\!+\!\pi_2^*\om_{\C}\big), \quad
\pi\!=\!\pi_2\!:\cZ\!\lra\C\,.$$
The $N\!=\!2$ case corresponds to the setting in the symplectic cut
construction of~\cite{L} with a separating hypersurface 
$$\wt{V}= U_{12;12}\equiv \mu_{12}^{-1}(0)\subset U_{12}\,.$$
In this case, $U_{12}$ is an open neighborhood of $\wt{V}$ with 
an action of~$(S^1)^2_{\bu}\!\approx\!S^1$.
The open subsets $U_1,U_2\!\subset\!X$ can be taken to be the unions of 
the topological components of $X\!-\!\wt{V}$ so~that 
$$(\mu_{12}(x))_1\!<\!(\mu_{12}(x))_2~~\forall\,x\!\in\!U_1\!\cap\!U_{12} \qquad\hbox{and}\qquad
(\mu_{12}(x))_2\!<\!(\mu_{12}(x))_1~~\forall\,x\!\in\!U_2\!\cap\!U_{12};$$
see Definition~\ref{SympCut_dfn0}\ref{UIJpos_it}.
This implies that~$\wt{V}$ is closed in~$X$ (and not just in~$U_{12}$).\\

\noindent
Let $\sC$ be an $N$-fold Hamiltonian configuration for~$(X,\om)$ as in~\eref{SympCutDfn_e}.
We call a subset $Y\!\subset\!X$ \sf{$\sC$-invariant}~if 
$$\phi_I\big((S^1)^I_{\bu}\!\times\!(Y\!\cap\!U_I)\big)=Y\!\cap\!U_I
\qquad\forall~I\!\in\!\cP^*(N).$$
For such a subset, let 
$$\sC|_Y = \big(Y\!\cap\!U_I,\phi_I|_{(S^1)^I_{\bu}\times(Y\cap U_I)},
\mu_I|_{Y\cap U_I}\big)_{I\in\cP^*(N)}.$$
If in addition $Y$ is a symplectic submanifold of~$(X,\om)$,
then $\sC|_Y$ is an $N$-fold Hamiltonian configuration for~$(Y,\om|_Y)$;
we call it the \sf{restriction of~$\sC$ to~$Y$}.
If  $\sC$ is a cutting configuration, then so is~$\sC|_Y$.\\

\noindent
If $(U_I)_{I\in\cP^*(N)}$ is a collection of subsets of~$X$,
another collection $(U_I')_{I\in\cP^*(N)}$ of subsets of~$X$ \sf{refines} 
$(U_I)_{I\in\cP^*(N)}$ if $U_I'\!\subset\!U_I$  for all $I\!\in\!\cP^*(N)$.
We  call such a refinement \sf{proper} if $\ov{U_I'}\!\subset\!U_I$
for all $I\!\in\!\cP^*(N)$.
If $\sC$ is an  $N$-fold Hamiltonian configuration for~$(X,\om)$ as in~\eref{SympCutDfn_e},
we  call  a collection $(U_I')_{I\in\cP^*(N)}$ of subsets of~$X$ refining~$(U_I)_{I\in\cP^*(N)}$
\sf{$\sC$-invariant}~if 
$$\phi_I\big((S^1)^I_{\bu}\!\times\!U_I'\big)=U_I' \qquad\forall~I\!\in\cP^*(N).$$
If in addition each $U_I'$ is open and the union of these subsets is $X'\!\subset\!X$, then
\BE{sCprdfn_e}\sC'\equiv
\big(U_I',\phi_I|_{(S^1)^I_{\bu}\times U_I'},\mu_I|_{U_I'}\big)_{I\in\cP^*(N)}\EE
is  an  $N$-fold Hamiltonian configuration for~$(X',\om|_{X'})$;
we  call it the \sf{restriction of~$\sC$ to $(U_I')_{I\in\cP^*(N)}$}.
If $\sC$ is a cutting configuration, then so is~$\sC'$.
Restricting to an open refinement covering~$X$ has no effect 
on the cut symplectic manifolds $(X_i,\om_i)$,
their submanifolds~$(X_I,\om_I)$, the symplectic manifold~$(\cZ,\om_{\cZ})$,
or the deformation equivalence class of the fibration $\pi\!:\cZ'\!\lra\!\C$
constructed in Sections~\ref{Thm12Pf_sec} and~\ref{SumCutSm_sec}, but
refines the open cover $\{X_I^{\circ}\}_{I\in\cP_i(N)}$ of each~$X_i$
and the open cover $\{\cZ_I^{\circ}\}_{I\in\cP^*(N)}$ of~$\cZ$ determined by~$\sC$.\\

\noindent
According to Definition~\ref{SympCut_dfn0}\ref{UIJpos_it}, 
$$(\mu_J(x))_i\!<\!(\mu_J(x))_j \qquad\forall~x\!\in\!U_I\!\cap\!U_J,~
i\!\in\!I\!\subset\!J\!\subset\![N],~j\!\in\!J\!-\!I.$$
We call an $N$-fold Hamiltonian configuration as in~\eref{SympCutDfn_e} 
\sf{maximal}~if
\BE{IJpos_e}\big\{x\!\in\!U_J\!:\,(\mu_J(x))_i\!<\!(\mu_J(x))_j
~\forall\,i\!\in\!I,~j\!\in\!J\!-\!I\big\}
\subset U_I\EE
whenever $I,J\!\in\!\cP^*(N)$ and $I\!\subset\!J$.
The condition~\eref{IJpos_e} is automatically satisfied if $I\!=\!J$.
The \hbox{$N\!=\!1,2$} configurations described above
and the configuration~\eref{THamConf_e} with a Hamiltonian $(S^1)^N_{\bu}$-action
on the entire symplectic manifold $(X,\om)$ are maximal.
Special cases of the latter include 
the basic local $N$-fold configuration~\eref{BasicEgConf_e}
and the configuration~\eref{toriccut_e} for a symplectic manifold with 
a Hamiltonian action of an abstract torus.
For a maximal Hamiltonian configuration, the set of conditions of Definition~\ref{SympCut_dfn1}
indexed by the pairs $I_0\!\subset\!I$ reduces to its subset with $I_0\!=\!I$.
In light of the second statement in Lemma~\ref{sCcover_lmm},
maximal cutting configurations give rise to open covers $\{X_I^{\circ}\}_{I\in\cP_i(N)}$ of 
the cut symplectic manifolds~$X_i$ in the spirit of toric and tropical geometries 
(i.e.~similar to Zariski open sets in algebraic geometry).
For such configurations, each $X_I^{\circ}$ contains the submanifold~$X_I$ outside of 
the submanifolds~$X_J$ with $J\!\supsetneq\!I$ (the real codimension of~$X_J$ in~$X_I$ is 
$2(|J|\!-\!|I|)$).\\

\noindent
A~maximal $N$-fold Hamiltonian configuration can be obtained from
any given $N$-fold Hamiltonian configuration~\eref{SympCutDfn_e} 
by taking
\begin{gather*}
U_I'=U_I\cup \bigcup_{J\in\cP_I(N)}\!\!\!\!\!\! 
\big\{x\!\in\!U_J\!:\,(\mu_J(x))_i\!<\!(\mu_J(x))_j
~\forall\,i\!\in\!I,~j\!\in\!J\!-\!I\big\},\\
\phi_I'(g;x)=\phi_J(g;x), \quad \mu_I'(x)=\mu_J(x)\big|_{\ft_{I;\bu}}
\qquad\forall~g\!\in\!(S^1)^I_{\bu},\,x\!\in\!U_I'\!\cap\!U_J,\,
J\!\in\!\cP_I(N).
\end{gather*}
By~\ref{IJinter_it} and~\ref{phiIUJ_it} in Definition~\ref{SympCut_dfn0},
$\phi_I'(g;x)$ and $\mu_I'(x)$ are independent of the possible choices of~$J$.
Since the moment map~$\mu_J$ is $(S^1)^I_{\bu}$-invariant whenever $J\!\in\cP_I(N)$,
the image of~$\phi_I'$ is contained in~$U_I'$.
Thus, $\phi_I'$ is an $(S^1)^I_{\bu}$-action on~$U_I'$ with moment map~$\mu_I'$.
It is immediate that the new collection
$$\sC'\equiv \big(U_I',\phi_I',\mu_I'\big)_{I\in\cP^*(N)}$$
satisfies Definition~\ref{SympCut_dfn0}\ref{phiIUJ_it};
verifying \ref{IJinter_it} and \ref{UIJpos_it} in Definition~\ref{SympCut_dfn0} and~\eref{IJpos_e} 
is a bit tedious, but straightforward.
If $\sC$ is a cutting configuration, then so is~$\sC'$.
Thus, every Hamiltonian (resp.~cutting) configuration is a restriction of 
a maximal Hamiltonian (resp.~cutting) configuration.
We also note that if all torus actions~$\phi_I$ are free, 
then so are the torus actions~$\phi_I'$.\\

\noindent
The open sets~$U_I$ in a maximal cutting configuration~$\sC$ can be inductively shrunk
so that the restricted actions in~\eref{sCprdfn_e} are free.
Suppose $I^*\!\in\!\cP^*(N)$, 
the actions~$\phi_I$ are free for all $I\!\supsetneq\!I^*$,
and~\eref{IJpos_e} holds for all $I\!\subsetneq\!I^*$.
Since the $\phi_{I^*}$-action on \hbox{$\mu_{I^*}^{-1}(0)\!=\!\mu_{I^*;I^*}^{-1}(0)$} is free,
it is also free on a $\phi_{I^*}$-invariant neighborhood~$U_{I^*}'$ 
of~$\mu_{I^*}^{-1}(0)$ in~$U_{I^*}$.
By the inductive assumption regarding~\eref{IJpos_e},
$$U_{I^*}-U_{I^*}'\subset \bigcup_{\eset\neq I\subsetneq I^*}\!\!\!\!U_I\,.$$
Thus, replacing $U_{I^*}$ with~$U_{I^*}'$ accomplishes the inductive step.
Proceeding in this way, we obtain a $\sC$-invariant open cover $(U_I')_{I\in\cP^*(N)}$
of~$X$ refining $(U_I)_{I\in\cP^*(N)}$ so that the actions~$\phi_I$ are free on~$U_I'$.\\

\noindent
In summary, any $N$-fold cutting configuration~$\sC$ for~$(X,\om)$ can replaced
by a maximal $N$-fold cutting configuration~$\sC'$ with free torus 
actions without any effect on the output of the symplectic cut construction:
the symplectic manifolds $(X_i,\om_i)$ with $i\!\in\![N]$,
their submanifolds~$(X_I,\om_I)$ with $I\!\in\!\cP_i(N)$,
the symplectic manifold~$(\cZ,\om_{\cZ})$,
or the deformation equivalence class of the fibration $\pi\!:\cZ'\!\lra\!\C$.

\subsection{Induced degenerations}
\label{SympCutThm_subs2}

\noindent
Suppose $\pi_{\fX}\!:\fX\!\lra\!B$ is a family of manifolds as in 
Definition~\ref{Manfam_dfn}\ref{Manfam_it}, 
$U_t\!\subset\!\fX_t$ is an open subset for each $t\!\in\!B$,
and $Y$ is a smooth manifold.
We  call a family \hbox{$(\phi_t\!:U_t\!\lra\!Y)_{t\in B}$} 
\sf{smooth}~if 
$$\bigcup_{t\in B}U_t\lra Y, \qquad (t,x)\lra \phi_t(x),$$
is a smooth map from an open subspace of~$\fX$.

\begin{dfn}\label{SympCutFam_dfn}
Let $N\!\in\!\Z^+$, $B$ be a manifold, possibly with boundary,
and $(\fX,\om_{\fX},\pi_{\fX})$ be a family of symplectic manifolds over~$B$
as in Definition~\ref{Manfam_dfn}\ref{ManSymp_it}.
A  \sf{family of $N$-fold Hamiltonian} (resp.~\sf{cutting}) 
\sf{configurations for~$(\fX,\om_{\fX},\pi_{\fX})$} is a~tuple 
\BE{SympCutFam_e}
(\sC_t)_{t\in B}
\equiv\big(U_{t;I},\phi_{t;I},\mu_{t;I}\big)_{I\in\cP^*(N),t\in B}\EE
such that $\sC_t$ is an $N$-fold Hamiltonian (resp.~cutting) configuration 
for~$(\fX_t,\om_t)$ for each $t\!\in\!B$ and the families of~maps
$$ \bigg(\phi_{t;I}\!: (S^1)^I_{\bu}\!\times\!U_{t;I}\lra \fX\bigg)_{\!t\in B}
\quad\hbox{and}\quad
 \bigg(\mu_{t;I}\!:U_{t;I}\lra\ft_{I;\bu}^*\bigg)_{\!t\in B}$$
are smooth for all $I\!\in\!\cP^*(N)$.
\end{dfn}

\noindent
As described in Section~\ref{SympFamil_subs}, a smooth path $(\om_t)_{t\in B}$
of symplectic forms on a manifold~$X$ determines a family $(\fX,\om_{\fX},\pi_{\fX})$ of 
symplectic manifolds over~$B$.
We call a  family of $N$-fold Hamiltonian (resp.~cutting) configurations for such 
a tuple $(\fX,\om_{\fX},\pi_{\fX})$
an \sf{$(\om_t)_{t\in B}$-family of $N$-fold  Hamiltonian (resp.~cutting) configurations for~$X$}.
We call $N$-fold Hamiltonian (resp.~cutting) configurations $\sC_0$ for $(X,\om_0)$ and
$\sC_1$ for $(X,\om_1)$ \sf{deformation equivalent} if 
there are a path $(\om_t)_{t\in\bI}$ of symplectic forms on~$X$
and an $(\om_t)_{t\in\bI}$-family $(\sC_t)_{t\in\bI}$ of $N$-fold 
Hamiltonian (resp.~cutting) configurations for~$X$.
For example, the restriction~$\sC'$ of a Hamiltonian or 
cutting configuration~$\sC$ as in~\eref{SympCutDfn_e} 
to a $\sC$-invariant open cover $(U_I')_{I\in\cP^*(N)}$ of~$X$ refining $(U_I)_{I\in\cP^*(N)}$
is deformation equivalent to~$\sC$.
In this case, we can define $U_{t;I}\!\subset\!\{t\}\!\times\!X$~by
$$\bigcup_{t\in\bI}U_{t;I} \equiv 
\big[0,1/2)\!\times\!U_I \cup [0,1]\!\times\!U_I' \subset \bI\!\times\!U_I$$
and take $\phi_{t;I}$ and $\mu_{t;I}$ to be the restrictions of~$\phi_I$
and~$\mu_I$, respectively.\\

\noindent
Let $\sC$ be an  $N$-fold cutting configuration for~$(X,\om)$ as in~\eref{SympCutDfn_e}.
For $i\!\in\!I\!\subset\![N]$, we~define
\BE{UiIdfn_e} U_{i;I}^{\le} = 
\big\{x\!\in\!U_I\!:\,(\mu_I(x))_i\!\le\!(\mu_I(x))_j~\forall\,j\!\in\!I\big\}; \EE
in particular, $U_{i;i}^{\le}\!=\!U_i$.
Since the $\phi_I$-action is Hamiltonian, the subset
$U_{i;I}^{\le}\!\subset\!U_I$ is $\phi_I$-invariant.
Let  
\begin{gather}
\label{ovUi_e}
U_i^{\le}= \bigcup_{I\in\cP_i(N)}\!\!\!\!\!\!U_{i;I}^{\le}\,,\\
\notag
\prt U_i^{\le}=\big\{x\!\in\!U_i^{\le}\!\!:\,
(\mu_I(x))_i\!=\!(\mu_I(x))_j~\hbox{for some}~I\!\in\!\cP_i(N),\,j\!\in\!I\!-\!i
~\hbox{s.t.}~x\!\in\!U_{i;I}^{\le}\big\}\,.
\end{gather}
By Lemmas~\ref{Xitopol_lmm1} and~\ref{Xitopol_lmm2}, 
the subsets $U_i^{\le},\prt U_i^{\le}\!\subset\!X$ are closed.
Since the sets~$U_{i;I}^{\le}$ with $i\!\in\!I$ cover~$U_I$
and the sets~$U_I$ with $I\!\in\!\cP^*(N)$ cover~$X$,
the collection $(U_i^{\le})_{i\in[N]}$ covers~$X$.\\

\noindent
In the $N\!=\!2$ case of \cite[Section~1.1]{L},
a cutting configuration for a symplectic manifold~$(X,\om)$ produces 
two symplectic manifolds, $(X_1,\om_1)$ and $(X_2,\om_2)$, with 
a common symplectic divisor~$X_{12}$.
They are obtained by cutting~$X$ into the closed subsets~$U_1^{\le}$ 
and~$U_2^{\le}$
along the zero set of the only (non-trivial) moment map~$\mu_{12}$
 and collapsing their boundary 
$$\prt U_1^{\le}=\prt U_2^{\le}=U_{12;12}$$ 
by the single $S^1$-action~$\phi_{12}$.
We show in Section~\ref{Thm12Pf_sec} that this construction extends
to an arbitrary $N$-fold cutting configuration in the sense of Definition~\ref{SympCut_dfn1}
and produces $N$~symplectic manifolds~$(X_i,\om_i)$ with common symplectic divisors~$X_{ij}$
which together form an SC symplectic divisor inside of another symplectic manifold~$(\cZ,\om_{\cZ})$.
The symplectic manifolds~$(X_i,\om_i)$ are  obtained  
by cutting~$X$ into the closed subspaces~$U_i^{\le}$
along the zero sets of the moment maps~$\mu_I$
and collapsing their boundaries and corners~$\prt U_i^{\le}$ 
by the $(S^1)^I_{\bu}$-actions~$\phi_I$.

\begin{numthm}\label{SympCut_thm12}
Suppose $N\!\in\!\Z^+$, $B$ is a  manifold, possibly with boundary, and
$(\fX,\om_{\fX},\pi_{\fX})$ is a family of symplectic manifolds over~$B$
as in Definition~\ref{Manfam_dfn}\ref{ManSymp_it}.
A family $(\sC_t)_{t\in B}$ of (fiberwise) $N$-fold cutting configurations 
for $(\fX,\om_{\fX},\pi_{\fX})$ as in~\eref{SympCutFam_e} determines
\begin{enumerate}[label=(\arabic*),leftmargin=*]

\item a family $(\sZ,\om_{\sZ},\pi_{\sZ})$ of symplectic manifolds over~$B$,

\item a family  $(\pi_I\!:\fX_I\!\lra\!B,\om_I)_{I\in\cP^*(N)}$
of $N$-fold SC symplectic configurations over~$B$ 
as in Definition~\ref{SCCfam_dfn} such that $\fX_{\eset}\!\subset\!\sZ$ and 
$\fX_{t;\eset}\!\subset\!\sZ_t$ is an SC symplectic divisor
for every $t\!\in\!B$,

\item a continuous map $q_{\eset}\!:\fX\!\lra\!\fX_{\eset}$ such that 
$q_{\eset}(\prt U_{t;i}^{\le})\!=\!\fX_{t;i}\!\cap\!\fX_{t;\prt}$ and 
$$q_{\eset}\!: 
\big( U_{t;i}^{\le}\!-\!\prt U_{t;i}^{\le},\om_t|_{U_{t;i}^{\le}-\prt U_{t;i}^{\le}}\big)\lra 
\big(\fX_{t;i}\!-\!\fX_{t;\prt},\om_{\sZ}|_{\fX_{t;i}-\fX_{t;\prt}}\big)$$
is a symplectomorphism for every $t\!\in\!B$ and $i\!\in\![N]$,

\item a family
$$(\sC_{\sZ;t})_{t\in B}\equiv 
\big(U_{\cZ;t;I},\phi_{\cZ;t;I},\mu_{\cZ;t;I}\big)_{I\in\cP^*(N),t\in B}$$ 
of (fiberwise) $N$-fold Hamiltonian configurations 
for $(\sZ,\om_{\sZ},\pi_{\sZ})$
which restricts to a family of cutting configurations over $\sZ\!-\!\fX_{\prt}$.
\end{enumerate}
If $\fX_t$ is compact, then so is~$\fX_{t;I}$ for every $I\!\in\!\cP^*(N)$.
\end{numthm}

\begin{numthm}\label{SympCut_thm3}
Let $N$, $B$, $(\fX,\om_{\fX},\pi_{\fX})$, $(\sC_t)_{t\in B}$, $(\fX_{t;\eset})_{t\in B}$,
$(\sZ,\om_{\sZ},\pi_{\sZ})$, and $(\sC_{\sZ;t})_{t\in B}$
 be as in Theorem~\ref{SympCut_thm12}.
The  family  $(\sC_t)_{t\in B}$  determines 
a continuous family $(\fh_t)_{t\in B}$ of (fiberwise) germs of deformation equivalence classes of 
 one-parameter families
$(\sZ'_t,\om_{\sZ}|_{\sZ'_t},\pi_t)$ 
of smoothings of $\fX_{t;\eset}$ with $\sZ_t'\!\subset\!\sZ_t$.
For every such germ~$\fh_t$, there are a representative $\pi_t\!:\sZ_t'\!\lra\!\C$ 
and a $\sC_{\sZ;t}$-invariant open cover $(U_{\sZ;t;I}')_{I\in\cP^*(N)}$ of~$\sZ_t'$
refining $(U_{\sZ;t;I})_{I\in\cP^*(N)}$ such that 
for every $\la\!\in\!\C$
the restriction~$\sC_{\sZ;t}'$ of~$\sC_{\sZ;t}$ to $(U_{\sZ;t;I}')_{I\in\cP^*(N)}$
restricts to  a Hamiltonian configuration~$\sC_{\sZ;t}'|_{\sZ_{t;\la}}$ for 
$$\big(\sZ_{t;\la},\om_{\sZ;t;\la}\big)\equiv 
\big(\pi_t^{-1}(\la),\om_{\sZ}|_{\pi_t^{-1}(\la)}\big);$$ 
this restriction is a cutting configuration if $\la\!\neq\!0$.
If  $\fX_t$ is compact, such a germ can be chosen so that 
for every \hbox{$\la\!\in\!\R^+$} sufficiently small
$\sC_{\sZ;t}'|_{\sZ_{t;\la}}$
is canonically isomorphic to the restriction of the original cutting configuration~$\sC_t$ 
for $(\fX_t,\om_t)$
to an open cover $(U_{t;I}')_{I\in\cP^*(N)}$ of~$\fX_t$
refining $(U_{t;I})_{I\in\cP^*(N)}$.
\end{numthm}

\noindent
By the first statement in Theorem~\ref{SympCut_thm3} and 
the family analogue of \cite[Proposition~5.1]{SympSumMulti}, 
a family  $(\sC_t)_{t\in B}$ of $N$-fold cutting configurations determines 
a homotopy class of trivializations of the complex line bundle~\eref{BprtO_e}
and thus
a continuous family $(\hb_t)_{t\in B}$ of homotopy classes of trivializations of 
the normal bundles~$\cO_{\fX_{t;\prt}}(\fX_{t;\eset})$ of the singular locus~$\fX_{t;\prt}$ 
of~$\fX_{t;\eset}$.
Given a single cutting configuration~$\sC$ (i.e.~$B$ is a point 
in Theorems~\ref{SympCut_thm12} and~\ref{SympCut_thm3}), 
we denote by~$\X(\sC)$ the associated $N$-fold SC symplectic configuration
and by~$\hb(\sC)$ the corresponding homotopy class of trivializations
of the line bundle $\cO_{X_{\prt}}(X_{\eset})$.

\begin{crl}\label{SympCut_crl}
Let $N\!\in\!\Z^+$, $\sC_0$ be an $N$-fold cutting configuration  for $(X,\om_0)$,
and $\sC_1$ be an $N$-fold cutting configuration for $(X,\om_1)$.
If~$\sC_0$ and~$\sC_1$ are deformation equivalent, 
then the $N$-fold SC symplectic configurations~$\X(\sC_0)$ and~$\X(\sC_1)$ lie
in the same path component
of $\Symp^+(\X)$ for some $N$-fold transverse configuration~$\X$
and $\hb(\sC_0)\!=\!\hb(\sC_1)$.
\end{crl}

\noindent
This corollary follows immediately from Theorems~\ref{SympCut_thm12} and~\ref{SympCut_thm3},
the family analogue of \cite[Proposition~5.1]{SympSumMulti}, and 
the triviality of families of SC symplectic configurations over~$\bI$
(see Section~\ref{SympFamil_subs}).
Corollary~\ref{SympCut_crl} implies that a deformation equivalence class of cutting configurations
determines a deformation equivalence class of SC symplectic varieties~$X_{\eset}$
with a homotopy class of trivializations of the normal bundle
$\cO_{X_{\prt}}(X_{\eset})$ in~\eref{PsiDfn_e}
of the singular locus $X_{\prt}\!\subset\!X_{\eset}$.

\section{Proof of Theorem~\ref{SympCut_thm12}}
\label{Thm12Pf_sec}

\noindent
Let $\sC$ be an $N$-fold cutting configuration on~$(X,\om)$ as in Definition~\ref{SympCut_dfn1}.
Section~\ref{TopolPrelim_subs} provides basic topological characterizations
of natural spaces determined by~$\sC$ that appear throughout the rest of this paper.
In Section~\ref{SympRed_subs}, we use the symplectic reduction technique of
\cite{Meyer,MaW}
to construct symplectic manifolds~$(\cZ_I^{\circ},\vp_I^{\circ})$ out of
the open subsets~$U_I$ of~$X$ so that
\BE{dimcZX_e} \dim_{\R}\cZ_I^{\circ}=\dim_{\R}X+2\,.\EE
We glue these global quotients in Section~\ref{NatCutSmoothSp_subs}
into a single symplectic manifold~$(\cZ,\om_{\cZ})$. 
As shown in Section~\ref{SCsympdiv_subs}, $(\cZ,\om_{\cZ})$ contains 
closed symplectic submanifolds $(X_I\!=\!\cZ_I,\om_I)$ with $I\!\in\!\cP^*(N)$ so~that 
$$\X(\sC)\equiv \big((X_I)_{I\in\cP^*(N)},(\om_i)_{i\in[N]}\big)$$
is an $N$-fold SC symplectic configuration in the sense of Definition~\ref{SCC_dfn}
and the corresponding SC symplectic variety~$X_{\eset}$ is embedded into~$(\cZ,\om_{\cZ})$
as an SC symplectic divisor.
A natural $N$-fold Hamiltonian configuration~$\sC_{\cZ}$ for~$(\cZ,\om_{\cZ})$,
which restricts to a cutting configuration  on the complement of the singular locus~$X_{\prt}$
of~$X_{\eset}$,
is constructed in Section~\ref{cZSympCutConf_subs}.\\

\noindent
As the constructions of Sections~\ref{SympRed_subs}-\ref{cZSympCutConf_subs}
involve no choices, they produce 
a family 
$$\big((\pi_I\!:\fX_I\!\lra\!B)_{I\in\cP^*(N)},(\om_i)_{i\in[N]}\big)$$
of $N$-fold SC symplectic configurations over~$B$  and
a family~$(\sC_{\sZ;t})_{t\in B}$ of $N$-fold Hamiltonian configurations 
for a family $(\sZ,\om_{\sZ},\pi_{\sZ})$ of symplectic manifolds over~$B$
when applied  to a family of $N$-fold cutting configurations. 
This establishes Theorem~\ref{SympCut_thm12}.

\subsection{Topological preliminaries}
\label{TopolPrelim_subs}

\noindent
The topological observations of Lemmas~\ref{sCcover_lmm}-\ref{cZtopol_lmm2} below
concern subspaces of~$X$, including 
$U_{i;I}^{\le}$ and~$U_i^{\le}$
defined in~\eref{UiIdfn_e} and~\eref{ovUi_e}, respectively,
as well as subspaces of~$X\!\times\!\C$.
For $I_0\!\in\!\cP^*(N)$ and $I\!\in\!\cP_{I_0}(N)$, let 
\begin{equation*}\begin{split}
U_{I_0;I}^< =\mu_{I_0;I}^{-1}(0)
=\big\{x\!\in\!U_I\!:\,&(\mu_I(x))_i\!=\!(\mu_I(x))_j~\forall\,i,j\!\in\!I_0,\\
~&(\mu_I(x))_i\!<\!(\mu_I(x))_j~\forall\,i\!\in\!I_0,\,j\!\in\!I\!-\!I_0\big\}.
\end{split}\end{equation*}
For $i\!\in\!I\!\subset\![N]$, let
\BE{cUiIdfn_e} \cU_{i;I}^{\le} = 
\big\{(x,z)\!\in\!U_I\!\times\!\C\!:\,
(\mu_I(x))_i\!-\!\frac12|z|^2\!\le\!(\mu_I(x))_j~\forall\,j\!\in\!I\big\}. \EE
While the components $(\mu_I(x))_i\!\in\!\R$ of $\mu_I(x)\!\in\!\ft_{I;\bu}^*$
are not well-defined, the numbers
\BE{hijdfn_e}
\big(\mu_I(x)\big)_{\!ij}\equiv\big(\mu_I(x)\!\big)_{\!j}\!-\!\big(\mu_I(x)\!\big)_{\!i}\in\R\EE
are well-defined; thus, so is the condition in~\eref{cUiIdfn_e}.
For $i\!\in\![N]$, let 
$$\cU_i^{\le}= \bigcup_{I\in\cP_i(N)}\!\!\!\!\!\cU_{i;I}^{\le}
\subset X\!\times\!\C\,.$$
The proofs of Lemmas~\ref{sCcover_lmm}-\ref{cZtopol_lmm2}
make no use of the properties
of Definition~\ref{SympCut_dfn1} and thus apply to any Hamiltonian configuration~$\sC$
of Definition~\ref{SympCut_dfn0}.

\begin{lmm}\label{sCcover_lmm}
The collection $\{U_{i;I}^{\le}\}_{i\in I\subset[N]}$ covers~$X$.
If $\sC$ is a maximal cutting configuration, 
then the collection $\{U_{I;I}^<\}_{I\in\cP^*(N)}$ covers~$X$.
\end{lmm}

\begin{proof} Let $x\!\in\!X$ and $I\!\in\!\cP^*(N)$ be such that $x\!\in\!U_I$.
The condition $(\mu_I(x))_i\!\le\!(\mu_I(x))_j$ defines a reflexive transitive relation~$\le$ 
on~$I$ such that either $i\!\le\!j$ or $j\!\ge\!i$ for all $i,j\!\in\!I$.
Thus, there exists $i\!\in\!I$ such that $(\mu_I(x))_i\!\le\!(\mu_I(x))_j$ for all
$j\!\in\!I$
and so $x\!\in\!U_{i;I}^{\le}$.
Let
$$I_0=\big\{i_0\!\in\!I\!:\,(\mu_I(x))_i\!=\!(\mu_I(x))_{i_0}\big\}\ni i.$$
By the choice of~$i\!\in\!I$, $(\mu_I(x))_i\!<\!(\mu_I(x))_j$ 
for all $j\!\in\!I\!-\!I_0$ and~so
$$x\in \big\{x'\!\in\!U_I\!:\,(\mu_I(x))_{i_0}\!<\!(\mu_I(x))_j~\forall\,i_0\!\in\!I_0,\,
j\!\in\!I\!-\!I_0\big\}.$$
If $\sC$ is maximal, \eref{IJpos_e} with $(I,J)$ replaced by $(I_0,I)$
then implies that $x\!\in\!U_{I_0;I_0}^<$.
\end{proof}

\begin{lmm}\label{Xitopol_lmm1}
For every $i\!\in\![N]$, the subspace $U_i^{\le}\!\subset\!X$  is closed.
\end{lmm}

\begin{proof}
Let $(x_k)_{k=1}^{\i}$ be a sequence in $U_i^{\le}$ converging to some $x\!\in\!X$.
Since $X$ is first countable, it is sufficient to show that $x\!\in\!U_i^{\le}$.
By Definition~\ref{SympCut_dfn0}\ref{IJinter_it}, for each $k\!\in\!\Z^+$
there is a unique maximal $I_k\!\in\!\cP_i(N)$ such that $x_k\!\in\!U_{i;I_k}^{\le}$. 
Passing to a subsequence, we can assume that $I_k\!\equiv\!I$ for all $k\!\in\!\Z^+$
and  for some fixed $I\!\in\!\cP_i(N)$.
Thus,
\BE{Xitopol1_e1}
\big(\mu_I(x_k)\!\big)_{\!ij} \ge0\qquad \forall\,k\!\in\!\Z^+,\,j\!\in\!I.\EE
Let $J$ be the maximal element of $\cP^*(N)$ such that $x\!\in\!U_J$. 
Since $U_J$ is open, 
$x_k\!\in\!U_I\!\cap\!U_J$ for all $k\!\in\!\Z^+$ sufficiently large.
By Definition~\ref{SympCut_dfn0}\ref{IJinter_it}, either $I\!\subsetneq\!J$ or $J\!\subset\!I$.\\

\noindent
Suppose first $I\!\subsetneq\!J$.
By~\ref{phiIUJ_it} and~\ref{UIJpos_it} in Definition~\ref{SympCut_dfn0}
and~\eref{Xitopol1_e1},
\BE{Xitopol1_e2ab}
\big(\mu_J(x_k)\!\big)_{\!ij}=\big(\mu_I(x_k)\!\big)_{\!ij} \ge0
\quad\forall\,j\!\in\!I, \qquad
\big(\mu_J(x_k)\!\big)_{\!ij}>0 \quad\forall\,j\!\in\!J\!-\!I\EE
for all $k\!\in\!\Z^+$  sufficiently large.
Thus, $x_k\!\in\!U_{i;J}^{\le}$;
this contradicts the maximality assumption on~$I$ above.\\

\noindent 
Suppose $J\!\subset\!I$ instead.
If $i\!\not\in\!J$, then $i\!\in\!I\!-\!J$ and
$$\big(\mu_I(x_k)\!\big)_{\!j}<\big(\mu_I(x_k)\!\big)_{\!i} \qquad\forall~j\!\in\!J$$
for all $k$ sufficiently large
by Definition~\ref{SympCut_dfn0}\ref{UIJpos_it}
with the roles of $(i,I)$ and $(j,J)$ interchanged;
this contradicts~\eref{Xitopol1_e1}.
By  Definition~\ref{SympCut_dfn0}\ref{phiIUJ_it} and~\eref{Xitopol1_e1}, 
$$ \big(\mu_J(x_k)\!\big)_{\!ij}= \big(\mu_I(x_k)\!\big)_{\!ij}\ge0 
\quad \forall\,x_k\!\in\!U_I\!\cap\!U_J,\,j\!\in\!J\,.$$
By the continuity of the functions~$(\mu_J(\cdot))_{ij}$, this implies that
$$\big(\mu_J(x)\!\big)_{\!ij}=
\lim_{k\lra\i}\big(\mu_J(x_k)\!\big)_{\!ij}\ge0 \qquad\forall~j\!\in\!J.$$
We conclude that $x\!\in\!U_{i;J}^{\le}\!\subset\!U_i^{\le}$.
\end{proof}

\noindent
For $I_0\!\in\!\cP^*(N)$ and $I\!\in\!\cP_{I_0}(N)$, define
\BE{UI0I_e} 
U_{I_0;I}^{\le}
= \big\{x\!\in\!U_I\!:
(\mu_I(x))_i\!\le\!(\mu_I(x))_j \,\forall\,i\!\in\!I_0,\,j\!\in\!I\big\}.\EE
By~\eref{UiIdfn_e},
$$U_{I_0;I}^{\le} = 
\big\{x\!\in\!U_{i;I}^{\le}\!:(\mu_I(x))_i\!=\!(\mu_I(x))_j\,\forall\,j\!\in\!I_0\big\}$$
for any $i\!\in\!I_0$.

\begin{lmm}\label{Xitopol_lmm2}
For every $I_0\!\in\!\cP^*(N)$, the subspace
$$U_{I_0}^{\le}\equiv 
\bigcup_{I\in\cP_{I_0}(N)}\!\!\!\!\!\!\!U_{I_0;I}^{\le}\subset X$$
is closed in~$X$.
For all $I_0,J_0\!\in\!\cP^*(N)$, $U_{I_0}^{\le}\!\cap\!U_{J_0}^{\le}\!=\!U_{I_0\cup J_0}^{\le}$.
\end{lmm}

\begin{proof} Let $i\!\in\!I_0$. Since 
$$U_{I_0}^{\le} = \cU_{i,I_0}^{\le}\cap\big(X\!\times\!\{0\}\big)
\qquad\hbox{and}\qquad
U_i^{\le} = \cU_i^{\le}\cap\big(X\!\times\!\{0\}\big),$$
the first claim follows from Lemmas~\ref{Xitopol_lmm1} and~\ref{cZtopol_lmm2}.\\

\noindent
It is immediate that $U_{I_0}^{\le},U_{J_0}^{\le}\!\supset\!U_{I_0\cup J_0}^{\le}$.
Suppose $x\!\in\!U_{I_0;I}^{\le}\!\cap\!U_{J_0;J}^{\le}$.
Thus, $I_0\!\subset\!I$, $J_0\!\subset\!J$, 
\BE{Xitopol2_e2}
\big(\mu_I(x)\!\big)_{\!i}\!\le\!\big(\mu_I(x)\!\big)_{\!j}
~~\forall\,i\!\in\!I_0,\,j\!\in\!I, \quad
\big(\mu_J(x)\!\big)_{\!i}\!\le\!\big(\mu_J(x)\!\big)_{\!j}
~~\forall\,i\!\in\!J_0,\,j\!\in\!J,\EE
and either $I\!\subset\!J$ or $I\!\supset\!J$ by Definition~\ref{SympCut_dfn0}\ref{IJinter_it}.
We can assume that $I\!\subset\!J$. 
By~\ref{phiIUJ_it} and~\ref{UIJpos_it} in  Definition~\ref{SympCut_dfn0}
and the first statement in~\eref{Xitopol2_e2}, 
$$\big(\mu_J(x)\!\big)_{\!i}\!\le\!\big(\mu_J(x)\!\big)_{\!j}
~~\forall\,i\!\in\!I_0,\,j\!\in\!J.$$
Combining this with  the second statement in~\eref{Xitopol2_e2}, we obtain
$$\big(\mu_J(x)\!\big)_{\!i}\!\le\!\big(\mu_J(x)\!\big)_{\!j}
~~\forall\,i\!\in\!I_0\!\cup\!J_0,\,j\!\in\!J.$$
Thus, $x\!\in\!U_{I_0\cup J_0;J}^{\le}\subset U_{I_0\cup J_0}^{\le}$.
\end{proof}

\begin{lmm}\label{cZtopol_lmm1} 
For all $i\!\in\!I\!\subset\![N]$, the subspace $\cU_{i;I}^{\le}\!\subset\!\cU_i^{\le}$ 
is open in~$\cU_i^{\le}$ and
$$\cU_i^{\le}\cap \big(U_I\!\times\!\C\big) =  \cU_{i;I}^{\le}\,.$$
\end{lmm}

\begin{proof}
The first claim follows from the second. Suppose 
$$(x,z)\in \cU_{i;J}^{\le} \cap \big(U_I\!\times\!\C\big)$$
for some $J\!\in\!\cP_i(N)$ and thus  
\BE{cZtopol1_e1a}
\big(\mu_J(x)\!\big)_{\!ij}+\frac12|z|^2 \ge0\qquad
\forall\,j\!\in\!J.\EE
By Definition~\ref{SympCut_dfn0}\ref{IJinter_it}, either 
$I\!\subset\!J$ or $J\!\subset\!I$.
If $I\!\subset\!J$, Definition~\ref{SympCut_dfn0}\ref{phiIUJ_it} implies
that~\eref{cZtopol1_e1a} holds with~$J$ replaced by~$I$ and so $(x,z)\!\in\!\cU_{i;I}^{\le}$.
If $J\!\subset\!I$, \ref{phiIUJ_it} and~\ref{UIJpos_it} in Definition~\ref{SympCut_dfn0}
imply that~\eref{Xitopol1_e2ab} holds for $J\!=\!I_0$ and $x_k\!=\!x$.
From this, we again conclude that~\eref{cZtopol1_e1a} holds with~$I$ and by~$J$ interchanged 
and so $(x,z)\!\in\!\cU_{i;I}^{\le}$.
This establishes the second claim of the lemma.
\end{proof}

\noindent
For all $i\!\in\!I\!\subset[N]$ and $I_0\!\subset\!I$, define
\BE{cUiI0I_e}
\cU_{i,I_0;I}= 
\big\{(x,z)\!\in\!U_I\!\times\!\C\!:\,
(\mu_I(x))_i\!-\!\frac12|z|^2\!=\!(\mu_I(x))_j
\,\forall\,j\!\in\!I_0\big\},\quad
\cU_{i,I_0;I}^{\le}= \cU_{i,I_0;I}\!\cap \cU_{i;I}^{\le}\,.\EE

\begin{lmm}\label{cZtopol_lmm2}
For all $i\!\in\![N]$ and  $I_0\!\subset\![N]$, the subspace
$$\cU_{i,I_0}^{\le}\equiv
\bigcup_{I\in\cP_i(N)\cap\cP_{I_0}(N)}
\!\!\!\!\!\!\!\!\!\!\!\!\!\!\cU_{i,I_0;I}^{\le} \subset \cU_i^{\le}$$
is closed.
\end{lmm}

\begin{proof} 
Let $(x_k,z_k)_{k=1}^{\i}$ be a sequence in $\cU_{i,I_0}^{\le}$ converging to 
some $(x,z)\!\in\!\cU_i^{\le}$.
Similarly to the proof of Lemma~\ref{Xitopol_lmm1},  we can assume 
$(x_k,z_k)_{k=1}^{\i}\!\in\!\cU_{i,I_0;I}^{\le}$  for all $k\!\in\!\Z^+$
and for some fixed $I\!\in\!\cP_i(N)\!\cap\!\cP_{I_0}(N)$.
Thus,
\BE{cZtopol12_e1}
\big(\mu_I(x_k)\!\big)_{\!ij}+\frac12|z_k|^2\ge0 \quad \forall\,j\!\in\!I, \qquad
\big(\mu_I(x_k)\!\big)_{\!ij}+\frac12|z_k|^2=0 \quad \forall\,j\!\in\!I_0.\EE
We also assume that $I\!\in\!\cP_i(N)\!\cap\!\cP_{I_0}(N)$ is the maximal element
with this property.\\  

\noindent
Let $J$ be the maximal element of $\cP_i(N)$ such that $x\!\in\!U_J$. 
By the same reasoning as in the second paragraph of the proof of Lemma~\ref{Xitopol_lmm1}, 
$J\!\subset\!I$.
Thus, $i\!\in\!J\!\subset\!I$ and $x_k\!\in\!U_J$ for all~$k$ sufficiently large.
By Definition~\ref{SympCut_dfn0}\ref{phiIUJ_it} and the first statement in~\eref{cZtopol12_e1}, 
$$\big(\mu_J(x_k)\!\big)_{\!ij}+\frac12|z_k|^2
=\big(\mu_I(x_k)\!\big)_{\!ij}+\frac12|z_k|^2\ge0
\quad \forall\,k\!\in\!\Z^+,\,j\!\in\!J.$$
By the continuity of the functions~$(\mu_J(\cdot))_{ij}$, 
this statement implies that $(x,z)\!\in\!\cU_{i;J}^{\le}$.
If $I_0\!\not\subset\!J$, then
$$(\mu_I(x_k))_i<(\mu_I(x_k))_{i_0} \qquad\forall~
i_0\!\in\!I_0\!-\!J\subset I\!-\!J$$
for all $k$ sufficiently large by Definition~\ref{SympCut_dfn0}\ref{UIJpos_it};
this contradicts the second statement in~\eref{cZtopol12_e1}. 
Thus,  $I_0\!\subset\!J$ and
$$\big(\mu_J(x_k)\!\big)_{\!ij}+\frac12|z_k|^2
=\big(\mu_I(x_k)\!\big)_{\!ij}+\frac12|z_k|^2
=0 \quad \forall\,k\!\in\!\Z^+,\,j\!\in\!I_0.$$
By the continuity of the functions~$(\mu_J(\cdot))_{ij}$, 
this statement implies that $(x,z)\!\in\!\cU_{i,I_0;J}^{\le}\!\subset\!\cU_{i,I_0}^{\le}$.
\end{proof}

\subsection{Symplectic reduction}
\label{SympRed_subs}

\noindent
For  $I_0\!\subset\!I\!\subset\![N]$ and $F\!=\!\R,\C$, we identify 
$$F^N_{I_0}\equiv\big\{(z_i)_{i\in[N]}\!\in\!F^N\!\!:z_i\!=\!0~\forall\,i\!\in\!I_0\big\}
\quad\hbox{and}\quad
F^I_{I_0}\equiv\big\{(z_i)_{i\in I}\!\in\!F^I\!\!:z_i\!=\!0~\forall\,i\!\in\!I_0\big\}$$
with $F^{[N]-I_0}$ and $F^{I-I_0}$, respectively.
For $i\!\in\!I\!\subset\![N]$, let
$$\C^N_i=\C^N_{\{i\}} \qquad\hbox{and}\qquad \C^I_i=\C^I_{\{i\}}.$$
We  denote~by  $\om_{\C^I}$ and $\om_{\C_{I_0}^I}$
the standard symplectic forms on~$\C^I$ and~$\C_{I_0}^I$, respectively.
Thus,  
$$\om_{\C^{I-I_0}}=\om_{\C^I}|_{\C^I_{I_0}}=\om_{\C^N}\big|_{\C^{I-I_0}}$$
under the above identifications.\\

\noindent
The standard $(S^1)^N$-action on $\C^N$ and a moment map for this action are given~by
\BE{phimustan_e}\begin{aligned}
\phi_{\C^N}\!: (S^1)^N\!\times\!\C^N&\lra\C^N, &\qquad
\phi_{\C^N}\big((\ne^{\fI\th_i})_{i\in[N]};(z_i)_{i\in[N]}\big)
&=\big(\ne^{\fI\th_i}z_i\big)_{i\in[N]},\\
\mu_{\C^N}\!: \C^N&\lra\R^N, &\qquad
\mu_{\C^N}\big((z_i)_{i\in[N]}\big)&=\frac12\big(|z_i|^2\big)_{i\in[N]}\,.
\end{aligned}\EE
For $I\!\in\!\cP^*(N)$, let
$$\phi_{\C^I;\bu}\!: (S^1)^I_{\bu}\!\times\!\C^I\lra\C^I
\qquad\hbox{and}\qquad \mu_{\C^I;\bu}\!: \C^I\lra\ft_{I;\bu}^*$$
be the restriction of this action and the induced moment map, respectively.\\

\noindent 
With notation as in~\eref{SympCutDfn_e}, let 
$$\pi_1,\pi_2\!:U_I\!\times\!\C^I\lra U_I,\C^I$$
be the projection maps and
$$\wt\om_I=\pi_1^*\om+\pi_2^*\om_{\C^I}$$
be the product symplectic form.
We lift $(\phi_I,\mu_I)$ to a Hamiltonian $(S^1)^I_{\bu}$-pair for $(U_I\!\times\!\C^I,\wt\om_I)$~by
\begin{alignat}{2}
\label{wtphiIdfn_e}
\wt\phi_I\!:(S^1)^I_{\bu}\times \big(U_I\!\times\!\C^I\big) &\lra U_I\!\times\!\C^I, 
&\qquad \wt\phi_I(g;x,z)&=\big(\phi_I(g;x),\phi_{\C^I;\bu}(g^{-1};z)\big),\\
\label{wtmuIdfn_e}
\wt\mu_I\!:U_I\!\times\!\C^I&\lra\ft_{I;\bu}^*, &\qquad
\wt\mu_I(x,z)&=\mu_I(x)\!-\!\mu_{\C^I;\bu}(z).
\end{alignat}
The action~$\wt\phi_I$ then preserves the symplectic submanifolds 
$U\!\times\!\C^I_{I_0}\!\subset\!U\!\times\!\C^I$
with $I_0\!\subset\!I$.
Let
\BE{wtcZIdfn_e} 
\wt\cZ_I^{\circ}\equiv\wt\mu_I^{-1}(0)
=\big\{\big(x,(z_i)_{i\in I}\big)\!\in\!U_I\!\times\!\C^I\!\!:
(\mu_I(x))_i\!-\!\frac12|z_i|^2\!=\!(\mu_I(x))_j\!-\!\frac12|z_j|^2
~\forall\,i,j\!\in\!I\big\}.\EE
Similarly to the situation at the beginning of Section~\ref{TopolPrelim_subs},
the condition on the elements of $U_I\!\times\!\C^I$ above is independent of 
the choice of representative for $\mu_I(x)$ in~$\R^I$ and is thus well-defined.

\begin{lmm}\label{wtcZreg_lmm}
Let $\sC$ be a cutting configuration as in~\eref{SympCutDfn_e}.
For all $I\!\in\!\cP^*(N)$ and $I_0\!\subset\!I$,
$0\!\in\!\ft_{I;\bu}^*$ is a regular value of
the restriction of~$\wt\mu_I$ to $U_I\!\times\!\C^I_{I_0}$.
For all $I\!\in\!\cP^*(N)$,
the restriction of the $\wt\phi_I$-action to~$\wt\cZ_I^{\circ}$ is free.
\end{lmm}

\begin{proof} 
Since $(\wt\phi_I,\wt\mu_I)$ restricts to a Hamiltonian pair on $U_I\!\times\!\C^I_{I_0}$,
the first claim is implied by the second; see \cite[Section~23.2.1]{daSilva}.
Let $\wt{x}\!=\!(x,(z_i)_{i\in I})$ be an element of $U_I\!\times\!\C^I$ 
such that $z_i\!=\!0$ if and only~if $i\!\in\!I_0$.
If
$$\wt\phi_I\big((\ne^{\fI\th_i})_{i\in I};x,(z_i)_{i\in I}\big)
=\big(x,(z_i)_{i\in I}\big),$$
then $\ne^{\fI\th_i}\!=\!1$ for every $i\!\in\!I\!-\!I_0$.
Thus,  
\BE{wtcZreg_e7b}
\big(\ne^{\fI\th_i}\big)_{i\in [N]}\in(S^1)^{I_0}_{\bu}, \qquad
\phi_I\big((\ne^{\fI\th_i})_{i\in I};x\big)=x.\EE
If $|I_0|\!\le\!1$, 
$\ne^{\fI\th_i}\!=\!1$ for all $i\!\in\![N]$ by the first statement in~\eref{wtcZreg_e7b}.
If $I_0\!\neq\!\eset$, then $x\!\in\!\mu_{I_0;I}^{-1}(0)$ by~\eref{wtcZIdfn_e} and~\eref{hdfn_e}.
By~\eref{wtcZreg_e7b} and the assumption of Definition~\ref{SympCut_dfn1},
this implies that $\ne^{\fI\th_i}\!=\!1$ for all $i\!\in\![N]$ and
establishes the second claim.
\end{proof}

\noindent
For $I\!\in\!\cP^*(N)$, define
\BE{wtcZiIdfn_e} 
\qquad \cZ_I^{\circ}=\wt\cZ_I^{\circ}\big/(S^1)^I_{\bu}\,.\EE
For example, $\cZ_{\{i\}}^{\circ}\!=\!U_{\{i\}}\!\times\!\C$.
Let 
$$ q_{\cZ;I}\!:\wt\cZ_I^{\circ}\lra \cZ_I^{\circ}$$
be the projection~map.
By Lemma~\ref{wtcZreg_lmm} and the Symplectic Reduction Theorem \cite[Theorem~23.1]{daSilva},
$\wt\cZ_I^{\circ}$ is a smooth submanifold of~$U_I\!\times\!\C^I$,
$\cZ_I^{\circ}$ is a smooth manifold, and
there is a unique symplectic form~$\vp_I$ on~$\cZ_I^{\circ}$ such~that
\BE{SympRed_e4} q_{\cZ;I}^{\,*}\vp_I=\wt\om_I\big|_{\wt\cZ_I^{\circ}}\,.\EE
Thus, $(\cZ_I^{\circ},\vp_I)$ is a symplectic manifold satisfying~\eref{dimcZX_e}.\\

\noindent
For $I\!\in\!\cP^*(N)$ and $I_0\!\subset\!I$, the subspace
$$\wt\cZ_{I_0;I}^{\circ}\equiv\wt\cZ_I^{\circ}\cap \big(U_I\!\times\!\C^I_{I_0}\big)
\subset U_I\!\times\!\C^I$$
is preserved by the $\wt\phi_I$-action.
Let 
\BE{SympCutSubm_e2} 
\cZ_{I_0;I}^{\circ}=\wt\cZ_{I_0;I}^{\circ}\big/(S^1)^I_{\bu}=
q_{\cZ;I}\big(\wt\cZ_{I_0;I}^{\circ}\big)\subset \cZ_I^{\circ},\quad
q_{\cZ;I_0;I}\!=\!q_{\cZ;I}|_{\wt\cZ_{I_0;I}^{\circ}}\!\!:
\wt\cZ_{I_0;I}^{\circ}\lra\cZ_{I_0;I}^{\circ}\,.\EE
By Lemma~\ref{wtcZreg_lmm} and the Symplectic Reduction Theorem,
$\wt\cZ_{I_0;I}^{\circ}$ is a smooth submanifold of~$U_I\!\times\!\C^I_{I_0}$,
$\cZ_{I_0;I}^{\circ}$ is a smooth manifold, and
there is a unique symplectic form~$\vp_{I_0;I}$ on~$\cZ_{I_0;I}^{\circ}$ such~that
\BE{SympRed_e5} q_{\cZ;I_0;I}^{\,*}\vp_{I_0;I}=\wt\om_I\big|_{\wt\cZ_{I_0;I}^{\circ}}\,.\EE
Thus, $(\cZ_{I_0;I}^{\circ},\vp_{I_0;I})$ is a symplectic manifold of real 
dimension~$\dim_{\R}\!X\!+\!2\!-\!2|I_0|$.\\  

\noindent
For $I\!\in\!\cP^*(N)$  and $I_0'\!\subset\!I_0\subset\!I$, let
\begin{alignat}{1}
\label{wtcZcNo_e}
\wt\pi_{I_0;I_0';I}^{\circ}\!: \wt\cN_{I_0;I_0';I}^{\circ}= 
\wt\cZ_{I_0;I}^{\circ}\!\times\!\C^{I_0}_{I_0'}&\lra \wt\cZ_{I_0;I}^{\circ},\\
\label{cZcNo_e}
\pi_{I_0;I_0';I}^{\circ}\!:
\cN_{I_0;I_0';I}^{\circ}=\wt\cN_{I_0;I_0';I}^{\circ}\big/(S^1)^I_{\bu}
&\lra  \wt\cZ_{I_0;I}^{\circ}\big/(S^1)^I_{\bu}=\cZ_{I_0;I}^{\circ}\,,
\end{alignat}
with the quotients taken by the restriction of the $(S^1)^I_{\bu}$-action~$\wt\phi_I$ to 
$$\wt\cZ_{I_0;I}^{\circ}\!\times\!\C^{I_0}_{I_0'} \subset U_I\!\times\!\C^I\,.$$
Since the $\wt\phi_I$-action on $\wt\cZ_{I_0;I}^{\circ}$ is free, 
\eref{cZcNo_e} is a complex vector bundle.

\begin{lmm}\label{SympNB_lmm0}
Suppose  $\sC$ is a cutting configuration as in~\eref{SympCutDfn_e} and $I\!\in\!\cP^*(N)$.
For all \hbox{$I_0'\!\subset\!I_0\!\subset\!I$},
$(\cZ_{I_0;I}^{\circ},\vp_{I_0;I})$ is a symplectic submanifold of 
$(\cZ_{I_0';I}^{\circ},\vp_{I_0';I})$
with the oriented normal bundle 
canonically isomorphic to~\eref{cZcNo_e}.
\end{lmm}

\begin{proof}
By~\eref{SympRed_e4}-\eref{SympRed_e5},
\BE{SympNB0_e1}q_{\cZ;I_0;I}^{\,*}\big(\vp_I|_{\cZ_{I_0;I}^{\circ}}\big)
=\big(q_{\cZ;I}^{\,*}\vp_I\big)\big|_{\wt\cZ_{I_0;I}^{\circ}}
=\wt\om_I\big|_{\wt\cZ_{I_0;I}^{\circ}}
=q_{\cZ;I_0;I}^{\,*}\vp_{I_0;I}\,.\EE
By the uniqueness part of \cite[Theorem~23.1]{daSilva} and~\eref{SympNB0_e1},
 $\vp_{I_0;I}\!=\!\vp_I|_{\cZ_{I_0;I}^{\circ}}$.
Thus, $(\cZ_{I_0;I}^{\circ},\vp_{I_0;I})$
is a symplectic submanifold of $(\cZ_I^{\circ},\vp_I)$.
Since $\cZ_{I_0;I}^{\circ}\!\subset\!\cZ_{I_0';I}^{\circ}$ whenever 
$I_0\!\supset\!I_0'$, 
it follows that $(\cZ_{I_0;I}^{\circ},\vp_{I_0;I})$ is a symplectic submanifold of 
$(\cZ_{I_0';I}^{\circ},\vp_{I_0';I})$.\\

\noindent
For all $I_0'\!\subset\!I_0\!\subset\!I$,
$\wt\cZ_{I_0';I}^{\circ}\!\subset\!U_I\!\times\!\C^I_{I_0'}$.
The projection $\C^I_{I_0'}\!\lra\!\C^{I_0}_{I_0'}$ induces
a vector bundle homomorphism
\BE{cNisom_e}\cN_{\wt\cZ_{I_0';I}^{\circ}}\wt\cZ_{I_0;I}^{\circ}\equiv 
\frac{T\wt\cZ_{I_0';I}^{\circ}|_{\wt\cZ_{I_0;I}^{\circ}}}{T\wt\cZ_{I_0;I}^{\circ}}
\lra \wt\cZ_{I_0;I}^{\circ}\!\times\!\C^{I_0}_{I_0'}\equiv\wt\cN_{I_0;I_0';I}^{\circ}\,.\EE
By Lemma~\ref{wtcZreg_lmm},
the homomorphism~\eref{cNisom_e} is surjective and thus an isomorphism for dimensional
reasons.
Since it is $\wt\phi_I$-equivariant, it descends to a vector bundle isomorphism
\BE{cNisom_e2}
\cN_{\cZ_{I_0';I}^{\circ}}\cZ_{I_0;I}^{\circ}\equiv 
\frac{T\cZ_{I_0';I}^{\circ}|_{\cZ_{I_0;I}^{\circ}}}{T\cZ_{I_0;I}^{\circ}}
\lra  \cN_{I_0;I_0';I}^{\circ}\EE
over~$\cZ_{I_0;I}^{\circ}$.
The fiberwise symplectic form on the left-hand side of~\eref{cNisom_e2} 
corresponds
to the symplectic form on the right-hand side of~\eref{cNisom_e2}
induced by the standard symplectic form on~$\C^{I_0}_{I_0'}$;
the latter is compatible with the complex orientation of~$\C^{I_0}_{I_0'}$.
Thus, the isomorphism~\eref{cNisom_e2} is orientation-preserving. 
\end{proof}

\subsection{The ambient space}
\label{NatCutSmoothSp_subs}

\noindent
We will glue the symplectic manifolds $(\cZ_I^{\circ},\vp_I)$ with 
$I\!\in\!\cP^*(N)$ into a single symplectic manifold~$(\cZ,\om_{\cZ})$.
For $i\!\in\!I,J\!\subset\![N]$, define
\BE{S1IJhomom_e}
\vr_{i;J,I}\!:(S^1)^I_{\bu}\lra(S^1)^J_{\bu}, \quad
\big(\vr_{i;J,I}\big((\ne^{\fI\th_k})_{k\in I}\big)\!\big)_l
=\begin{cases}
\ne^{\fI\th_i}\!\!\!\!\!\prod\limits_{k\in I-J}\!\!\!\!\!\ne^{\fI\th_k},&\hbox{if}~l\!=\!i;\\
\ne^{\fI\th_l},&\hbox{if}~l\!\in\!I\!\cap\!J\!-\!i;\\
1,&\hbox{if}~l\!\in\!J\!-\!I.
\end{cases}\EE  
For $I,J\!\in\!\cP^*(N)$,  let
\BE{whThiIJdfn_e0}
\wt\cZ_{I,J}^{\circ}=\wt\cZ_I^{\circ}\cap\big((U_I\!\cap\!U_J)\!\times\!\C^I\big)\,.\EE
By Definition~\ref{SympCut_dfn0}\ref{IJinter_it}, $\wt\cZ_{I,J}^{\circ}\!=\!\eset$
unless $I\!\subset\!J$ or $I\!\supset\!J$. 
By Definition~\ref{SympCut_dfn0}\ref{UIJpos_it} and~\eref{wtcZIdfn_e},
\begin{alignat}{2}
\label{IsubJ_e}
\big(\mu_J(x)\big)_{\!ij}\!\equiv\!\big(\mu_J(x)\big)_{\!j}&\!-\!\big(\mu_J(x)\big)_{\!i}>0
&\qquad&\forall~\big(x,(z_k)_{k\in I}\big)\!\in\!\wt\cZ_{I,J}^{\circ},\,
i\!\in\!I\!\subset\!J,\,j\!\in\!J\!-\!I;\\
\label{IsupJ_e}
z_j&\neq0 &\qquad&\forall~\big(x,(z_k)_{k\in I}\big)\!\in\!\wt\cZ_{I,J}^{\circ},\,
\eset\!\neq\!J\!\subset\!I,\,j\!\in\!I\!-\!J.
\end{alignat}
If $i\!\in\!J\!\subset\!I\!\subset\![N]$, define 
$$\vph_{I,J;i}\!:\wt\cZ_{I,J}^{\circ}\lra(S^1)^I_{\bu}, \quad
\big(\vph_{I,J;i}\big(x,(z_k)_{k\in I}\big)\!\big)_{\!l}
=\begin{cases}
\prod\limits_{k\in I-J}\!\frac{\ov{z_k}}{|z_k|},&\hbox{if}~l\!=\!i;\\
1,&\hbox{if}~l\!\in\!J\!-\!i;\\
\frac{z_l}{|z_l|},&\hbox{if}~l\!\in\!I\!-\!J.
\end{cases}$$
If $J\!\subset\!I\!\subset\![N]$ and $i,j\!\in\!J$ are distinct, define 
$$\vph_{I,J;i,j}\!:\wt\cZ_{I,J}^{\circ}\lra(S^1)^J_{\bu}, \quad
\big(\vph_{I,J;i,j}\big(x,(z_k)_{k\in I}\big)\!\big)_{\!l}
=\begin{cases}
\prod\limits_{k\in I-J}\!\frac{\ov{z_k}}{|z_k|},&\hbox{if}~l\!=\!i;\\
\prod\limits_{k\in I-J}\!\frac{z_k}{|z_k|},&\hbox{if}~l\!=\!j;\\
1,&\hbox{if}~l\!\in\!J\!-\!\{i,j\}.
\end{cases}$$
By~\eref{IsupJ_e}, both maps are well-defined ($z_k\!\neq\!0$ for all $k\!\in\!I\!-\!J$).
We take $\vph_{I,J;i,i}$ to be the map taking $\wt\cZ_{I,J}^{\circ}$
to the identity in~$(S^1)^J_{\bu}$.\\

\noindent
For $i\!\in\!I,J\!\subset\![N]$,  we define
\BE{whThiIJdfn_e}\begin{split}
\wt\Th_{J,I;i}\!: \wt\cZ_{I,J}^{\circ}&\lra\wt\cZ_{J,I}^{\circ}, \\
\wt\Th_{J,I;i}\big(x,(z_k)_{k\in I}\big) &=
\begin{cases}
\big(x,(z_k)_{k\in I},\big(\sqrt{2(\mu_J(x))_{ik}\!+\!|z_i|^2}\big)_{k\in J-I}\big),&
\hbox{if}~I\!\subset\!J;\\
\wt\phi_I\big(\vph_{I,J;i}\big(x,(z_k)_{k\in I}\big);x,(z_k)_{k\in J}\big),&
\hbox{if}~I\!\supset\!J.
\end{cases}
\end{split}\EE
By~\eref{wtcZIdfn_e} and~\eref{IsubJ_e},
this map is well-defined in either case.
It is $(\wt\phi_J,\wt\phi_I)$-equivariant with respect to the homomorphism~\eref{S1IJhomom_e}
and 
\BE{whThiIJprp_e0}
\wt\Th_{J,I;i}^{\,*}\big(\wt\om_J\big|_{\wt\cZ_{J,I}^{\circ}}\big)
=\wt\om_I\big|_{\wt\cZ_{I,J}^{\circ}}.\EE
Furthermore,
\BE{whThiIJprp_e12}\begin{split}
\wt\Th_{I,J;i}\big(\wt\Th_{J,I;i}(x,z)\big)&=
\begin{cases}
(x,z),&\hbox{if}~I\!\subset\!J;\\
\wt\phi_I\big(\vph_{I,J;i}(x,z);x,z\big)
,&\hbox{if}~I\!\supset\!J;
\end{cases}\\
\wt\Th_{K,I;i}\big|_{\wt\cZ_{I,J}^{\circ}\cap\wt\cZ_{I,K}^{\circ}}
&=\wt\Th_{K,J;i}\!\circ\!\wt\Th_{J,I;i}\big|_{\wt\cZ_{I,J}^{\circ}\cap\wt\cZ_{I,K}^{\circ}}
\quad\forall~i\!\in\!I\!\subset\!J\!\subset\!K\!\subset\![N]\,.
\end{split}\EE
For $i,j\!\in\!I,J$,
\BE{whThiIJprp_e3}
\wt\Th_{J,I;i}(x,z)=\begin{cases}
\wt\Th_{J,I;j}(x,z),&\hbox{if}~I\!\subset\!J;\\
\wt\phi_J\big(\vph_{I,J;i,j}(x,z);  
\wt\Th_{J,I;j}(x,z)\big),&\hbox{if}~I\!\supset\!J;
\end{cases}\EE
the claim in the first case above follows from~\eref{wtcZIdfn_e}.\\

\noindent
For $I,J\!\in\!\cP^*(N)$, let
\BE{cZiIJdfn_e}\cZ_{I,J}^{\circ}=q_{\cZ;I}\big(\wt\cZ_{I,J}^{\circ}\big)
\subset \cZ_I^{\circ}.\EE
By the $(\wt\phi_J,\wt\phi_I)$-equivariance of the smooth map~$\wt\Th_{J,I;i}$
with $i\!\in\!I,J$,
it descends to a smooth~map  
\BE{ThiJIdfn_e2}\Th_{J,I}\!:\cZ_{I,J}^{\circ}\lra \cZ_{J,I}^{\circ}\,.\EE
By~\eref{whThiIJprp_e3}, $\Th_{J,I}$ is independent of the choice of $i\!\in\!I,J$.
By~\eref{whThiIJprp_e12},
\BE{ThCoc_e}
\Th_{K,I}\big|_{\cZ_{I,J}^{\circ}\cap\cZ_{I,K}^{\circ}}
=\Th_{K,J}\circ \Th_{J,I}\big|_{\cZ_{I,J}^{\circ}\cap\cZ_{I,K}^{\circ}}\EE
whenever $I\!\subset\!J\!\subset\!K$ or some permutation of this relation holds.\\

\noindent
We define
\BE{cZdfn_e}\cZ=\bigg(\bigsqcup_{I\in\cP^*(N)}\!\!\!\!\!\!\cZ_I^{\circ}\bigg)\Big/\!\!\sim,
\qquad
\cZ_{I,J}^{\circ}\ni x\sim \Th_{J,I}(x)\in \cZ_{J,I}^{\circ}
\quad\forall\,I,J\!\in\!\cP^*(N).\EE
For $i\!\in\![N]$, let
\BE{cZidfn_e}
\cZ_i^{\star}=\bigg(\bigsqcup_{I\in\cP_i(N)}\!\!\!\!\!\!\cZ_I^{\circ}\bigg)\Big/\!\!\sim,
\qquad
\cZ_{I,J}^{\circ}\ni x\sim \Th_{J,I}(x)\in \cZ_{J,I}^{\circ}
\quad\forall\,I,J\!\in\!\cP_i(N).\EE

\begin{lmm}\label{cZtopol_lmm3}
The quotient projection~map
\BE{cZproj_e}  q_{\cZ}\!:\bigsqcup_{I\in\cP^*(N)}\!\!\!\!\!\!\cZ_I^{\circ}\lra \cZ \EE
is open and its restriction to each subspace~$\cZ_I^{\circ}$ is injective.
For each $i\!\in\![N]$, $\cZ_i^{\star}$ is an open subspace of~$\cZ$.
\end{lmm}

\begin{proof}
By~\eref{ThCoc_e}, the relations~$\sim$ in~\eref{cZdfn_e} and~\eref{cZidfn_e}
are equivalence relations and thus each $\cZ_i^{\star}$ is a subspace of~$\cZ$.
Since $\Th_{I,I}$ is the identity on $\cZ_{I,I}^{\circ}\!=\!\cZ_I^{\circ}$,
\eref{cZdfn_e} implies that the map~\eref{cZproj_e}
is injective on each subspace~$\cZ_I^{\circ}$.\\

\noindent
Since the maps~\eref{ThiJIdfn_e2} are homeomorphisms between open subsets of
the domain of~\eref{cZproj_e}, the latter map is open.
Since the preimage of~$\cZ_i^{\star}$ under the restriction of~\eref{cZproj_e} 
to each~$\cZ_I^{\circ}$ is the union of the open subsets
$\cZ_{I,J}^{\circ}\!\subset\!\cZ_I^{\circ}$ with 
$J\!\!\in\!\cP_i(N)\!\cap\!\cP_I(N)$, 
$\cZ_i^{\star}\!\subset\!\cZ$ is open. 
\end{proof}

\noindent
For $I,J\!\in\!\cP^*(N)$ and $I_0\!\subset\!I,J$, the overlap map~\eref{ThiJIdfn_e2} 
takes $\cZ_{I_0;I}^{\circ}\!\cap\!\cZ_{I,J}^{\circ}$ to 
$\cZ_{I_0;J}^{\circ}\!\cap\!\cZ_{J,I}^{\circ}$.
Let 
\BE{cZI0dfn_e}
\cZ_{I_0}=
\bigg(\bigsqcup_{I\in\cP^*(N)\cap\cP_{I_0}(N)}
\!\!\!\!\!\!\!\!\!\!\!\cZ_{I_0;I}^{\circ}\bigg)\Big/\!\!\sim,\qquad
\cZ_{I_0;I}^{\circ}\!\cap\!\cZ_{I,J}^{\circ}\ni x\sim 
\Th_{J,I}(x)\in \cZ_{I_0;J}^{\circ}\!\cap\!\cZ_{J,I}^{\circ}.\EE
For $i\!\in\![N]$, let 
$$\cZ_{i,I_0}^{\star}=\bigg(\bigsqcup_{I\in\cP_i(N)\cap\cP_{I_0}(N)}
\!\!\!\!\!\!\!\!\!\!\!\cZ_{I_0;I}^{\circ}\bigg)\Big/\!\!\sim,\qquad
\cZ_{I_0;I}^{\circ}\!\cap\!\cZ_{I,J}^{\circ}\ni x\sim 
\Th_{J,I}(x)\in \cZ_{I_0;J}^{\circ}\!\cap\!\cZ_{J,I}^{\circ}.$$
Since the relations~$\sim$ in~\eref{cZdfn_e} and~\eref{cZidfn_e} are equivalence relations,
$\cZ_{I_0}$ is a subspace of~$\cZ$ and 
$\cZ_{i,I_0}^{\star}$ is a subspace of~$\cZ_i^{\star}$.
Furthermore,
\BE{cZI0cZi_e}\cZ_{\eset}=\cZ, \quad \cZ_{i,\eset}^{\star}=\cZ_i^{\star}~~\forall\,i\!\in\![N],
\quad
\cZ_{i,I_0}^{\star}=\cZ_{I_0}\!\cap\!\cZ_i^{\star}\subset \cZ
~~\forall\,i\!\in\![N],\,I_0\!\subset\![N].
\EE
For $i\!\in\![N]$ and $I_0\!\in\!\cP^*(N)$, let 
\BE{Xidfn_e} X_i=\cZ_{\{i\}}, \qquad X_{I_0}=\cZ_{I_0}\,.\EE
By~\eref{cZI0dfn_e} and~\eref{cZidfn_e}, $X_i\!\subset\!\cZ_i^{\star}$.
If $i\!\in\!I_0$, then $X_{I_0}\!\subset\!X_i$.\\

\noindent
For $i\!\in\!I\!\subset\![N]$, let $\cU_{i;I}^{\le}\!\subset\!U_I\!\times\!\C$
be as in~\eref{cUiIdfn_e} and define
$$g_{i;I}\!:\cU_{i;I}^{\le}\lra\C^I \quad\hbox{by}\quad
\big(g_{i;I}(x,z)\!\big)_{\!j}=
\begin{cases}
z,&\hbox{if}~j\!=\!i;\\
\sqrt{2(\mu_I(x))_{ij}\!+\!|z|^2},&\hbox{if}~j\!\in\!I\!-\!i.
\end{cases}$$
Thus, the map
\BE{giIgraphdfn_e}
\cU_{i;I}^{\le}\lra\wt\cZ_{I;i}^{\ge}
\equiv\big\{\big(x,(z_j)_{j\in I}\big)\!\in\!\wt\cZ_I^{\circ}\!:
z_j\!\in\!\R^{\ge0}~\forall\,j\!\in\!I\!-\!i\big\}, \quad
(x,z)\lra \big(x,g_{i;I}(x,z)\big),\EE
is a homeomorphism.
If $i\!\in\!I_0\!\subset\!I$, the restriction 
\begin{gather*}
g_{I_0;I}\!: U_{I_0;I}^{\le}\!=\big\{(x,0)\!\in\!\cU_{i;I}^{\le}:
(\mu_I(x))_i\!=\!(\mu_I(x))_j~\forall\,j\!\in\!I_0\big\}\lra
\big(\R^{\ge0}\big)^{I-I_0}\!\subset\!\C^I, \\
g_{I_0;I}(x)=\Big(\sqrt{2(\mu_I(x))_{ij}}\,\Big)_{\!j\in I-I_0}\,,
\end{gather*}
of~$g_{i;I}$ is independent of the choice of $i\!\in\!I_0$.

\begin{lmm}\label{cZtopol_lmm}
For every $i\!\in\![N]$, the~map
\BE{surjmap_e2} q_i\!: \cU_i^{\le}\lra \cZ_i^{\star}, \quad
(x,z)\lra q_{\cZ}\big(q_{\cZ;I}\big(x,g_{i;I}(x,z)\!\big)\!\big) 
\quad\forall~(x,z)\!\in\!\cU_{i;I}^{\le},\,I\!\in\!\cP_i(N),\EE
is well-defined, continuous, surjective, and closed.
For $I_0\!\in\!\cP_i(N)$, it restricts to the surjective closed~map
\BE{surjmap_e} q_{I_0}\!: U_{I_0}^{\le}\lra X_{I_0}, \quad
x\lra q_{\cZ}\!\big(q_{\cZ;I_0;I}\big(x,g_{I_0;I}(x)\big)\!\big)
~~\forall~x\!\in\!U_{I_0;I}^{\le},\,I\!\in\!\cP_{I_0}(N),\EE
which does not depend on the choice of $i\!\in\!I_0$.
\end{lmm}

\begin{proof}
If $(x,z)\!\in\!\cU_{i;I}^{\le}\!\cap\!\cU_{i;J}^{\le}$, 
then either $I\!\subset\!J$ or $J\!\subset\!I$
by Definition~\ref{SympCut_dfn0}\ref{IJinter_it}.
In either case,
$$\Th_{J,I}\big(q_{\cZ;I}\big(x,g_{i;I}(x,z)\!\big)\!\big)
=q_{\cZ;J}\big(x,g_{i;J}(x,z)\!\big)$$
by~\eref{cZidfn_e} and~\eref{whThiIJdfn_e}. 
Thus, the map~\eref{surjmap_e2} is well-defined.
By Lemma~\ref{cZtopol_lmm1}, $\cU_{i;I}^{\le}\!\subset\!\cU_i^{\le}$ is open.
The map~\eref{surjmap_e2} is continuous because its restriction to each of 
these open subspaces is continuous.\\

\noindent
For every $y\!\in\!\cZ_i^{\star}$, there are $I\!\in\!\cP_i(N)$ and
$$\wt{x}\equiv \big(x,(z_j)_{j\in I}\big)\in\wt\cZ_I^{\circ}
\qquad\hbox{s.t.}\quad
y=q_{\cZ}\big(q_{\cZ;I}(\wt{x})\!\big)\,.$$
Thus,
\BE{cZtopol_e3}
\big(\mu_I(x)\!\big)_{\!ij}+\frac12|z_i|^2
=\frac12|z_j|^2\ge0 \qquad\forall~j\!\in\!I\EE
and so $(x,z_i)\!\in\!\cU_{i;I}^{\le}$.
For each $j\!\in\!I\!-\!i$, let $\ne^{\fI\th_j}\!\in\!S^1$ be 
such that $\ne^{-\fI\th_j}\!z_j\!\in\!\R^{\ge0}$ and
$$\ne^{\fI\th_i}=\prod_{j\in I-i}\!\!\!\ne^{-\fI\th_j}\in S^1\,.$$
Then, $(\ne^{\fI\th_j})_{j\in I}$ is an element of $(S^1)^I_{\bu}$.
Since the $\phi_I$-action is Hamiltonian and $(x,z_i)\!\in\!\cU_{i;I}^{\le}$,
$$(x',z')\equiv
\big(\phi_I\big((\ne^{\fI\th_j})_{j\in I};x\big),\ne^{-\fI\th_i}z_i\big)\in \cU_{i;I}^{\le}
\quad\hbox{and}\quad
\wt\phi_I\big((\ne^{\fI\th_j})_{j\in I};\wt{x}\big)=
\big(x',g_{i;I}(x',z')\!\big).$$
Thus,
$$y=q_{\cZ}\big(q_{\cZ;I}(\wt{x})\!\big)
=q_{\cZ}\big(q_{\cZ;I}\big(\wt\phi_I((\ne^{\fI\th_j})_{j\in I};\wt{x})\!\big)\!\big)\\
=q_{\cZ}\big(q_{\cZ;I}\big(x',g_{i;I}(x',z')\!\big)\!\big)\,.$$
This establishes the surjectivity of the map~\eref{surjmap_e2}.\\

\noindent
For $I_0\!\in\!\cP_i(N)$, the restriction of the map~\eref{surjmap_e2}  
to the subspace 
$$\cU_{i,I_0}^{\le} =\bigcup_{I\in\cP_{I_0}(N)}\!\!\!\!\!\!\cU_{i,I_0;I}^{\le}
=\bigcup_{I\in\cP_{I_0}(N)}\!\!\!\!\!\!U_{I_0;I}^{\le}=U_{I_0}^{\le}$$
is the map~\eref{surjmap_e}.
Since the map~$g_{I_0;I}$ does not depend on the choice of $i\!\in\!I_0$,
neither does the map~\eref{surjmap_e}.
If $I\!\supset\!I_0$  and $y\!\in\!X_{I_0}$ are as in the previous paragraph, 
then $z_j\!=\!0$ for every $j\!\in\!I_0$.
By~\eref{cZtopol_e3}, this implies~that $\wt{x}\!\in\!U_{I_0}^{\le}$ and thus
$$x'\equiv\phi_I\big((\ne^{\fI\th_j})_{j\in I};x\big)\in U_{I_0}^{\le}
\quad\hbox{and}\quad
y=q_{\cZ}\!\big(q_{\cZ;I_0;I}\big(x',g_{I_0;I}(x')\big)\!\big).$$
This establishes the surjectivity of the map~\eref{surjmap_e}.\\

\noindent
For $i\!\in\!I\!\subset\![N]$ and $\cW\!\subset\!X\!\times\!\C$, let
$$\cW_{i;I} \equiv\big\{\wt\phi_I\big(g;x,g_{i;I}(x,z)\big)\!:
g\!\in\!(S^1)^I_{\bu},\,(x,z)\!\in\!\cU_{i;I}^{\le}\!\cap\!\cW\big\}\subset \wt\cZ_I^{\circ}.$$
By the second statement of Lemma~\ref{cZtopol_lmm1},
$$\cU_i^{\le}\cap\big(U_I\!\times\!\C\big)\cap\cW =  
\cU_{i;I}^{\le}\cap\cW\,.$$
Thus, the preimage of $q_i(\cU_i^{\le}\!\cap\!\cW)$ in~$\wt\cZ_I^{\circ}$  
under $q_{\cZ}\!\circ\!q_{\cZ;I}$ is~$\cW_{i;I}$.
Since the map~\eref{giIgraphdfn_e} is a homeomorphism and the group~$(S^1)^I_{\bu}$ is compact, 
$\cW_{i;I}$ is closed~in
$$\wt\phi_I\big((S^1)^I_{\bu}\!\times\!\wt\cZ_{I;i}^{\ge}\big)= \wt\cZ_I^{\circ}$$
if $\cW$ is closed in~$X\!\times\!\C$.
Thus, the map~\eref{surjmap_e2} is closed.
Since $q_i^{-1}(X_{I_0})\!=\!U_{I_0}^{\le}$ for $I_0\!\in\!\cP_i(N)$, 
the map~\eref{surjmap_e} is closed as~well.
\end{proof}

\begin{crl}\label{cZtopol_crl}
The topological spaces~$\cZ_i^{\star}$ with~$i\in\![N]$ and~$\cZ$ are Hausdorff.
They inherit symplectic forms~$\vp_i$ and~$\om_{\cZ}$ from the symplectic manifolds
$(\cZ_I^{\circ},\vp_I)$ with $i\!\in\!I\!\subset\![N]$. 
\end{crl}

\begin{proof}
By Lemma~\ref{cZtopol_lmm}, \eref{surjmap_e2} is a closed quotient map.
Since its domain is metrizable (being a subspace of a manifold),
it is normal.
By \cite[Lemma~73.3]{Mu}, $\cZ_i^{\star}$ is thus a normal (and in particular Hausdorff) 
topological space.\\

\noindent
Suppose $x,y\!\in\!\cZ$ are two distinct points.
Let $I,J\!\in\!\cP^*(N)$, $\wt{x}\!\in\!\cZ_I^{\circ}$, and 
$\wt{y}\!\in\!\cZ_J^{\circ}$ be such that $x\!=\!q_{\cZ}(\wt{x})$ and 
$y\!=\!q_{\cZ}(\wt{y})$.
If $I\!\not\subset\!J$ and $I\!\not\supset\!J$, then 
$$q_{\cZ}\big(\cZ_I^{\circ}\big)\cap q_{\cZ}\big(\cZ_J^{\circ}\big)=\eset$$
by~\eref{cZdfn_e} and Definition~\ref{SympCut_dfn0}\ref{IJinter_it}.
By Lemma~\ref{cZtopol_lmm3}, $q_{\cZ}(\cZ_I^{\circ})$ and $q_{\cZ}(\cZ_J^{\circ})$
are open in~$\cZ$.
If \hbox{$i\!\in\!I\!\subset\!J$}, then $x,y\!\in\!\cZ_i^{\star}$. 
Since $\cZ_i^{\star}$ is Hausdorff, there exist disjoint open neighborhoods~$U_x$ of~$x$ 
and~$U_y$ of~$y$ in~$\cZ_i^{\star}$.
By Lemma~\ref{cZtopol_lmm3}, $U_x$ and~$U_y$ are open in~$\cZ$.
Thus, $\cZ$ is Hausdorff.\\

\noindent
Since the identification maps~$\Th_{I,J}$ are diffeomorphisms between open subspaces 
of manifolds and the restriction of~\eref{cZproj_e} to each~$\cZ_I^{\circ}$
is a homeomorphism onto its image, 
the smooth structure on the domain of~\eref{cZproj_e}
descends to a smooth structure on~$\cZ$.
By~\eref{SympRed_e4} and~\eref{whThiIJprp_e0},
\BE{omcomp_e}
q_{\cZ;I}^*\Th_{J,I}^*\big(\vp_J|_{\cZ_{J,I}^{\circ}}\big) 
=\wt\Th_{J,I}^*q_{\cZ;J}^*\big(\vp_J|_{\cZ_{J,I}^{\circ}}\big) 
=\wt\Th_{J,I}^*\big(\wt\om_J|_{\wt\cZ_{J,I}^{\circ}}\big) 
=\wt\om_I\big|_{\wt\cZ_{I,J}^{\circ}}
=q_{\cZ;I}^*\big(\vp_I|_{\cZ_{I,J}^{\circ}}\big)\,.\EE
By the uniqueness part of \cite[Theorem~23.1]{daSilva}, \eref{omcomp_e} implies that
$$\Th_{J,I}^*\big(\vp_J|_{\cZ_{J,I}^{\circ}}\big) =\vp_I\big|_{\cZ_{I,J}^{\circ}}\,.$$
Thus, $q_{\cZ}$ induces a symplectic structure on~$\cZ$.
\end{proof}

\subsection{The SC symplectic divisor}
\label{SCsympdiv_subs}

\noindent
By Lemmas~\ref{cZtopol_lmm3} and~\ref{cZtopol_lmm} and 
Corollary~\ref{cZtopol_crl}, $(\cZ,\om_{\cZ})$ 
is a  symplectic manifold obtained by collapsing the boundary and corners
of the subspaces $\cU_i^{\le}\!\subset\!X\!\times\!\C$ to form 
open symplectic manifolds~$\cZ_i^{\star}$ with $i\!\in\![N]$ and
gluing the latter together along the common open subspaces~$\cZ_I^{\circ}$
with $I\!\in\!\cP_i(N)$.
We next describe the image of the boundary and corners of~$\cU_i^{\le}$
under the collapsing map as an SC symplectic divisor in $(\cZ,\om_{\cZ})$ in the sense of
Definition~\ref{SCD_dfn}.

\begin{lmm}\label{cZsub_lmm}
Let $I_0\!\subset\![N]$.
The topological spaces~$\cZ_{i,I_0}^{\star}$ with $i\!\in\![N]$ and~$\cZ_{I_0}$ 
are closed symplectic submanifolds of~$(\cZ_i^{\star},\vp_i)$ and of~$(\cZ,\om_{\cZ})$,
respectively, of codimension~$2|I_0|$.
If $I_0\!\neq\!\eset$ and $X$ is compact, then $X_{I_0}\!=\!\cZ_{I_0}$ is compact.
\end{lmm}

\begin{proof}
By~\eref{SympCutSubm_e2} and the surjectivity of the map~\eref{surjmap_e2}, 
$\cZ_{i,I_0}^{\star}\!\subset\!\cZ_i^{\star}$ is the image of the restriction 
of this map to~$\cU_{i,I_0}^{\le}$.
By Lemma~\ref{cZtopol_lmm2}, $\cU_{i,I_0}^{\le}\!\subset\!\cU_i^{\le}$ 
is a closed subspace.
Since~\eref{surjmap_e2} is a closed map, 
$\cZ_{i,I_0}^{\star}\!\subset\!\cZ_i^{\star}$ is thus a closed subspace.
By the third identity in~\eref{cZI0cZi_e}, the intersection of 
$\cZ_{I_0}$ with each~$\cZ_i^{\star}$ is thus closed in~$\cZ_i^{\star}$.
By Lemma~\ref{cZtopol_lmm3}, $\{\cZ_i^{\star}\}_{i\in[N]}$ is an open cover of~$\cZ$.
Therefore, $\cZ_{I_0}$ is closed in~$\cZ$.  
Since~$\cZ_i^{\star}$ and~$\cZ$ are Hausdorff by Corollary~\ref{cZtopol_crl}, 
so are $\cZ_{i,I_0}^{\star}\!\subset\cZ_i^{\star}$ and~$\cZ_{I_0}\!\subset\cZ$.\\

\noindent
By Lemma~\ref{SympNB_lmm0}, $(\cZ_{I_0;I}^{\circ},\vp_{I_0;I})$ is a symplectic submanifold of 
$(\cZ_I^{\circ},\vp_I)$ for every $I\!\in\!\cP_{I_0}(N)$.
Since the symplectic form~$\om_{\cZ}$ on~$\cZ$ is induced by 
the symplectic forms~$\vp_I$ on~$\cZ_I^{\circ}$,
the symplectic form~$\om_{I_0}$ on~$X_{I_0}\!=\!\cZ_{I_0}$ induced by
the symplectic forms~$\vp_{I_0;I}$ on~$\cZ_{I_0;I}^{\circ}$
is the restriction of~$\om_{\cZ}$.\\

\noindent
Suppose $I_0\!\neq\!\eset$ and $X$ is compact. 
By Lemma~\ref{Xitopol_lmm2}, $U_{I_0}^{\le}\!\subset\!X$ is then also compact. 
Since the map~\eref{surjmap_e} is continuous and surjective, 
it follows that $X_{I_0}\!=\!\cZ_{I_0}$ is compact as~well.
\end{proof}

\noindent
As in the proof of Lemma~\ref{cZsub_lmm}, 
we denote by~$\om_{I_0}$ the symplectic form on~$\cZ_{I_0}$
induced by the symplectic forms~$\vp_{I_0;I}$ on~$\cZ_{I_0;I}^{\circ}$.
In particular,
$$\big(X_i,\om_i\big) \equiv \big(\cZ_{\{i\}},\om_{\{i\}}\big)
=\big(\cZ_{\{i\}},\om_{\cZ}|_{\cZ_{\{i\}}}\big)$$
is a symplectic manifold for each $i\!\in\![N]$.\\

\noindent
Let $i\!\in\!I\!\subset\![N]$  and $I_0'\!\subset\!I_0\subset\!I$.
With notation as in~\eref{wtcZcNo_e} and~\eref{whThiIJdfn_e0}, we define 
$$\wt\Th_{J,I;i}\!: \wt\cN_{I_0;I_0';I}^{\circ}
\big|_{\wt\cZ_{I_0;I}^{\circ}\cap\wt\cZ_{I,J}^{\circ}} 
\lra \wt\cN_{I_0;I_0';J}^{\circ}\big|_{\wt\cZ_{I_0;J}^{\circ}\cap\wt\cZ_{J,I}^{\circ}}$$
by the formula in~\eref{whThiIJdfn_e} with $|z_i|^2$ replaced by~0 if $i\!\in\!I_0\!-\!I_0'$.
This diffeomorphism is still 
$(\wt\phi_J,\wt\phi_I)$-equivariant with respect to the homomorphism~\eref{S1IJhomom_e},
satisfies the analogues of~\eref{whThiIJprp_e0}-\eref{whThiIJprp_e3} and 
$$  \wt\pi_{I_0;I_0';J}^{\circ} \circ\wt\Th_{J,I;i}
=\wt\Th_{J,I;i}\circ\wt\pi_{I_0;I_0';I}^{\circ}
\big|_{\wt\cZ_{I_0;I}^{\circ}\cap\wt\cZ_{I,J}^{\circ}}\,,$$
and is complex linear on the fibers of $\wt\pi_{I_0;I_0';I}^{\circ}$.
Thus, it descends to a diffeomorphism
$$\Th_{J,I}\!: 
\cN_{I_0;I_0';I}^{\circ}\big|_{\cZ_{I_0;I}^{\circ}\cap \cZ_{I,J}^{\circ}} 
\lra \cN_{I_0;I_0';J}^{\circ}\big|_{\cZ_{I_0;J}^{\circ}\cap \cZ_{J,I}^{\circ}}$$
which is independent of $i\!\in\!I$, satisfies the analogue of~\eref{ThCoc_e} and
\BE{piI0I0pr_e0}  \pi_{I_0;I_0';J}^{\circ} \circ\Th_{J,I}
=\Th_{J,I}\circ\pi_{I_0;I_0';I}^{\circ}\big|_{\cZ_{I_0;I}^{\circ}\cap \cZ_{I,J}^{\circ}}\,,\EE
and is complex linear on the fibers of~$\pi_{I_0;I_0';I}^{\circ}$.
Define
\begin{equation*}\begin{split}
&\cN_{I_0;I_0'}=
\bigg(\bigsqcup_{I\in\cP_{I_0}(N)}
\!\!\!\!\!\!
\cN_{I_0;I_0';I}^{\circ}\bigg)\Big/\!\!\sim,\qquad\hbox{where}\\
&\cN_{I_0;I_0';I}^{\circ}\big|_{\cZ_{I_0;I}^{\circ}\cap \cZ_{I,J}^{\circ}}
\ni \big[\wt{x},z\big]\sim \Th_{J,I}\big([\wt{x},z]\big)
\in \cN_{I_0;I_0';J}^{\circ}\big|_{\cZ_{I_0;J}^{\circ}\cap \cZ_{J,I}^{\circ}}.
\end{split}\end{equation*}
By~\eref{piI0I0pr_e0}, 
the projections~$\pi_{I_0;I_0';I}^{\circ}$ induce a complex vector bundle
\BE{cZcN_e}\pi_{I_0;I_0'}\!:  \cN_{I_0;I_0'}\lra \cZ_{I_0}\,.\EE

\begin{lmm}\label{SympNB_lmm2}
For all \hbox{$I_0'\!\subset\!I_0\!\subset\!I$},
$(\cZ_{I_0},\om_{I_0})$ is a symplectic submanifold of 
$(\cZ_{I_0'},\om_{I_0'})$  with the oriented normal bundle 
canonically isomorphic to~\eref{cZcN_e}.
\end{lmm}

\begin{proof}
By Lemma~\ref{cZsub_lmm}, 
$(\cZ_{I_0},\om_{I_0})$ and $(\cZ_{I_0'},\om_{I_0'})$
are symplectic submanifolds of~$(\cZ,\om_{\cZ})$.
Since $\cZ_{I_0}\!\subset\!\cZ_{I_0'}$, 
$(\cZ_{I_0},\om_{I_0})$  is a  symplectic submanifold of $(\cZ_{I_0'},\om_{I_0'})$.\\

\noindent
Since the collection of all isomorphisms~\eref{cNisom_e2} respects 
the patching maps~$\Th_{J,I}$
for the manifold~$\cZ_{I_0'}$ and for the bundle~\eref{cZcN_e},
the isomorphisms~\eref{cNisom_e2} induce a vector bundle isomorphism 
$$\cN_{\cZ_{I_0'}}\cZ_{I_0}\equiv 
\frac{T\cZ_{I_0'}|_{\cZ_{I_0}}}{T\cZ_{I_0}}\lra  \cN_{I_0;I_0'}$$
of oriented vector bundles over~$\cZ_{I_0}$.
This isomorphism is orientation-preserving because 
the isomorphisms~\eref{cNisom_e2} and the identifications~$\Th_{J,I}$
are.
\end{proof}

\begin{crl}\label{SympNB_crl2}
The collection $\{X_i\}_{i\in[N]}$ is an SC symplectic divisor in $(\cZ,\om_{\cZ})$
with the associated $N$-fold symplectic configuration 
$\X\!\equiv\!((X_I)_{I\in\cP^*(N)},(\om_i)_{i\in[N]})$.
Furthermore, the map 
\BE{qesetdfn_e}q_{\eset}\!:X\lra X_{\eset}\!\subset\!\cZ, \qquad
q_{\eset}(x)=q_{\{i\}}(x)\quad\forall\,x\!\in\!U_i^{\le},\,i\in\![N],\EE
with $q_{\{i\}}$ as in~\eref{surjmap_e}, 
is well-defined, continuous, and surjective.
For each $i\!\in\![N]$, it
takes $\prt U_i^{\le}$ onto $X_i\!\cap\!X_{\prt}$ and restricts to a symplectomorphism 
\BE{qesetdfn_e2}\big(U_i^{\le}\!-\!\prt U_i^{\le},\om|_{U_i^{\le}-\prt U_i^{\le}}\big)
\lra \big(X_i\!-\!X_{\prt},\om_i|_{X_i\!-\!X_{\prt}}\big).\EE
\end{crl}

\begin{proof}
By Lemma~\ref{cZsub_lmm}, $(X_i,\om_i)$ is a closed symplectic submanifold
of $(\cZ,\om_{\cZ})$ for every $i\!\in\![N]$.
By construction,
$$X_I=\cZ_I=\bigcap_{i\in I}\!X_i\subset\cZ \qquad\forall~I\!\subset\![N].$$
By Lemma~\ref{SympNB_lmm2}, the real codimension of~$X_i$ in~$\cZ$ is~2 
for every $i\!\in\![N]$ and the homomorphism
\BE{SympNB_e2}
\cN_{\cZ}X_I\equiv \frac{T\cZ|_{X_I}}{TX_I}
\lra \bigoplus_{i\in I}\! \frac{T\cZ|_{X_I}}{TX_i|_{X_I}}
\equiv \bigoplus_{i\in I} \cN_{\cZ}X_i\big|_{X_I}\EE
induced by the natural inclusions 
is an orientation-preserving isomorphism for all $I\!\subset\![N]$.
Thus, $\{X_i\}_{i\in[N]}$ is a transverse (in fact orthogonal) collection of 
closed symplectic submanifolds of~$\cZ$ 
so that the symplectic orientation of each~$X_I$ (which orients 
the left-hand side of~\eref{SympNB_e2}) agrees with its intersection orientation
(which is determined by the orientation of the right-hand side of~\eref{SympNB_e2}).
Therefore, $\{X_i\}_{i\in[N]}$ is an SC symplectic divisor in~$(\cZ,\om_{\cZ})$.\\

\noindent
By Lemmas~\ref{sCcover_lmm} and~\ref{Xitopol_lmm1}, 
the closed subsets $U_i^{\le}\!\subset\!X$ cover~$X$.
By Lemmas~\ref{cZtopol_lmm} and~\ref{Xitopol_lmm2},  
$q_{\{i\}}$ is continuous on $U_i^{\le}$ and $q_{\{i\}}\!=\!q_{\{j\}}$ on 
$U_i^{\le}\!\cap\!U_j^{\le}\!=\!U_{ij}^{\le}$.
Thus, \eref{qesetdfn_e} is a well-defined continuous map.
By Lemma~\ref{cZtopol_lmm}, it 
takes $U_i^{\le}$ onto~$X_i$ for every $i\!\in\![N]$ and 
$U_{I_0}^{\le}$ onto~$X_{I_0}$ for every $I\!\in\!\cP^*(N)$.
Thus, $q_{\eset}$ restricts  to surjective continuous maps
$$
q_{\eset}\!:\prt U_i^{\le}\!\equiv\!
\bigcup_{\begin{subarray}{c}I_0\in\cP_i(N)\\ |I_0|\ge2\end{subarray}}\!\!\!\!\!\!U_{I_0}^{\le} 
\lra \bigcup_{\begin{subarray}{c}I_0\in\cP_i(N)\\ |I_0|\ge2\end{subarray}}\!\!\!\!\!\!X_{I_0}
=X_i\!\cap\!X_{\prt}, \quad
q_{\eset}\!:U_i^{\le}\!-\!\prt U_i^{\le}\lra X_i\!-\!X_{\prt}\,.$$
The latter map is described by its restrictions to the open subsets $U_{i;I}^<$
with $I\!\in\!\cP_i(N)$:
\begin{gather*}
q_{\eset}\!:U_{i;I}^<\!\equiv\!\big\{x\!\in\!U_I\!:
(\mu_I(x))_i\!<\!(\mu_I(x))_j~\forall\,j\!\in\!I\!-\!i\big\}, \quad
q_{\eset}(x)=q_{\cZ}\big(q_{\cZ;i;I}\big(x,g_{i;I}(x)\big)\!\big),\\
\hbox{where}\qquad 
g_{i;I}(x)=\Big(\sqrt{2(\mu_I(x))_{ij}}\,\Big)_{j\in I-i}\in (\R^+)^{I-i}\,.
\end{gather*}
Since the restriction of~\eref{giIgraphdfn_e} to~$U_{i;I}^<$ 
is a homeomorphism~onto 
$$\wt{X}_{i;I}^> \equiv \wt\cZ_{\{i\};I}^{\circ}\!\cap\! \wt\cZ_{I;i}^{\ge}-
\bigcup_{\begin{subarray}{c} i\in I_0\subset I\\ |I_0|\ge2\end{subarray}}\!\!\!\!
\wt\cZ_{I_0;I}^{\circ}
=\wt\cZ_{\{i\};I}^{\circ}\!\cap\!\wt\cZ_{I;i}^{\ge}- q_{\cZ;i;I}^{-1}\big(q_{\cZ}^{-1}(X_{\prt})\big)
\subset U_{i;I}^<\!\times\!(\R^+)^{I-i}$$
and $\wt{X}_{i;I}^>$ is a slice for the $\wt\phi_I$-action on 
$\wt\cZ_{\{i\};I}^{\circ}\!-\!q_{\cZ;i;I}^{-1}(q_{\cZ}^{-1}(X_{\prt}))$,
the restriction of~$q_{\eset}$ to~$U_{i;I}^<$ is injective.
Since $U_{i;I}^<$ is $\phi_I$-invariant, \eref{cZI0dfn_e} then implies that 
the map~\eref{qesetdfn_e2} is injective.
Since  restriction of~$q_{\eset}$ to~$U_{i;I}^<$ is a composition of smooth maps,
the map~\eref{qesetdfn_e2} is smooth.
Since the argument of every component of the~map
$$g_{i;I}\!:U_{i;I}^<\lra\C^{I-i}$$ 
is fixed, \eref{SympRed_e5} implies~that
\begin{equation*}\begin{split}
\big\{q_{\eset}|_{U_{i;I}^<}\big\}^*\om_i
=\big\{\id\!\times\!g_{i;I}\big\}^{\!*}
q_{\cZ;i;I}^{\,*}\vp_{i;I}
&=\big\{\id\!\times\!g_{i;I}\big\}^{\!*}\pi_1^*\om
+\big\{\id\!\times\!g_{i;I}\big\}^{\!*}\pi_2^*\om_{\C^{I-i}}\\
&=\om|_{U_{i;I}^<}+g_{i;I}^{\!*}\om_{\C^{I-i}}=\om|_{U_{i;I}^<}\,.
\end{split}\end{equation*}
Thus, the map~\eref{qesetdfn_e2} is a symplectomorphism.
\end{proof}

\subsection{The Hamiltonian configuration}
\label{cZSympCutConf_subs}

\noindent
We next describe an $N$-fold Hamiltonian configuration
\BE{cZSympCutDfn_e}
\sC_{\cZ}\equiv\big(U_{\cZ;I},\phi_{\cZ;I},\mu_{\cZ;I}\big)_{I\in\cP^*(N)}\EE
for~$(\cZ,\om_{\cZ})$.
For each $I\!\in\!\cP^*(N)$, let $U_{\cZ;I}\!=\!q_{\cZ}(\cZ_I^{\circ})$.
By Lemma~\ref{cZtopol_lmm3}, $U_{\cZ;I}$ is open in~$\cZ$.
By~\eref{cZdfn_e}, $(U_{\cZ;I})_{I\in\cP^*(N)}$ is a cover of~$\cZ$.
Since $U_I\!\cap\!U_J\!=\!\eset$ unless $I\!\subset\!J$ or $I\!\subset\!J$,
\eref{cZdfn_e} implies that 
the same is the case for~$U_{\cZ;I}$ and~$U_{\cZ;J}$.\\

\noindent
For $I\!\in\!\cP^*(N)$, we define a Hamiltonian $(S^1)^I_{\bu}$- pair for 
$(U_I\!\times\!\C^I,\wt\om_I\big)$~by
\begin{alignat}{2}
\label{wtcZphiiIdfn_e}
\wt\phi_{\cZ;I}\!: 
(S^1)^I_{\bu}\!\times\!\big(U_I\!\times\!\C^I\big)&\lra U_I\!\times\!\C^I,
&\qquad
\wt\phi_{\cZ;I}\big(g;x,z\big)&=\big(\phi_I(g;x),z\big);\\
\label{wtcZhiIdfn_e}
\wt\mu_{\cZ;I}\!: U_I\!\times\!\C^I&\lra\ft_{I;\bu}^*, &\qquad
\wt\mu_{\cZ;I}(x,z)&=\mu_I(x).
\end{alignat}
By~\ref{phiIUJ_it} and~\ref{UIJpos_it} in Definition~\ref{SympCut_dfn0},
\BE{cZphiIUJ_e12}\begin{aligned}
\wt\mu_{\cZ;J}(x,z)|_{\ft_{I;\bu}}&=\wt\mu_{\cZ;I}(x,z)
&\quad&\forall~(x,z)\!\in\!(U_I\!\cap\!U_J)\!\times\!\C^I,\,I\!\subset\!J\!\subset\![N],\\
\big(\wt\mu_{\cZ;J}(x)\big)_i&<\big(\wt\mu_{\cZ;J}(x)\big)_j
&\quad&\forall\,
(x,z)\!\in\!(U_I\!\cap\!U_J)\!\times\!\C^I,\,i\!\in\!I\!\subset\!J\!\subset\![N],
\,j\!\in\!J\!-\!I,
\end{aligned}\EE
respectively.\\

\noindent
Since the $\phi_I$-action is Hamiltonian, \eref{wtcZphiiIdfn_e} 
restricts to an $(S^1)^I_{\bu}$-action 
\BE{Ham_e3a}
\wt\phi_{\cZ;I}\!: (S^1)^I_{\bu}\!\times\!\wt\cZ_I^{\circ} \lra \wt\cZ_I^{\circ}\,.\EE
Since the $\wt\phi_I$-action commutes with this action and preserves
the moment map~$\wt\mu_{\cZ;I}$,
\eref{Ham_e3a} and~\eref{wtcZhiIdfn_e} determine a Hamiltonian pair
\BE{Ham_e3b}
\phi_{\cZ;I}^{\circ}\!: (S^1)^I_{\bu}\!\times\!\cZ_I^{\circ} \lra\cZ_I^{\circ}
\qquad\hbox{and}\qquad
\mu_{\cZ;I}^{\circ}\!: \cZ_I^{\circ} \lra\ft_{I;\bu}^*\EE
for $(\cZ_I^{\circ},\vp_I)$.
Since the restriction of~\eref{cZproj_e} to~$\cZ_I^{\circ}$ is a symplectomorphism 
onto~$U_{\cZ;I}$, $\phi_{\cZ;I}^{\circ}$ and~$\mu_{\cZ;I}^{\circ}$ induce smooth~maps
\BE{Ham_e3d}
\phi_{\cZ;I}\!: (S^1)^I_{\bu}\!\times\!U_{\cZ;I}\lra U_{\cZ;I}, \qquad
\mu_{\cZ;I}\!: U_{\cZ;I} \lra\ft_{I;\bu}^*,\EE
so that $\phi_{\cZ;I}$  is a Hamiltonian $(S^1)^I_{\bu}$-action 
with moment map $\mu_{\cZ;I}$.

\begin{lmm}\label{cZHamConf_lmm}
The tuple~\eref{cZSympCutDfn_e}  described above is an $N$-fold Hamiltonian configuration
for~$(\cZ,\om_{\cZ})$.
\end{lmm}

\begin{proof}
As explained above, $(U_{\cZ;I})_{I\in\cP^*(N)}$ is an open cover of~$\cZ$
satisfying Definition~\ref{SympCut_dfn0}\ref{IJinter_it}.
By~\eref{cZphiIUJ_e12}, 
the moment maps~$\mu_{\cZ;I}$ satisfy~\ref{phiIUJ_it}
and~\ref{UIJpos_it} in Definition~\ref{SympCut_dfn0}.
\end{proof}

\noindent
For $I\!\in\!\cP^*(N)$, we define
$$\wt\cZ_{\prt;I}^{\circ}
=\bigcup_{\begin{subarray}{c}I_0\subset I\\ |I_0|=2\end{subarray}}
\!\!\!\wt\cZ_{I_0;I}^{\circ}\subset \wt\cZ_I^{\circ}\subset U_I\!\times\!\C^I\,.$$
If in addition $I_0\!\subset\!I$, let 
\BE{wthcZiII0Idfn_e}\begin{split}
\wt\mu_{\cZ;I_0;I}\!: \big\{\wt{x}\!\in\!\wt\cZ_I^{\circ}\!\!:
\big(\wt\mu_{\cZ;I}(\wt{x})\!\big)_{\!i}\!<\!\big(\wt\mu_{\cZ;I}(\wt{x})\!\big)_{\!j}
~&\forall\,i\!\in\!I_0,\,j\!\in\!I\!-\!I_0\big\}\lra \ft_{I_0;\bu}^*,\\
\wt\mu_{\cZ;I_0;I}(\wt{x})=&\wt\mu_{\cZ;I}(\wt{x})\big|_{\ft_{I_0;\bu}}\,.
\end{split}\EE
The commuting torus actions~\eref{wtphiIdfn_e} and~\eref{wtcZphiiIdfn_e}
induce an  $(S^1)^{I_0}_{\bu}\!\times\!(S^1)^I_{\bu}$-action on~$U_I\!\times\!\C^I$:
\BE{CombAct_e} 
\big(g_0,g\big)\cdot \wt{x}
=\wt\phi_{\cZ;I}\big(g_0;\wt\phi_I(g;\wt{x})\big).\EE
Since the $\phi_I$-action is Hamiltonian, 
the action~\eref{CombAct_e} preserves $\wt\cZ_{\prt;I}^{\circ}$ and $\wt\mu_{\cZ;I_0;I}^{\,-1}(0)$.

\begin{lmm}\label{wtcZiIreg_lmm2}
For all $I_0\!\in\!\cP^*(N)$ and $I\!\in\!\cP_{I_0}(N)$,
 the restriction of the torus action~\eref{CombAct_e} 
to $\wt\mu_{\cZ;I_0;I}^{\,-1}(0)\!-\!\wt\cZ_{\prt;I}^{\circ}$
is~free.
\end{lmm}

\begin{proof} 
Let $\wt{x}\!=\!(x,(z_j)_{j\in I})$ be an element of
 $\wt\mu_{\cZ;I_0;I}^{\,-1}(0)\!-\!\wt\cZ_{\prt;I}^{\circ}$.
If the action~\eref{CombAct_e} by some $(g_0,g)$ fixes~$\wt{x}$,
then $g\!=\!\id$ because $z_j\!=\!0$ for at most one element $j\!\in\!I$.
By~\eref{wthcZiII0Idfn_e}, \eref{wtcZhiIdfn_e}, and~\eref{hdfn_e},  
$$\wt\mu_{\cZ;I_0;I}(\wt{x})=\mu_{I_0;I}(x)\,.$$
Thus, $g_0\!=\!\id$ by Definition~\ref{SympCut_dfn1}.
This establishes the second claim.
\end{proof}

\begin{crl}\label{cZHamConf_crl}
The restriction of the $N$-fold Hamiltonian configuration~\eref{cZSympCutDfn_e} 
to $\cZ\!-\!X_{\prt}$
is an $N$-fold cutting configuration.
\end{crl}

\begin{proof}
In light of Lemma~\ref{cZHamConf_lmm},
it remains to show that the restriction satisfies the  conditions of 
Definition~\ref{SympCut_dfn1}.
Fix $I_0\!\in\!\cP^*(N)$ and $I\!\in\!\cP_{I_0}(N)$.
We denote by~$\mu_{\cZ;I_0;I}$ the analogue of the map~\eref{hdfn_e}
for the Hamiltonian configuration~\eref{cZSympCutDfn_e}
and by~$\phi_{\cZ;I_0;I}$ the restriction of the $(S^1)^I_{\bu}$-action~$\phi_{\cZ;I}$
to~$(S^1)^{I_0}_{\bu}$.
Let 
\begin{gather*}
\cZ_{\prt;I}^{\circ}= 
q_{\cZ;I}\big(\wt\cZ_{\prt;I}^{\circ}\big)\subset \cZ_I^{\circ}, \qquad
\phi_{\cZ;I_0;I}^{\circ}\!=\!
\phi_{\cZ;I}^{\circ}\big|_{(S^1)^{I_0}_{\bu}\times\cZ_I^{\circ}}\!:
(S^1)^{I_0}_{\bu}\!\times\!\cZ_I^{\circ}\lra \cZ_I^{\circ},\\
\mu_{\cZ;I_0;I}^{\circ}\!: \big\{x\!\in\!\cZ_I^{\circ}\!\!:
\big(\mu_{\cZ;I}^{\circ}(x)\!\big)_{\!i}\!<\!\big(\mu_{\cZ;I}^{\circ}(x)\!\big)_{\!j}
\,\forall\,i\!\in\!I_0,\,j\!\in\!I\!-\!I_0\big\}\lra \ft_{I_0;\bu}^*,~~
\mu_{\cZ;I_0;I}^{\circ}(x)=\mu_{\cZ;I}^{\circ}(x)\big|_{\ft_{I_0;\bu}}.
\end{gather*}
Thus,
$$q_{\cZ}\!\circ\!\phi_{\cZ;I_0;I}^{\circ}=
\phi_{\cZ;I_0;I}\!\circ\!\big\{\id\!\times\!q_{\cZ}\big\}
\qquad\hbox{and}\qquad
\mu_{\cZ;I_0;I}^{\circ}=\mu_{\cZ;I_0;I}\!\circ\!q_{\cZ}\,.$$
Since 
$$q_{\cZ}\!:\Dom\big(\mu_{\cZ;I_0;I}^{\circ}\big)\!-\!\cZ_{\prt;I}^{\circ}\lra
\Dom\big(\mu_{\cZ;I_0;I}\big)\!-\!X_{\prt}$$
is a diffeomorphism, it is sufficient to show~that
the restriction of the $\phi_{\cZ;I_0;I}^{\circ}$-action to 
$\mu_{\cZ;I_0;I}^{\circ~-1}(0)$ is free.
Since $\cZ_I^{\circ}$ is the quotient of~$\wt\cZ_I^{\circ}$ by the $\wt\phi_I$-action
and
$$q_{\cZ}\!\circ\!\wt\phi_{\cZ;I}\big|_{(S^1)^{I_0}_{\bu}\times\wt\cZ_I^{\circ}}
=\phi_{\cZ;I_0;I}^{\circ}\!\circ\!\big\{\id\!\times\!q_{\cZ;I}\big\},$$
this follows from Lemma~\ref{wtcZiIreg_lmm2}.
\end{proof}

\section{Proof of Theorem~\ref{SympCut_thm3}}
\label{SumCutSm_sec}

\noindent
By Corollary~\ref{cZtopol_crl},
an $N$-fold cutting configuration~$\sC$ determines 
an SC symplectic configuration~$\X(\sC)$ and a symplectic  manifold 
$(\cZ,\om_{\cZ})$ containing the SC symplectic variety~$X_{\eset}$ associated
with~$\X(\sC)$ as an SC symplectic divisor.
We show in Section~\ref{SympNBpf_subs} that~$\sC$ also 
gives rise to a one-parameter family \hbox{$\pi\!:\cZ'\!\lra\!\C$} of smoothings of~$X_{\eset}$ 
so that~$\cZ'$ is a neighborhood of~$X_{\eset}$ in~$\cZ$.
This family of smoothings satisfies the last two claims of Theorem~\ref{SympCut_thm3}.
Unlike the constructions of Sections~\ref{SympRed_subs}-\ref{cZSympCutConf_subs},
the construction of Sections~\ref{SympNBpf_subs} involves choices.
However, these choices are deformation equivalent and
the resulting one-parameter family of smoothings is well-defined up to deformations.
For a family of $N$-fold cutting configurations,
the choices involved in the construction of Sections~\ref{SympNBpf_subs} can be made systemically
on sufficiently small neighborhoods of all cutting configurations
and thus result in a continuously varying family of deformation equivalence classes of
one-parameter families of smoothings.
This establishes Theorem~\ref{SympCut_thm3}.

\subsection{Geometric preliminaries}
\label{HomotTriv_subs}

\noindent
We begin with a lemma which enables us to refine open covers as in Definition~\ref{SympCut_dfn0}.
We then establish several local statements that are patched together in Section~\ref{SympNBpf_subs}
to construct a one-parameter family of smoothings of~$X_{\eset}$.

\begin{lmm}\label{sCref_lmm}
Let $\sC$ be an $N$-fold Hamiltonian configuration for $(X,\om)$ 
as in~\eref{SympCutDfn_e}.
There exists a $\sC$-invariant open cover $(U_I')_{I\in\cP^*(N)}$ of~$X$
properly refining~$(U_I)_{I\in\cP^*(N)}$.
For every open cover $(U_I')_{I\in\cP^*(N)}$ of~$X$
properly refining~$(U_I)_{I\in\cP^*(N)}$, there exists
a $\sC$-invariant open cover $(U_I'')_{I\in\cP^*(N)}$  of~$X$
properly refining~$(U_I)_{I\in\cP^*(N)}$ and 
properly refined by~$(U_I')_{I\in\cP^*(N)}$.
\end{lmm}

\begin{proof}
We modify {\it Step~1} in the proof of \cite[Theorem~36.1]{Mu} to take into account 
the torus actions using the following observation.
Suppose $I\!\in\!\cP^*(N)$,
$W\!\subset\!U_I$ is an open subset such that 
\hbox{$\phi_I((S^1)^I_{\bu}\!\times\!W)\!=\!W$}, 
and $A\!\subset\!X$ is a closed subset such that $A\!\subset\!W$.
Since $X$ is normal, there exists an open subset $W'\!\subset\!X$ such~that 
$A\!\subset\!W'$ and the closure~$\ov{W'}$ of~$W'$ in~$X$ is contained in~$W$.
Since the $\phi_I$-action is continuous,
the subspace  
$$W''\equiv \phi_I\big((S^1)^I_{\bu}\!\times\!W'\big)\subset U_I$$
is open in~$U_I$ and thus in~$X$.
Since the group $(S^1)^I_{\bu}$ is compact and \hbox{$\phi_I((S^1)^I_{\bu}\!\times\!W)\!=\!W$}, 
$$\ov{W''}=\phi_I\big((S^1)^I_{\bu}\!\times\!\ov{W'}\big)\subset W\,,$$
where the closure is taken in~$X$.
Thus, $W''\!\subset\!X$ is an open subset such that 
\BE{sCref_e3}A\subset W'', \qquad  \ov{W''}\subset W, \quad\hbox{and}\quad 
\phi_I\big((S^1)^I_{\bu}\!\times\!W''\big)= W''.\EE
For the remainder of this proof, we fix a total order~$<$ on the subsets $I\!\subset\![N]$ 
so~that  $I\!<\!I^*$ whenever $|I|\!>\!|I^*|$.\\

\noindent
Suppose $I^*\!\in\!\cP^*(N)$ and $(U_I')_{I\in\cP^*(N)}$ is
a $\sC$-invariant open cover of~$X$ refining~$(U_I)_{I\in\cP^*(N)}$ such~that 
\BE{sCref_e5}  \ov{U_I'}\subset U_I \qquad\forall~I\!<\!I^*.\EE
Since $(U_I')_{I\in\cP^*(N)}$ is an open cover of~$X$, the closed subset 
$$A\equiv X - \bigcup_{\begin{subarray}{c}I\in\cP^*(N)\\ I\neq I^*\end{subarray}}
\!\!\!\!\!\!U_I'$$
is contained in $U_{I^*}'\!\subset\!U_{I^*}$. 
By the previous paragraph with $I\!=\!I^*$ and $W\!=\!U_{I^*}$, there exists 
an open subset $W''\!\subset\!X$ satisfying~\eref{sCref_e3}.
Replacing the open subset~$U_{I^*}'$ with~$W''$, we obtain 
a $\sC$-invariant open cover $(U_I')_{I\in\cP^*(N)}$ of~$X$ refining~$(U_I)_{I\in\cP^*(N)}$ 
such~that the inclusion in~\eref{sCref_e5} holds for all $I\!\le\!I^*$.
Continuing in this way, we obtain a $\sC$-invariant open cover $(U_I')_{I\in\cP^*(N)}$ of~$X$
properly refining~$(U_I)_{I\in\cP^*(N)}$.\\

\noindent
We next establish the second claim.
Suppose $I^*\!\in\!\cP^*(N)$ and $(U_I'')_{I\in\cP^*(N)}$ is
a $\sC$-invariant open cover of~$X$ refining~$(U_I)_{I\in\cP^*(N)}$ 
and properly refined by~$(U_I')_{I\in\cP^*(N)}$ such~that 
\BE{sCref_e7}  \ov{U_I''}\subset U_I \qquad\forall~I\!<\!I^*.\EE
By the observation in the first paragraph applied with $I\!=\!I^*$,
$A\!=\!\ov{U_{I^*}'}$, and $W\!=\!U_{I^*}$, there exists 
an open subset $W''\!\subset\!X$ satisfying~\eref{sCref_e3}.
Replacing the open subset~$U_{I^*}''$ with~$W''$, we obtain 
a $\sC$-invariant open cover $(U_I'')_{I\in\cP^*(N)}$ of~$X$ refining~$(U_I)_{I\in\cP^*(N)}$
and properly refined by~$(U_I')_{I\in\cP^*(N)}$
such~that the inclusion in~\eref{sCref_e7} holds for all $I\!\le\!I^*$.
Continuing in this way, we obtain a $\sC$-invariant open cover  
$(U_I'')_{I\in\cP^*(N)}$  of~$X$
properly refining~$(U_I)_{I\in\cP^*(N)}$ and 
properly refined by~$(U_I')_{I\in\cP^*(N)}$.
\end{proof}

\noindent
For a $\sC$-invariant open cover $(U_I')_{I\in\cP^*(N)}$  of~$X$
properly refining~$(U_I)_{I\in\cP^*(N)}$,  we denote~by
$$\wt\cZ_I'^{\circ}\equiv \wt\cZ_I^{\circ}\cap\big(U_I'\!\times\!\C^I\big),\quad
\cZ_I'^{\circ}\subset \cZ_I^{\circ}, \quad
\wt\cZ_{I,J}'^{\circ}\subset \wt\cZ_I'^{\circ}\!\cap\!\wt\cZ_{I,J}^{\circ} 
\subset \wt\cZ_I^{\circ},\quad
\cZ_{I,J}'^{\circ}\subset \cZ_I'^{\circ}\!\cap\!\cZ_{I,J}^{\circ} 
\subset \cZ_I^{\circ}\,,$$
the spaces as in~\eref{wtcZIdfn_e}, 
\eref{wtcZiIdfn_e}, \eref{whThiIJdfn_e0}, and~\eref{cZiIJdfn_e} 
corresponding to 
the restriction~$\sC'$ of~$\sC$ to  $(U_I')_{I\in\cP^*(N)}$ defined by~\eref{sCprdfn_e}.
For each $I\!\in\!\cP^*(N)$, the torus action~$\phi_{\cZ;I}^{\circ}$
in~\eref{Ham_e3b} preserves the subspaces $\cZ_I'^{\circ}$ and~$\cZ_{I,J}'^{\circ}$
of~$\cZ_I^{\circ}$.

\begin{lmm}\label{fIcoll_lmm}
Let $(U_I')_{I\in\cP^*(N)}$ be a $\sC$-invariant open cover of~$X$
properly refining~$(U_I)_{I\in\cP^*(N)}$.
There exists a tuple $(f_I\!:\cZ_I'^{\circ}\!\lra\!\R^+)_{I\in\cP^*(N)}$
of smooth functions such~that each~$f_I$ is $\phi_{\cZ;I}^{\circ}$-invariant, 
\BE{fIcoll_e}\begin{split} 
f_I\big([x,(z_j)_{j\in I}]\big)&=
f_J\big(\Th_{J,I}\big([x,(z_j)_{j\in I}]\big)\!\big) \!\!\!\!
\prod_{j\in J-I}\!\!\!\!\!\sqrt{2(\mu_J(x))_{ij}\!+\!|z_i|^2}\\
&\hspace{1.5in}~\forall\,
\big[x,(z_j)_{j\in I}\big]\!\in\!\cZ_{I,J}'^{\circ},\,i\!\in\!I\!\subset\!J\!\subset\![N],
\end{split}\EE
with $(\mu_J(x))_{ij}$ as in~\eref{hijdfn_e}.
\end{lmm}

\begin{proof} 
Choose a total order~$<$ on the subsets $I\!\subset\![N]$ 
so~that  $I\!<\!I^*$ whenever $|I|\!>\!|I^*|$.
Suppose $I^*\!\in\cP^*(N)$ and we have constructed
\begin{enumerate}[label=$\bu$,leftmargin=*]

\item a $\sC$-invariant open cover $(U_I'')_{I\in\cP^*(N)}$
 of~$X$ refining~$(U_I)_{I\in\cP^*(N)}$ and properly refined by 
the open cover $(U_I')_{I\in\cP^*(N)}$,

\item smooth functions $f_I\!:\cZ_I''^{\circ}\!\lra\!\R^+$
for all  $I\!<\!I^*$ such that
each $f_I$ is $\phi_{\cZ;I}^{\circ}$-invariant
and the equality in~\eref{fIcoll_e}
holds for all $[x,(z_j)_{j\in I}]\!\in\!\cZ_{I,J}''^{\circ}$ and
$i\!\in\!I\!\subset\!J$ with $I\!<\!I^*$.\\

\end{enumerate}

\noindent
By Lemma~\ref{sCref_lmm}, there exists
a $\sC$-invariant open cover $(U_I''')_{I\in\cP^*(N)}$   of~$X$
properly refining~$(U_I'')_{I\in\cP^*(N)}$ and 
properly refined by~$(U_I')_{I\in\cP^*(N)}$.
Let 
$$W''=\bigcup_{I^*\subsetneq J\subset[N]}\!\!\!\!\!\!\!\cZ_{I^*,J}''^{\circ},
\qquad 
W'''=\bigcup_{I^*\subsetneq J\subset[N]}\!\!\!\!\!\!\!\cZ_{I^*,J}'''^{\circ}.$$
We define $f\!: W''\lra\R^+$ by choosing $i\!\in\!I^*$ and setting
\BE{fdfn_e}\begin{split}
f\big([x,(z_j)_{j\in I^*}]\big)&=
f_J\big(\Th_{J,I^*}\big([x,(z_j)_{j\in I^*}]\big)\!\big) \!\!\!\!
\prod_{j\in J-I^*}\!\!\!\!\!\sqrt{2(\mu_J(x))_{ij}\!+\!|z_i|^2}\\
&\hspace{1.5in}~\forall\,
\big[x,(z_j)_{j\in I^*}\big]\!\in\!\cZ_{I^*,J}''^{\circ},\,
I^*\!\subsetneq\!J\!\subset\![N].
\end{split}\EE
By~\eref{wtcZIdfn_e} and Definition~\ref{SympCut_dfn0}\ref{UIJpos_it},
$f$ is well-defined on each $\cZ_{I^*,J}''^{\circ}$
and is independent of the choice of $i\!\in\!I^*$.
By~\ref{IJinter_it} and~\ref{phiIUJ_it} in Definition~\ref{SympCut_dfn0}, \eref{ThCoc_e},
and the last inductive assumption, $f$ is well-defined on the overlaps.
Since each map $\Th_{J,I^*}$ is 
$(\phi_{\cZ;J}^{\circ},\phi_{\cZ;I^*}^{\circ})$-equivariant
with respect to the homomorphism~\eref{S1IJhomom_e} with $I\!=\!I^*$ 
and  each map $f_J$ is $\phi_{\cZ;J}^{\circ}$-invariant,
the map~$f$ is $\phi_{\cZ;I^*}^{\circ}$-invariant.\\  

\noindent
Since  $\ov{U_I'''}\!\subset\!U_I''$ for all $I\!\in\!\cP^*(N)$,
the closure of $\wt\cZ_{I^*,J}'''^{\circ}$ in~$\wt\cZ_{I^*}''^{\circ}$ is contained
in $\wt\cZ_{I^*,J}''^{\circ}$ for all $J\!\in\!\cP_{I^*}(N)$.
Thus, the closure $\ov{W'''}$ of~$W'''$ in~$\cZ_{I^*}''^{\circ}$ is contained in~$W''$.
Therefore, there exists a smooth~function 
\BE{fIcoll_e5}f_{I^*}\!:\cZ_{I^*}''^{\circ}\lra\R^+ 
\qquad\hbox{s.t.}\quad f_{I^*}\big|_{W'''}=f\,.\EE
Since the group $(S^1)^{I^*}_{\bu}$ is compact, we can make $f_{I^*}$ 
$\phi_{\cZ;I^*}^{\circ}$-invariant by averaging it over the group action.
Since $f$ is $\phi_{\cZ;I^*}^{\circ}$-invariant, this does not change~$f_{I^*}$  
over~$W'''$ and so~\eref{fIcoll_e5} still holds.
By~\eref{fIcoll_e5} and~\eref{fdfn_e}, the equality in~\eref{fIcoll_e} with $I\!=\!I^*$ holds
for all $[x,(z_j)_{j\in I}]\!\in\!\cZ_{I^*,J}'''^{\circ}$ and
$J\!\in\!\cP_{I^*}(N)$.\\

\noindent
The open cover $(U_I''')_{I\in\cP^*(N)}$ of~$X$ and the tuple
$(f_I)_{I\le I^*}$ of smooth functions
satisfy the inductive assumptions of the bullet points 
with $U_I''$ replaced by $U_I'''$ for all $I\!\le\!I^*$.
Continuing in this way, we obtain a $\sC$-invariant open cover  
$(U_I'')_{I\in\cP^*(N)}$  of~$X$
 refining~$(U_I)_{I\in\cP^*(N)}$ and properly refined by~$(U_I')_{I\in\cP^*(N)}$
and a tuple $(f_I\!:\cZ_I''^{\circ}\!\lra\!\R^+)_{I\in\cP^*(N)}$ 
of smooth functions so that each $f_I$ is $\phi_{\cZ;I}^{\circ}$-invariant 
and the equality in~\eref{fIcoll_e}
holds for all $[x,(z_j)_{j\in I}]\!\in\!\cZ_{I,J}''^{\circ}$ and
$i\!\in\!I\!\subset\!J\!\subset\![N]$.
Restricting each $f_I$ to~$\cZ_I'^{\circ}$, we obtain 
a desired tuple of functions.
\end{proof}

\noindent
For $I\!\in\!\cP^*(N)$, we define
$$\wt\cZ_{0;I}^{\circ}
=\bigcup_{\eset\neq I_0\subset I}
\hspace{-.12in}\wt\cZ_{I_0;I}^{\circ}\subset \wt\cZ_I^{\circ},\qquad
\wt\cZ_{\prt;I}^{\circ}
=\bigcup_{\begin{subarray}{c}I_0\subset I\\ |I_0|=2\end{subarray}}
\!\!\!\wt\cZ_{I_0;I}^{\circ}\subset \wt\cZ_{0;I}^{\circ}\,.$$

\begin{lmm}\label{fIsub_lmm}
Suppose $I\!\in\!\cP^*(N)$, $\wt{f}_I\!:\wt\cZ_I^{\circ}\!\lra\!\R^+$
is a smooth function, and
\BE{wtpicZIdfn_e}\wt\pi_{\cZ;I}\!: \wt\cZ_I^{\circ}\lra\C, \qquad
\wt\pi_{\cZ;I}\big(x,(z_j)_{j\in I}\big)=\wt{f}_I(x,(z_j)_{j\in I}\big)\prod_{j\in I}\!z_j\,.\EE
Then there exists an  open subset $\wt{W}_I\!\subset\!\wt\cZ_I^{\circ}$ 
containing $\wt\cZ_{0;I}^{\circ}$ such~that the homomorphism
\BE{fIsub_e2a}\nd_{\wt{x}}\wt\pi_{\cZ;I}\!:
\ker\nd_{\wt{x}}\pi_1\big|_{T_{\wt{x}} \wt\cZ_I^{\circ}}\lra \C,\EE
where  $\pi_1\!:U_I\!\times\!\C^I\!\lra\!U_I$ is the projection,
is surjective for all $\wt{x}\!\in\!\wt{W}_I\!-\!\wt\cZ_{\prt;I}^{\circ}$.
\end{lmm}

\begin{proof}
Choose $i\!\in\!I$ and define 
$$h_{i;I}\!:\C^I\lra\R^{I-i}, \qquad
h_{i;I}\big((z_j)_{j\in I}\big)=
\frac12\big(|z_j|^2\!-\!|z_i|^2\big)_{\!\!j\in I-i}\,.$$
Thus, 
\BE{fIsub_e4}\ker\nd_{\wt{x}}\pi_1\big|_{T_{\wt{x}}\wt\cZ_I^{\circ}}=\ker\nd_zh_{i;I}
\subset\C^I \subset T_xX\oplus T_z\C^I=T_{\wt{x}}\big(U_I\!\times\!\C^I\big)
\quad\forall\,\wt{x}\!\equiv\!(x,z)\!\in\!\wt\cZ_I^{\circ}.\EE
Let $\wt{x}\!=\!(x,z)$ be an element of
 $\wt\cZ_I^{\circ}\!-\!\wt\cZ_{\prt;I}^{\circ}$ and $z\!=\!(z_j)_{j\in I}$. 
Since $\wt{x}\!\not\in\!\wt\cZ_{\prt;I}^{\circ}$,
$z_j\!=\!0$ for at most one element $j\!\in\!I$.\\

\noindent
Suppose $z_{j^*}\!=\!0$ for some $j^*\!\in\!I$.
By~\eref{fIsub_e4},
$\ker\nd_{\wt{x}}\pi_1|_{T_{\wt{x}}\wt\cZ_I^{\circ}}$ then contains
the component $\C^{\{j^*\}}\!\subset\!\C^I$~and 
$$\nd_{\wt{x}}\wt\pi_{\cZ;I}\!=
\bigg(\wt{f}_I(\wt{x}\big)\!\!\!\!\!\prod_{j\in I-j^*}\!\!\!\!\!z_j\!\bigg)\Id: 
\C^{\{j^*\}}\lra\C\,.$$
This homomorphism is surjective if $z_j\!\neq\!0$ for all $j\!\in\!I\!-\!j^*$.\\

\noindent
From now on, we assume that $z_j\!\neq\!0$ for all $j\!\in\!I$.
Denote~by 
$$\R^I\subset \C^I\subset T_{\wt{x}}\big(U_I\!\times\!\C^I)$$
the radial tangent directions. 
Let
$$\R^I_{\wt{x}} = 
\big(\!\ker\nd_{\wt{x}}\pi_1\big|_{T_{\wt{x}}\wt\cZ_I^{\circ}}\big) \cap\R^I$$
and $V_{\wt{x}}\!\subset\!T_{\wt{x}}\wt\cZ_I^{\circ}$ be the span of $\R^I_{\wt{x}}$
and of the angular tangent space in the $i$-th $\C$-factor $\C^{\{i\}}\!\subset\!\C$.
We will show that the homomorphism
$$\nd_{\wt{x}}\wt\pi_{\cZ;I}\!: V_{\wt{x}}\lra \C$$
is an isomorphism if $\wt{x}$ lies in a sufficiently small 
neighborhood  $\wt{W}_I\!\subset\!\wt\cZ_I^{\circ}$ 
of~$\wt\cZ_{0;I}^{\circ}$.\\

\noindent
Define
\begin{alignat*}{3}
\pi_{\C^I}\!: \C^I&\lra\C, &\qquad 
\pi_{\C^I}^{\R}\!: \C^I&\lra\R, &\qquad  
H_{i;I}\!: \C^I&\lra\R^I\!\equiv\!\R^{I-i}\!\times\!\R,\\
\pi_{\C^I}\big((z_j)_{j\in I}\big)&=\prod_{j\in I}\!z_j, &\qquad  
\pi_{\C^I}^{\R}\big((z_j)_{j\in I}\big)&=\prod_{j\in I}\!|z_j|, &\qquad  
H_{i;I}(z)&=\big(h_{i;I}(z),\pi_{\C}^{\R}(z)\!\big).
\end{alignat*}
The determinant of the restriction of $\nd_zH_{i;I}$ 
to the radial tangent directions \hbox{$\R^I\!\subset\!T_z\C^I$} is given~by 
$$\det\big(\nd_zH_{i;I}\!:\R^I\!\lra\!\R^I\big)
=\!\sum_{j\in I}\prod_{k\in I-j}\!\!\!\!|z_k|^2\,.$$
This implies that there exists a universal constant~$\de_I$ (dependent only on~$|I|$)
such~that
$$\big|\nd_z\pi_{\C^I}^{\R}\!:\R_{\wt{x}}^I\!\lra\!\R\big|
\ge \de_I\sum_{j\in I}\prod_{k\in I-j}\!\!\!\!|z_k|$$
with respect to the standard norms on $\R_{\wt{x}}^I\!\subset\!\C^I$ and~$\R\!\subset\!\C$.
Since the differential of the angular component of $\pi_{\C^I}$
along the angular direction in~$V_{\wt{x}}$ satisfies the same bound,
it follows that the homomorphism
$$\nd_{\wt{x}}\pi_{\C^I}\!+\!D\!:  V_{\wt{x}}\lra\C$$
is an isomorphism whenever $D\!:V_{\wt{x}}\!\lra\!\C$ is a homomorphism
such that 
$$\|D\|<\de_I \sum_{j\in I}\prod_{k\in I-j}\!\!\!\!|z_k|.$$\\

\noindent
By the definition of~$\wt\pi_{\cZ;I}$,
$$\nd_{\wt{x}}\wt\pi_{\cZ;I}=\wt{f}_I(\wt{x})\,\nd_z\pi_{\C^I}
+\bigg(\prod_{k\in I}\!z_k\bigg)\nd_{\wt{x}}\wt{f}_I.$$
Thus, the restriction of $\nd_{\wt{x}}\wt\pi_{\cZ;I}$ to $V_{\wt{x}}$
is an isomorphism~if
$$ \frac{\|\nd_{\wt{x}}\wt{f}_I\!:V_{\wt{x}}\!\lra\!\C\|}{\wt{f}_I(\wt{x})}
\big|z_j\big|<\de_I \qquad\hbox{for some}~~j\!\in\!I.$$
This specifies an open subset $\wt{W}_I\!\subset\!\wt\cZ_I^{\circ}$  
containing~$\wt\cZ_{0;I}^{\circ}$ so that the homomorphism~\eref{fIsub_e2a} is 
surjective on $\wt{W}_I\!-\!\wt\cZ_{\prt;I}^{\circ}$.
\end{proof}

\noindent
For $I\!\in\!\cP^*(N)$, we define
\BE{wtcZiIRdfn_e}\wt\cZ_{I;i}^>=\big\{\big(x,(z_j)_{j\in I}\big)\!\in\!\wt\cZ_I^{\circ}\!:\,
z_i\!\in\!\C^*,\,z_j\!\in\!\R^+\,\forall\,j\!\in\!I\!-\!i\big\}
\subset U_I\!\times\!\C^I.\EE
By~\eref{wtphiIdfn_e}, $\wt\cZ_{I;i}^>$ is a slice for the $\wt\phi_I$-action
on~$\wt\cZ_I^{\circ}\!-\!\wt\cZ_{0;I}^{\circ}$.
By Lemma~\ref{wtcZreg_lmm}, 
$\wt\cZ_{I;i}^>$ is a smooth submanifold of $U_I\!\times\!\C^I$
of $\dim_{\R}\!X\!+\!2$.

\begin{lmm}\label{fIgraph_lmm}
Let $I$,  $\wt{f}_I$, and $\wt\pi_{\cZ;I}$ be as in Lemma~\ref{fIsub_lmm}.
Then there exists an
open subset $\wt{W}_I\!\subset\!\wt\cZ_I^{\circ}$ 
containing $\wt\cZ_{0;I}^{\circ}$ such that
for every $i\!\in\!I$ the restriction of the smooth~map
\BE{fIgraph_e}
\pi_1\!\times\!\wt\pi_{\cZ;I}\!: \wt\cZ_{I;i}^>\lra U_I\!\times\!\C\EE
to $\wt\cZ_{I;i}^>\!\cap\!\wt{W}_I$ 
is a diffeomorphism onto the intersection of $U_I\!\times\!\C^*$
with an open neighborhood of~$U_I\!\times\!0$ in~$U_I\!\times\!\C$.
\end{lmm}

\begin{proof}
Let $\wt{W}_I\!\subset\!\wt\cZ_I^{\circ}$ be an open subspace
provided by Lemma~\ref{fIsub_lmm}.
By  Lemma~\ref{fIsub_lmm} and its proof,
the homomorphisms
$$\nd_{\wt{x}}\wt\pi_{\cZ;I}\!: \ker\nd_{\wt{x}}\pi_1\big|_{T_{\wt{x}}\wt\cZ_I^{\circ}}
\lra \C
\qquad\hbox{and}\qquad
\nd_{\wt{x}}\pi_1\!:T_{\wt{x}}\wt\cZ_I^{\circ}\lra T_{\pi_1(\wt{x})}X $$
are surjective for all $\wt{x}\!\in\!\wt{W}_I\!-\!\wt\cZ_{\prt;I}^{\circ}$.
For dimensional reasons, this implies~that the differential of~\eref{fIgraph_e}
is an isomorphism for all $\wt{x}\!\in\!\wt\cZ_{I;i}^>\!\cap\!\wt{W}_I$.
From the Inverse Function Theorem \cite[Theorem~1.30]{Warner},
we then conclude that~\eref{fIgraph_e} is a local diffeomorphism from 
$\wt\cZ_{i;I}^>\!\cap\!\wt{W}_I$ onto an open subset of $U_I\!\times\!\C^*$.\\

\noindent
We next show that~\eref{fIgraph_e} is injective on~$\wt\cZ_{I;i}^>\!\cap\!\wt{W}_I$,
after possibly shrinking the neighborhood~$\wt{W}_I$ of~$\wt\cZ_{0;I}^{\circ}$.
For each $x\!\in\!U_I$, let 
$$\vr_{I;i}(x)=\max\!\big\{\!-\!2(\mu_I(x))_{ij}\!:j\!\in\!I\big\}
\in\R^{\ge0},$$
with $(\mu_I(x))_{ij}$ given by~\eref{hijdfn_e}.
For $\la\!\in\!\C^*$ and $j\!\in\!I$, define 
$$g_{\la;I;i;j}\!:
\big\{(x,\vr)\!\in\!U_I\!\times\!\R\!:\,\vr\!>\!\vr_{I;i}(x)\big\}\lra\C^*, \quad
g_{\la;I;i;j}(x,\vr)=\begin{cases} \frac{\la}{|\la|}\sqrt{\vr},&\hbox{if}~i\!=\!j;\\
\sqrt{2(\mu_I(x))_{ij}\!+\!\vr},&\hbox{if}~j\!\in\!I\!-\!i.
\end{cases}$$
By~\eref{wtcZIdfn_e}, \eref{wtcZiIRdfn_e}, and~\eref{wtpicZIdfn_e}, 
every preimage  $\wt{x}\!\equiv\!(x,(z_j)_{j\in I})$ of \hbox{$(x,\la)\!\in\!U_I\!\times\!\C^*$}
under~\eref{fIgraph_e} is of the~form
$$\big(x,(z_j)_{j\in I}\big)=
\big(x,\big(g_{\la;I;i;j}(x,\vr)\!\big)_{\!j\in I}\big)$$
with $\vr\!=\!\vr_i(\la,x)$ being a solution~of the~equation
\BE{vrdfn_e}
\wt{f}_I\big(x,\big(g_{\la;I;i;j}(x,\vr)\!\big)_{\!j\in I}\big)^2
\vr\!\!\!\prod_{j\in I-i}\!\!\!g_{\la;I;i;j}(x,\vr)^2=|\la|^2,
\quad \vr\in\big(\vr_{I;i}(x),\i\big).\EE
At $\vr\!=\!\vr_{I;i}(x)$, the first non-vanishing derivative of the left-hand side
of~\eref{vrdfn_e} is strictly positive.
Thus, there exist continuous functions 
$$\de_{I;i},\ve_{I;i}\!:U_I\lra\R^+$$
such~that the left-hand side of~\eref{vrdfn_e} as a function of~$\vr$
is injective on \hbox{$[\vr_{I;i}(x),\vr_{I;i}(x)\!+\!\de_{I;i}(x)]$} and 
 the image of this interval contains the interval $[0,\ve_{I;i}(x)]$.\\

\noindent
By the previous paragraph, the map~\eref{fIgraph_e} is injective on the intersection of
$\wt\cZ_{I;i}^>$~with
$$\wt{W}_{i;I}'\equiv\big\{(x,(z_j)_{j\in I})\!\in\!\wt{W}_I\!:\,
\min_{j\in I}\!|z_j|^2\!<\!\de_{I;i}(x)\big\}$$
and the image of this intersection under~\eref{fIgraph_e} contains
$$\big\{(x,z)\!\in\!U_I\!\times\!\C^*\!:\,|z|\!<\!\ve_{I;i}(x)\big\}
\subset U_I\!\times\!\C\,.$$
Since $\wt{W}_{I;i}'\!\subset\!\wt{W}_I$ is an open neighborhood of~$\wt\cZ_{0;I}^{\circ}$,
the proof is completed by taking $\wt{W}_I$ in the statement of Lemma~\ref{fIgraph_lmm}
to be the intersection of the sets $\wt{W}_{I;i}'$ over $i\!\in\!I$. 
\end{proof}

\begin{crl}\label{fIgraph_crl}
Let $I$,  $\wt{f}_I$, and $\wt\pi_{\cZ;I}$ be as in Lemma~\ref{fIsub_lmm}.
Then there exists an open subset $\wt{W}_I\!\subset\!\wt\cZ_I^{\circ}$ 
containing~$\wt\cZ_{0;I}^{\circ}$ such~that 
$$\wt\om_I^n\big|_{T_{\wt{x}}(\wt\pi_{\cZ;I}^{-1}(\la))}\neq0
 \quad\forall~\wt{x}\!\in\!\wt\pi_{\cZ;I}^{-1}(\la),\,\la\!\in\!\C^*\,,
\qquad\hbox{where}\quad 2n\!=\!\dim_{\R}\!X.$$
\end{crl}

\begin{proof}
Let $\la\!\in\!\C^*$, $\wt{W}_I\!\subset\!\wt\cZ_I^{\circ}$ be an open subspace
provided by Lemma~\ref{fIgraph_lmm}, $i\!\in\!I$, and 
$$\wt\cZ_{\la;I;i}^>=\wt\pi_{\cZ;I}^{-1}(\la)\!\cap\!\wt\cZ_{I;i}^>.$$
In particular, $\wt\cZ_{\la;I;i}^>$ is a smooth submanifold of~$U_I\!\times\!\C^I$ and
the~map
$$\pi_1\!:\wt\cZ_{\la;I;i}^>\lra U_I$$
is a diffeomorphism onto an open subset~$U_{\la;I;i}$.
Thus, there exists a smooth function 
\begin{gather*}
g_{\la;I;i}\!: U_{\la;I;i}\lra 
\big\{(z_j)_{j\in I}\!\in\!\C^I\!:\,
z_i\!\in\!\frac{\la}{|\la|}\R^+,\,z_j\!\in\!\R^+\,\forall\,j\!\in\!I\!-\!i\big\}\\
\hbox{s.t.}\qquad
\wt\cZ_{\la;I;i}^>=\big\{\big(x,g_{\la;I;i}(x)\big)\!:\,x\!\in\!U_{\la;I;i}\big\}.
\end{gather*}
Since the argument of every component of $g_{\la;I;i}\!:U_{\la;I;i}\!\lra\!\C^I$ is fixed,
\BE{fIgraph_e17}
\big\{\id\!\times\!g_{\la;I;i}\big\}^{\!*}\wt\om_I
=\big\{\id\!\times\!g_{\la;I;i}\big\}^{\!*}\pi_1^*\om
+\big\{\id\!\times\!g_{\la;I;i}\big\}^{\!*}\pi_2^*\om_{\C^I}
=\om|_{U_{\la;I;i}}+g_{\la;I;i}^{\!*}\om_{\C^I}=\om|_{U_{\la;I;i}}.\EE
This implies that
$$\big\{\id\!\times\!g_{\la;I;i}\big\}^{\!*}
\big(\wt\om_I^n\big|_{T_{\wt{x}}(\wt\pi_{\cZ;I}^{-1}(\la))}\big)
=\big\{\id\!\times\!g_{\la;I;i}\big\}^{\!*}\big(\wt\om_I^n\big|_{\wt{x}}\big)
=\om^n\big|_{\pi_1(\wt{x})}\neq0 \qquad\forall~\wt{x}\!\in\!\wt\cZ_{\la;I;i}^>,$$
since $\om$ is a symplectic form.
Since the $\wt\phi_I$-action  consists of a family of symplectomorphisms
on $U_I\!\times\!\C^I$, this establishes the claim.
\end{proof}

\subsection{One-parameter family of smoothings}
\label{SympNBpf_subs}

\noindent
By Lemma~\ref{sCref_lmm}, there exists a $\sC$-invariant open cover 
$(U_I')_{I\in\cP^*(N)}$ of~$X$ properly refining~$(U_I)_{I\in\cP^*(N)}$. 
By Lemma~\ref{fIcoll_lmm}, there exists a tuple 
$(f_I\!:\cZ_I'^{\circ}\!\lra\!\R^+)_{I\in\cP^*(N)}$
of smooth functions such~that each~$f_I$ is $\phi_{\cZ;I}^{\circ}$-invariant
and~\eref{fIcoll_e} holds.
Define
\begin{gather*}
\pi\!: \cZ\lra\C \quad\hbox{by}\quad
\pi\big(\big[[x,(z_i)_{i\in I}]\big]\big)=
f_I\big([x,(z_i)_{i\in I}]\big)\!\prod_{i\in I}\!z_i
~~~\forall\,
\big[x,(z_i)_{i\in I}\big]\!\in\!\cZ_I'^{\circ},\,I\!\in\!\cP^*(N).
\end{gather*}
By~\eref{wtphiIdfn_e}, $\pi$ is independent of the choice of representative for 
an element of $\cZ_I'^{\circ}$ for $I\!\in\!\cP^*(N)$ fixed.
By~\eref{cZdfn_e} and~\eref{fIcoll_e},  
$\pi$ is also independent of the choice of $I\!\in\!\cP^*(N)$
and so is well-defined.
Since all~$f_I$ take values in~$\R^+$, 
$$\cZ_0\equiv \pi^{-1}(0)= \cZ_{\{1\}}\!\cup\!\ldots\!\cup\!\cZ_{\{N\}}
=X_1\!\cup\!\ldots\!\cup\!X_N \equiv X_{\eset}\,;$$
the second equality above holds by~\eref{Xidfn_e}.

\begin{lmm}\label{cZfamsm_lmm}
There exists a neighborhood~$\cZ'$ of $\cZ_0\!=\!X_{\eset}$ such that 
$(\cZ',\om_{\cZ}|_{\cZ'},\pi|_{\cZ'})$ is a one-parameter family 
of smoothings of the SC symplectic variety~$X_{\eset}$ 
associated with the SC symplectic configuration~$\X(\sC)$ of Corollary~\ref{SympNB_crl2}. 
\end{lmm}

\begin{proof}
For each $I\!\in\!\cP^*(N)$, define
\begin{alignat}{2}
\label{wtfiIdfn_e}
\wt{f}_I\!: \wt\cZ_I'^{\circ}&\lra\R^+, &\qquad
\wt{f}_I\big(x,(z_i)_{i\in I}\big)&=f_I\big([x,(z_i)_{i\in I}]\big),\\
\notag
\wt\pi_{\cZ;I}\!: \wt\cZ_I'^{\circ}&\lra\C, &\qquad
\wt\pi_{\cZ;I}\big(x,(z_i)_{i\in I}\big)&=\pi\big(\big[[x,(z_i)_{i\in I}]\big]\big).
\end{alignat}
Let $\wt{W}_I\!\subset\!\wt\cZ_I'^{\circ}$ be an open subset as in Lemmas~\ref{fIsub_lmm}
and~\ref{fIgraph_lmm} with $(U_I)_{I\in\cP^*(N)}$ replaced by $(U_I')_{I\in\cP^*(N)}$.
By Lemma~\ref{cZtopol_lmm3}, the image $W_I\!\subset\!\cZ$ of~$\wt{W}_I$
under the quotient map 
\BE{wtcZiItocZ_e} q_{\cZ}\!\circ\!q_{\cZ;I}\!:\wt\cZ_I'^{\circ}\!\lra\!\cZ \EE 
is an open subset. 
The union~$\cZ'$ of the open subsets~$W_I$ taken over all 
$I\!\in\!\cP^*(N)$ is an open neighborhood of $\cZ_0\!=\!X_{\eset}$
in~$\cZ$.\\

\noindent
Since~$\wt\pi_{\cZ;I}$ factors through~$\pi$,
the surjectivity of the homomorphism~\eref{fIsub_e2a} implies that 
$\pi$~is a submersion on~$W_I$ outside~of 
$$X_{\prt}=\bigcup_{\begin{subarray}{c}I\subset[N]\\ |I|=2\end{subarray}}\!\!\!\cZ_I\subset \cZ\,.$$
For any $\la\!\in\!\C^*$, the pullback of $\om_{\cZ}|_{\pi^{-1}(\la)}$ by~\eref{wtcZiItocZ_e}
is $\wt\om_I\big|_{T(\wt\pi_{\cZ;I}^{-1}(\la))}$.
By Corollary~\ref{fIgraph_crl},  $\om_{\cZ}^n|_{\pi^{-1}(\la)}$ 
thus does not vanish over~$W_I$
and so the restriction of~$\om_{\cZ}$ to $\pi^{-1}(\la)\!\cap\!W_I$ is nondegenerate.
Thus, $(\cZ',\om_{\cZ}|_{\cZ'},\pi|_{\cZ'})$ is a nearly regular symplectic fibration 
in the sense of Definition~\ref{SimpFibr_dfn} with $\{\pi|_{\cZ'}\}^{-1}(0)\!=\!X_{\eset}$.
\end{proof}
 
\noindent
For each $I\!\in\!\cP^*(N)$, the subspace $\wt\cZ_{0;I}'^{\circ}\!\subset\!\wt\cZ_I'^{\circ}$
is preserved by the $\wt\phi_{\cZ;I}$-action in~\eref{wtcZphiiIdfn_e}.
By replacing $\wt{W}_I$ in the proof of Lemma~\ref{cZfamsm_lmm} with 
$$ \bigcap_{g\in(S^1)^I_{\bu}}\!\!\!\!\!
\wt\phi_{\cZ;I}\big(\{g\}\!\times\!\wt{W}_I\big)
\subset U_I'\!\times\!\C^I\,,$$
we can thus assume that $\wt{W}_I$ is $\wt\phi_{\cZ;I}$-invariant.
This implies that the subspace $W_I\!\subset\!U_{\cZ;I}$ is preserved
by the $\phi_{\cZ;I}$-action in~\eref{Ham_e3d}.
The collection $(W_I)_{I\in\cP^*(N)}$ is then a $\sC_{\cZ}$-invariant 
open cover of the subspace~$\cZ'$ of~$\cZ$
(corresponding to $(U_{\sZ;t;I}')_{I\in\cP^*(N)}$ in Theorem~\ref{SympCut_thm3}).
We denote by
$$\sC_{\cZ}'\equiv\big(W_I,\phi_{\cZ;I}',\mu_{\cZ;I}'\big)_{I\in\cP^*(N)}$$
the restriction of~$\sC_{\cZ}$ to~$(W_I)_{I\in\cP^*(N)}$.

\begin{lmm}\label{cZfamcutconf_lmm}
For every $\la\!\in\!\C$,
the restriction $\sC_{\cZ}'|_{\cZ'\cap\pi^{-1}(\la)}$ of $\sC_{\cZ}'$ to 
$\cZ'\!\cap\!\pi^{-1}(\la)$
is an $N$-fold Hamiltonian configuration for 
$(\cZ'\!\cap\!\pi^{-1}(\la),\om_{\cZ}|_{\cZ'\cap\pi^{-1}(\la)})$.
It is a cutting configuration if $\la\!\in\!\C^*$.
\end{lmm}

\begin{proof}
By Lemma~\ref{cZfamsm_lmm}, $\om_{\cZ}|_{\cZ'\cap\pi^{-1}(\la)}$ is a symplectic form
on $\cZ'\!\cap\!\pi^{-1}(\la)$.
Since the function~$f_I$ is $\phi_{\cZ;I}^{\circ}$-invariant for each $I\!\in\!\cP^*(N)$,
the restriction of~$\pi$ to~$W_I$ is $\phi_{\cZ;I}'$-invariant.
Thus, $\cZ'\!\cap\!\pi^{-1}(\la)$ is $\sC_{\cZ}'$-invariant.
This establishes the first claim.
If  $\la\!\in\!\C^*$, then 
$$\pi^{-1}(\la)\subset \cZ\!-\!X_{\eset}\subset \cZ\!-\!X_{\prt}$$
and so $\sC_{\cZ}'|_{\cZ'\cap\pi^{-1}(\la)}$ is a restriction of 
the cutting configuration of Corollary~\ref{cZHamConf_crl}.
This establishes the second claim.
\end{proof}

\begin{lmm}\label{cZfamcmpt_lmm}
Suppose $X$ is compact,  
$(U_I'')_{I\in\cP^*(N)}$ is a $\sC$-invariant open cover
properly refining $(U_I')_{I\in\cP^*(N)}$, 
$\sC''$ is the restriction of~$\sC$ to $(U_I'')_{I\in\cP^*(N)}$, and 
$\sC_{\cZ}''$ is the analogue of 
the $N$-fold Hamiltonian configuration~$\sC_{\cZ}$ in~\eref{cZSympCutDfn_e}
for~$\sC''$.
Then, there exist
neighborhoods $\cZ'\!\subset\!\cZ$ of~$X_{\eset}$ 
and $\De\!\subset\!\C$ of~$0$ such that
the symplectic manifold $(\pi^{-1}(\la),\om_{\cZ}|_{\pi^{-1}(\la)})$ with 
the cutting configuration~$\sC_{\cZ}''|_{\pi^{-1}(\la)}$
is canonically isomorphic to the symplectic manifold~$(X,\om)$ 
with the cutting configuration~$\sC''$ for every $\la\!\in\!\De\!\cap\!\R^+$.
\end{lmm}

\begin{proof} For $i\!\in\!I\!\subset\![N]$ and $\la\!\in\!\R^+$, let  
$$\ve_{I;i}\!:U_I'\lra\R^+ \qquad\hbox{and}\qquad
g_{\la;I;i}\!:U_{\la;I;i}'\lra (\R^+)^I$$
be as in the proofs of Lemma~\ref{fIgraph_lmm} and  Corollary~\ref{fIgraph_crl},
respectively, with $(U_I)_{I\in\cP^*(N)}$ replaced by $(U_I')_{I\in\cP^*(N)}$.
Since the subsets $\ov{U_I''}\!\subset\!U_I'$ are compact, 
$$\ep\equiv \min_{i\in I\subset[N]}\big\{\ve_{I;i}(x)\!:\,x\!\in\!\ov{U_I''}\big\}\in\R^+\,.$$
Let $\De\!\subset\!\C$ be the disk of radius~$\ep$ around the origin
and $\cZ'\!=\!\pi^{-1}(\De)$.
Since each function~$f_I$ is $\phi_{\cZ;I}^{\circ}$-invariant,
$\cZ'\!\subset\!\cZ$ is a $\sC_{\cZ}$-invariant subspace.
By the proofs of Corollary~\ref{fIgraph_crl} and Lemma~\ref{cZfamsm_lmm},
it satisfies the condition of the latter.\\

\noindent
By~\eref{wtfiIdfn_e} and~\eref{fIcoll_e},
\begin{equation*}\begin{split} 
\wt{f}_I\big(x,(z_j)_{j\in I}\big)&=
\wt{f}_J\big(\wt\Th_{J,I}\big(x,(z_j)_{j\in I}\big)\!\big) \!\!\!\!
\prod_{j\in J-I}\!\!\!\!\sqrt{2(\mu_J(x))_{ij}\!+\!|z_i|^2}\\
&\hspace{1.5in}~\forall~
\big(x,(z_j)_{j\in I}\big)\!\in\!\wt\cZ_{I,J}'^{\circ},~i\!\in\!I\!\subset\!J\!\subset\![N].
\end{split}\end{equation*}
By the uniqueness of~$g_{\la;i;I}$, this implies~that
\BE{cmptcase_e5} 
\wt\Th_{J,I}\big(x,g_{\la;I;i}(x)\big)=\big(x,g_{\la;J;i}(x)\big) \qquad
\forall~x\!\in\!U_{\la;I;i}'\!\cap\!U_{\la;J;i}',~
i\!\in\!I,J\!\subset\![N].\EE
Furthermore, 
\BE{cmptcase_e7} 
g_{\la;I;i}(x) = g_{\la;I;j}(x) 
\qquad\forall~x\!\in\!U_{\la;I;i}'\!\cap\!U_{\la;I;j}',
~i,j\!\in\!I\!\subset\![N].\EE
In contrast to~\eref{cmptcase_e5}, \eref{cmptcase_e7} does not hold for 
$\la\!\in\!\C^*\!-\!\R^+$
(even after passing to the quotient~$\cZ_I'^{\circ}$).\\

\noindent
By the definition of~$\ep$, $U_{\la;I;i}''\!=\!U_I''$ for all $\la\!\in\!\De\!-\!0$
and $i\!\in\!I\!\subset\![N]$.
Define
$$f_{\la}\!: X\lra \pi^{-1}(\la)\!\subset\!\cZ', \qquad
f_{\la}(x)=q_{\cZ}\big(q_{\cZ;I}\big(x,g_{\la;I;i}(x)\big)\big)
\quad\forall~x\!\in\!U_I'',\,i\!\in\!I\!\subset\![N].$$
By~\eref{cmptcase_e5} and~\eref{cZdfn_e},  
$f_{\la}$ is independent of the choice of $I\!\in\!\cP_i(N)$ for  $i\!\in\![N]$ fixed.
By~\eref{cmptcase_e7}, $f_{\la}$ is also independent of the choice of 
$i\!\in\![N]$ and so is well-defined.\\

\noindent
Since the graph of~$g_{\la;I;i}|_{U_I''}$ is a slice
for the $\wt\phi_I$-action on $\wt\pi_{\cZ;I}^{-1}(\la)\!\cap\!\wt\cZ_I''^{\circ}$,
$$f_{\la}(U_I'')=\pi^{-1}(\la)\!\cap\!U_{\cZ;I}'' \qquad\forall~I\!\in\!\cP^*(N).$$
Since the sets $U_{\cZ;I}''\!=\!q_{\cZ}(q_{\cZ;I}(\wt\cZ_{I}''^{\circ}))$ cover~$\cZ$,
the map~$f_{\la}$ is thus surjective.
By~\eref{cZdfn_e} and \eref{cmptcase_e5},  it is injective.
Since the restriction of~$f_{\la}$ to each~$U_I''$
 is a composition of smooth maps, $f_{\la}$ is a smooth map. 
By~\eref{SympRed_e4} and~\eref{fIgraph_e17}, 
\begin{equation*}\begin{split}
f_{\la}^*\om_{\cZ}\,\big|_{U_I''}
=\big\{\id\!\times\!g_{\la;I;i}\big\}^{\!*}
q_{\cZ;I}^{\!*}q_{\cZ}^*\om_{\cZ}\,\big|_{U_I''}
&=\big\{\id\!\times\!g_{\la;I;i}\big\}^{\!*}
q_{\cZ;I}^{\!*}\vp_I\,\big|_{U_I''}\\
&=\big\{\id\!\times\!g_{\la;I;i}\big\}^{\!*}\wt\om_I\,\big|_{U_I''}
=\om\big|_{U_I''}\,.
\end{split}\end{equation*}
Thus, $f_{\la}$ is a symplectomorphism from $(X,\om)$ to 
$(\pi^{-1}(\la),\om_{\cZ}|_{\pi^{-1}(\la)})$.\\

\noindent
By the construction of~$\mu_{\cZ;I}$ in Section~\ref{cZSympCutConf_subs}, 
$$\mu_{\cZ;I}\!\circ\!f_{\la}\,\big|_{U_I''}
=\wt\mu_{\cZ;I}\!\circ\!\big\{\id\!\times\!g_{\la;I;i}\big\}\,\big|_{U_I''}
= \mu_I\big|_{U_I''}
\qquad\forall~I\!\in\!\cP^*(N).$$
Since $\wt{f}_I$ is a $\wt\phi_{\cZ;I}$-invariant map,
the uniqueness of~$g_{\la;I;i}$ implies that 
$g_{\la;I;i}|_{U_I''}$ is  a $\phi_I$-invariant map.
Thus, the maps $\{\id\!\times\!g_{\la;I;i}\}|_{U_I''}$ 
and $f_{\la}|_{U_I''}$ are $(\wt\phi_{\cZ;I},\phi_I)$-equivariant
and $(\phi_{\cZ;I},\phi_I)$-equivariant, respectively.
We~conclude that $f_{\la}$ is an isomorphism from the cutting configuration~$\sC''$  
for~$(X,\om)$ to the cutting configuration $\sC_{\cZ}''|_{\pi^{-1}(\la)}$
for $(\pi^{-1}(\la),\om_{\cZ}|_{\pi^{-1}(\la)})$.
\end{proof}

\noindent
The projection  $\pi\!:\cZ\!\lra\!\De$ depends on the choice of the functions~$f_I$.
By the same inductive reasoning as in the proof of Lemma~\ref{fIcoll_lmm},
any two such collections of functions are homotopic after shrinking their domains.
The latter does not effect the space~$\cZ$.
If $U_{[N]}\!=\!X$, there is a natural choice of~$\pi$ given by $f_{[N]}\!=\!1$;
see the proof of Proposition~\ref{THamConf_prp}.

\section{NC degenerations of Hamiltonian manifolds}
\label{FullTor_sec}

\noindent
The tori $(S^1)^I_{\bu}$ whose actions~$\phi_I$  appear in 
a Hamiltonian configuration~\eref{SympCutDfn_e} are subtori~of
\begin{gather}
(S^1)^N_{\bu}\equiv (S^1)^{[N]}_{\bu}=\rho_{\bu}^{-1}(\id),  
\qquad\hbox{where}\quad
\label{S1Ndiag_e} 
\rho_{\bu}\!:(S^1)^N \lra S^1,  ~~ 
\rho_{\bu}\big((\ne^{\fI\th_i})_{i\in[N]}\big)
=\prod_{i\in[N]}\!\!\!\ne^{\fI\th_i}.
\end{gather}
If the domain~$U_{[N]}$ of the $\phi_{[N]}$-action is the entire manifold~$X$,
the constructions of Sections~\ref{Thm12Pf_sec} and~\ref{SumCutSm_sec} greatly simplify
and result in a richer structure.
By Proposition~\ref{THamConf_prp}, there is then a one-parameter 
family~$\pi$
of smoothings of the associated SC symplectic variety~$X_{\eset}$
defined over the entire symplectic manifold~$(\cZ,\om_{\cZ})$.
Furthermore, a  Hamiltonian pair~$(\phi_{\T},\mu_{\T})$ for~$(X,\om)$
compatible with $(\phi_{[N]},\mu_{[N]})$
induces a Hamiltonian $\T$-action
on~$(\cZ,\om_{\cZ})$ preserving the associated cut symplectic manifolds~$(X_i,\om_i)$
and the projection~$\pi$.
We illustrate this situation on the simple local example of Section~\ref{BasicEg_subs}
when all relevant objects can be readily described explicitly.\\

\noindent
If the $\phi_{[N]}$-action is in addition the restriction of 
a Hamiltonian action of the entire torus~$(S^1)^N$ on~$(X,\om)$,
then there is a natural Hamiltonian $S^1$-pair $(\phi_{\cZ;S^1},\mu_{\cZ;S^1})$
for $(\cZ,\om_{\cZ})$ so that the projection~$\pi$
is $S^1$-equivariant; see Lemma~\ref{THamConf_lmm2}.
If $X$ is  also compact, then~$\pi$ can be ``cut"
to a one-parameter family
\BE{whpiwhcZa_e}\wh\pi\!:\wh\cZ_a\lra\P^1\EE
of smoothings of~$X_{\eset}$ with all smooth fibers (i.e.~those not over~$[1,0]$)
canonically isomorphic to~$(X,\om)$; see Corollary~\ref{CmptHam_crl}.\\

\noindent
As shown in Section~\ref{MomPolytThm_subs},
a Hamiltonian action of an abstract torus~$\T$ on~$(X,\om)$ gives rise to a plethora of 
Hamiltonian $(S^1)^N$-pairs $(\phi,\mu)$ for~$(X,\om)$ 
and often to cutting configurations.
Theorem~\ref{toriccut_thm} describes the effect 
of the constructions of  Sections~\ref{Thm12Pf_sec} and~\ref{SumCutSm_sec}
with cutting configurations arising in this way on 
the moment polytope of the original torus action. 
Proposition~\ref{CombCond_prp} 
provides a combinatorial criterion for a Hamiltonian configuration
obtained as in Section~\ref{MomPolytThm_subs} to be a cutting configuration;
it is especially effective in the case of toric symplectic manifolds.\\

\noindent
Let $\T\!\approx\!(S^1)^k$ be a $k$-torus.
A \sf{Hamiltonian $\T$-manifold} 
(\sf{Hamiltonian $\T$-space} of \cite[Definition~22.1]{daSilva})
is a tuple $(X,\om,\phi,\mu)$ such that 
$(X,\om)$ is a symplectic manifold and 
$(\phi,\mu)$ is a Hamiltonian $\T$-pair for~$(X,\om)$.
We call such a Hamiltonian $\T$-manifold \sf{compact} 
(resp.~\sf{connected}) if $X$ is compact (resp.~connected).
We call Hamiltonian $\T$-pair $(\phi,\mu)$ and $\T'$-pair $(\phi',\mu')$ for~$(X,\om)$ 
\sf{compatible} if the actions of~$\phi$ and~$\phi'$ commute,
$\mu'$ is $\phi$-invariant, and  $\mu$ is $\phi'$-invariant.
This means that the two actions and the two moment maps form a Hamiltonian 
$(\T\!\times\!\T')$-pair for~$(X,\om)$.
If $(\phi,\mu)$ is a  Hamiltonian $\T$-pair for~$(X,\om)$,
then the pair obtained by restricting the action~$\phi$ 
to a subtorus $\T'\!\subset\!\T$ and composing~$\mu$ with the restriction
$\ft^*\!\lra\!\ft'^*$ is compatible with~$(\phi,\mu)$.
Compatible pairs of a different kind play an important role in Theorem~\ref{toriccut_thm}.\\

\noindent
Suppose $X$ and $X'$ are topological spaces with $\T$-actions~$\phi$ and~$\phi'$, respectively,
and $\mu$ and~$\mu'$ are $\ft^*$-valued maps on~$X$ and~$X'$, respectively.
A~map \hbox{$f\!:X\!\lra\!X'$} \sf{intertwines} $(\phi,\mu)$ and  $(\phi',\mu')$~if
$$f\big(\phi(g;x)\big)=\phi'\big(g;f(x)\big) \quad\hbox{and}\quad
\mu(x)=\mu'\big(f(x)\big) \qquad\forall\,g\!\in\!\T,\,x\!\in\!X.$$
In such a case, we call
$$f\!:(X,\phi,\mu)\lra(X',\phi',\mu')$$
a \sf{morphism}.

\subsection{Basic setup and output}
\label{HamSpaGen_subs}

\noindent
Let $N\!\in\!\Z^+$ and $(X,\om,\phi,\mu)$ be a Hamiltonian $(S^1)^N_{\bu}$-manifold.
For each $I\!\in\!\cP^*(N)$, define
\BE{UIdfn_e}\begin{split}
U_I=\big\{x\!\in\!X\!:
(\mu(x))_i\!<\!(\mu(x))_j~&\forall\,i\!\in\!I,\,j\!\in\![N]\!-\!I\big\},\\
\phi_I\!=\!\phi|_{(S^1)^I_{\bu}\times U_I}\!: 
(S^1)^I_{\bu}\!\times\!U_I\lra U_I, &\qquad
\mu_I\!=\!r_I\!\circ\!\mu|_{U_I}\!:U_I\lra\ft_{I;\bu}^*,
\end{split}\EE
where $r_I\!:\ft_{N;\bu}^*\!\lra\!\ft_{I;\bu}^*$ is the restriction homomorphism.
We call $(X,\om,\phi,\mu)$ \sf{regular} if
the $\phi_I$-action of~$(S^1)^I_{\bu}$ on $\mu_I^{-1}(0)$ is free
for every $I\!\in\!\cP^*(N)$.

\begin{lmm}\label{THamConf_lmm}
Let $N\!\in\!\Z^+$ and $(X,\om,\phi,\mu)$ be a Hamiltonian $(S^1)^N_{\bu}$-manifold.
The~tuple
\BE{THamConf_e}   \sC_{\phi,\mu} \equiv\big(U_I,\phi_I,\mu_I\big)_{I\in\cP^*(N)} \EE
defined by~\eref{UIdfn_e} is a maximal $N$-fold Hamiltonian configuration.
It is a cutting configuration if and only if $(X,\om,\phi,\mu)$ is regular.
\end{lmm}

\begin{proof}
Since $U_{[N]}\!=\!X$ and $\mu$ is continuous, $\{U_I\}_{I\in\cP^*(N)}$ is an open cover of~$X$.
Since~$\mu$ is $\phi$-invariant, each $U_I$ is preserved by the $\phi$-action and
thus by its restriction to~$(S^1)^I_{\bu}$.
It follows that $\phi_I$ is a Hamiltonian $(S^1)^I_{\bu}$-action on~$U_I$ with moment map~$\mu_I$.
It is immediate from~\eref{UIdfn_e} that the tuple~\eref{THamConf_e} satisfies 
\ref{IJinter_it}-\ref{UIJpos_it} in Definition~\ref{SympCut_dfn0} and~\eref{IJpos_e}.
This establishes the first claim.
The regularity conditions on $(X,\om,\phi,\mu)$ are the $I_0\!=\!I$ cases of the conditions
of Definition~\ref{SympCut_dfn1} for a Hamiltonian configuration to be a cutting configuration.
On the other hand, \eref{THamConf_e} is a maximal Hamiltonian configuration
and so the domain of~$\mu_{I_0;I}$ in~\eref{hdfn_e} is an open subspace of~$U_{I_0}$.
Furthermore, $\mu_{I_0;I_0}\!=\!\mu_{I_0}$.
By~\eref{UIdfn_e},  
$\mu_{I_0;I}$ and the restriction of the $\phi_I$-action to~$(S^1)^{I_0}_{\bu}$
are the restrictions of~$\mu_{I_0}$ and the $\phi_{I_0}$-action to $\Dom(\mu_{I_0;I})$.
This establishes the second claim.
\end{proof}

\begin{prp}\label{THamConf_prp}
Suppose $N\!\in\!\Z^+$, $(X,\om,\phi,\mu)$ is a regular Hamiltonian $(S^1)^N_{\bu}$-manifold,
 $(\cZ,\om_{\cZ})$ is the symplectic manifold determined by~\eref{THamConf_e} via 
the construction of Section~\ref{Thm12Pf_sec}, and 
\hbox{$q_{\eset}\!:X\!\lra\!X_{\eset}$}
is the corresponding surjection.
There are then natural continuous maps
\BE{picZCdfn_e} \pi\!:\cZ\lra\C  \qquad\hbox{and}\qquad
F\!: \R^{\ge0}\!\times\!(S^1)^N\!\times\!X\lra\cZ\EE
so that $\pi$ is a one-parameter family of smoothings of~$X_{\eset}$
representing the germ of deformation equivalence classes of Theorem~\ref{SympCut_thm3}
determined by~\eref{THamConf_e}, $F$ is smooth outside of $F^{-1}(X_{\prt})$, 
\BE{THamConf_e1a} 
F(0,\id,x)=q_{\eset}(x) ~~\forall\,x\!\in\!X, \quad
\pi\big(F(r,g,x)\big)=r\,\rho_{\bu}(g)
~~\forall~(r,g,x)\in \R^{\ge0}\!\times\!(S^1)^N\!\!\times\!X,\EE
with $\rho_{\bu}$ given by~\eref{S1Ndiag_e}, and 
\BE{THamConf_e1d}F_{r,g}\!: (X,\om)\lra
\big(\pi^{-1}\big(r\rho_{\bu}(g)\big),\om_{\cZ}\big|_{\pi^{-1}(r\rho_{\bu}(g))}\big),
\quad x\lra F(r,g,x),\EE
is a symplectomorphism for every $(r,g)\!\in\!\R^+\!\times\!(S^1)^N$.
Every Hamiltonian $\T$-pair $(\phi_{\T},\mu_{\T})$ for  $(X,\om)$
compatible with~$(\phi,\mu)$
determines a Hamiltonian $\T$-pair $(\phi_{\cZ;\T},\mu_{\cZ;\T})$ 
for~$(\cZ,\om_{\cZ})$ so that $(\phi_{\cZ;\T},\mu_{\cZ;\T})$ is 
compatible with the  Hamiltonian $(S^1)^N_{\bu}$-pair $(\phi_{\cZ},\mu_{\cZ})$
for $(\cZ,\om_{\cZ})$ determined by~$(\phi,\mu)$
and intertwined with~$(\phi_{\T},\mu_{\T})$ by~$F$.
The Hamiltonian configuration~\eref{cZSympCutDfn_e} for~$(\cZ,\om_{\cZ})$
determined by~\eref{THamConf_e} is~$\sC_{\phi_{\cZ},\mu_{\cZ}}$.
\end{prp}

\begin{proof}
Since $U_{[N]}\!=\!X$ in this case,
\BE{THamConf_e3}\begin{split}
\big(\cZ,\om_{\cZ}\big)=\big(\cZ_{[N]}^{\circ},\vp_{[N]}\big), &\qquad
(X_i,\om_i)=\big(\cZ_{\{i\};[N]}^{\circ},\vp_{\{i\};[N]}\big)
\quad\forall\,i\!\in\![N], \\
(X_I,\om_I)&=\big(\cZ_{I;[N]}^{\circ},\vp_{I;[N]}\big)
\quad\forall\,I\!\in\!\cP^*(N);
\end{split}\EE
see~\eref{cZdfn_e}, \eref{Xidfn_e}, and~\eref{cZI0dfn_e}.
The $I\!=\![N]$ cases of~\eref{wtphiIdfn_e} and~\eref{wtcZIdfn_e} become 
\begin{gather}
\label{THamConf_e5a}
\wt\phi\!\equiv\!\wt\phi_{[N]}:(S^1)^N_{\bu}\!\times\!\big(X\!\times\!\C^N\big) \lra X\!\times\!\C^N, 
\quad \wt\phi(g;x,z)=\big(\phi(g;x),\phi_{\C^N}(g^{-1};z)\big),\\
\label{THamConf_e5b}
\wt\cZ\!\equiv\!\wt\cZ_{[N]}^{\circ}
=\big\{\big(x,(z_j)_{j\in[N]}\big)\!\in\!X\!\times\!\C^N\!\!:
(\mu(x))_i\!-\!\frac12|z_i|^2\!=\!(\mu(x))_j\!-\!\frac12|z_j|^2
~\forall\,i,j\!\in\![N]\big\}.
\end{gather}
By~\eref{THamConf_e3}, \eref{wtcZiIdfn_e}, and~\eref{SympRed_e4},  the symplectic manifold
$(\cZ,\om_{\cZ})$ of Corollary~\ref{cZtopol_crl} is given~by
\BE{THamConf_e6}\cZ=\wt\cZ/\wt\phi, \qquad 
q^*\om_{\cZ}=\big(\pi_1^*\om\!+\!\pi_2^*\om_{\C^N}\big)\big|_{\wt\cZ},\EE
where $q\!:\wt\cZ\!\lra\!\cZ$ is the quotient map.
The manifolds $(X_i,\om_i)$  and their submanifolds~$(X_I,\om_I)$ are 
the symplectic submanifolds of~$\cZ$ with $z_i\!=\!0$ and $z_j\!=\!0$ for all $j\!\in\!I$, 
respectively.
They are the images under~$q$ of the subspaces~$\wt{X}_i$
and~$\wt{X}_I$ of~$\wt\cZ$ described in the same way.
The map~$q_{\eset}$ in~\eref{qesetdfn_e} is now given~by
\BE{UIgeTor_e}\begin{split}
q_{\eset}\!:U_i^{\le}&\!=\!\big\{x\!\in\!X\!: (\mu(x))_i\!\le\!(\mu(x))_j~\forall\,j\!\in\![N]\big\}
\lra X_i\subset X_{\eset},\\
&q_{\eset}(x)=
q\Big(x,\Big(\sqrt{2(\mu(x))_{ij}}\,\Big)_{\!j\in[N]}\Big),
\end{split}\EE
for each $i\!\in\![N]$, with $(\mu(x))_{ij}$ as in~\eref{hijdfn_e}.\\

\noindent
A smooth map $\pi\!:\cZ\!\lra\!\C$ as in Section~\ref{SympNBpf_subs}
is determined by a tuple $(f_I)_{I\in\cP^*(N)}$
of smooth functions as in Lemma~\ref{fIcoll_lmm}.
Since $U_{\cZ;[N]}\!=\!\cZ$ in this case, such a tuple can be constructed without shrinking
the original cover.
For $I\!\in\!\cP^*(N)$, $i\!\in\!I$, and $[x,(z_j)_{j\in I}]\!\in\!\cZ_I^{\circ}$,
define
$$f_I\big(\big[x,(z_j)_{j\in I}\big]\big)
=\prod_{j\in [N]-I}\!\!\!\!\!\!\sqrt{2(\mu(x))_{ij}\!+\!|z_i|^2}>0\,;$$
the inequality holds by~\eref{UIdfn_e}.
By~\eref{THamConf_e5b}, the function~$f_I$ does not depend on~$i\!\in\!I$.
The associated projection map is~then given~by
\BE{THamConf_e15b}\pi\big(\big[x,(z_i)_{i\in [N]}\big]\big)
=f_{[N]}\big(\big[x,(z_i)_{i\in[N]}\big]\big)
\!\!\prod_{i\in[N]}\!\!\!z_i
=\prod_{i\in[N]}\!\!\!z_i\,.\EE
The proofs of Lemmas~\ref{fIsub_lmm} and~\ref{fIgraph_lmm} with $I\!=\![N]$ and
$\wt{f}_{[N]}\!=\!1$ show that $\wt{W}_{[N]}\!=\!\wt\cZ_{[N]}^{\circ}$ satisfies the conditions
stated in these lemmas.
Thus, $\wt{W}_{[N]}\!=\!\wt\cZ_{[N]}^{\circ}$ also satisfies the condition of
Corollary~\ref{fIgraph_crl}. 
In the proof of Lemma~\ref{cZfamsm_lmm}, $W_{[N]}$ then becomes $q_{\cZ}(\cZ_{[N]}^{\circ})$.
From~\eref{THamConf_e3}, we  conclude~that 
$$\cZ'=q_{\cZ}(\cZ_{[N]}^{\circ})=\cZ\,,$$
i.e.~\eref{THamConf_e15b} defines a one-parameter family~$\pi$ 
of smoothings of $X_{\eset}\!=\!\pi^{-1}(0)$ without shrinking~$\cZ$.
This conclusion also follows from~\eref{THamConf_e1a} and~\eref{THamConf_e1d}.\\

\noindent
Fix $i\!\in\![N]$ and define
\BE{vridfn_E}\vr_i\!:X\lra\R^{\ge0}, \qquad
\vr_i(x)=\max\!\big\{\!-\!2(\mu(x))_{ij}\!:j\!\in\![N]\big\}.\EE
By the same reasoning as in the proof of Lemma~\ref{fIgraph_lmm}, 
the~equation 
\BE{THamConf_e18}\prod_{j\in[N]}\!\!\!\big(2(\mu(x))_{ij}\!+\!\vr\big)=r^2\EE
has a unique solution $\vr\!=\!\vr_i(r,x)$ in $[\vr_i(x),\i)$ for each $r\!\in\!\R^{\ge0}$.
For $r\!\in\!\R^+$, it lies in $(\vr_i(x),\i)$ and depends smoothly on~$(r,x)$.
For $x\!\in\!X$ such that the maximum in~\eref{vridfn_E} is achieved
at a unique value $j\!\in\![N]$,
the function 
$$(r,x')\lra \sqrt{2(\mu(x'))_{ij}\!+\!\vr_i(r,x')}$$
is smooth around $(x,0)$.
Define~$F$ in~\eref{picZCdfn_e}~by
\BE{THamConf_e19}
F\big(r,(\ne^{\fI\th_j})_{j\in[N]},x\big)=
q\Big(x,\big(\ne^{\fI\th_j}\!\sqrt{2(\mu(x))_{ij}\!+\!\vr_i(r,x)}\big)_{\!j\in[N]}\Big).\EE
This function is independent of the choice of $i\!\in\![N]$.
It is continuous everywhere and smooth outside of the points $(0,g,x)$ such that 
$$\big|\big\{j\!\in\![N]\!: \vr_i(x)\!=\!-2(\mu(x))_{ij}\big\}\big|\ge2,$$
i.e.~on the complement of $F^{-1}(X_{\prt})$.
It restricts to~\eref{UIgeTor_e} over $\{(0,\id)\}\!\times\!X$.
By~\eref{THamConf_e15b}, \eref{S1Ndiag_e}, and \eref{THamConf_e18},
$F$~satisfies the second property in~\eref{THamConf_e1a} as~well.
By the same reasoning as in the proof of Corollary~\ref{SympNB_crl2},
each map~\eref{THamConf_e1d} is a symplectomorphism.\\

\noindent
Let $(\phi_{\T},\mu_{\T})$ be a Hamiltonian $\T$-pair for $(X,\om)$
compatible with~$(\phi,\mu)$.
We define a Hamiltonian $\T$-pair for $(X\!\times\!\C^N,\pi_1^*\om\!+\!\pi_2^*\om_{\C^N})$ by 
\begin{alignat}{2}
\label{THamConf_e7a} 
\wt\phi_{\cZ;\T}\!:\T\!\times\!\big(X\!\times\!\C^N\big)&\lra X\!\times\!\C^N, 
&\qquad \wt\phi_{\cZ;\T}(g;x,z)&=\big(\phi_{\T}(g;x),z\big),\\
\label{THamConf_e7b} 
\wt\mu_{\cZ;\T}\!:X\!\times\!\C^N&\lra\ft^*, &\qquad
\wt\mu_{\cZ;\T}(x,z)&=\mu_{\T}(x)\,.
\end{alignat}
Since the $\phi_{\T}$-action preserves~$\mu$ and commutes with the $\phi$-action, 
the action~\eref{THamConf_e7a} restricts to an action on~\eref{THamConf_e5b}
and descends to a $\T$-action 
$$ \phi_{\cZ;\T}\!:\T\!\times\!\cZ\lra \cZ$$
on~$\cZ$. 
Since the moment map~$\mu_{\T}$ is $\phi$-invariant, \eref{THamConf_e7b} descends to a smooth~map
$$\mu_{\cZ;\T}\!:\cZ\lra \ft^*.$$
By~\eref{THamConf_e6}, $(\phi_{\cZ;\T},\mu_{\cZ;\T})$ is a Hamiltonian $\T$-pair
for~$(\cZ,\om_{\cZ})$.
Since the actions of~$\phi$ and~$\phi_{\T}$ commute,
$\mu_{\T}$ is $\phi$-invariant, and  $\mu$ is $\phi_{\T}$-invariant,
\eref{THamConf_e7a} and~\eref{THamConf_e7b} imply that 
$(\phi_{\cZ;\T},\mu_{\cZ;\T})$ is  compatible $(\phi_{\cZ},\mu_{\cZ})$.
By~\eref{THamConf_e19}, \eref{THamConf_e7a}, and~\eref{THamConf_e7b},
\BE{THamConf_e1b}
\phi_{\cZ;\T}\big(g';F(r,g,x)\big)=F\big(r,g,\phi_{\T}(g';x)\big), \qquad
\mu_{\cZ;\T}\big(F(r,g,x)\big)=\mu_{\T}(x),\EE
i.e.~$F$ intertwines $(\phi_{\T},\mu_{\T})$ and $(\phi_{\cZ;\T},\mu_{\cZ;\T})$.\\

\noindent
It remains to establish the last claim.
Since $U_{\cZ;I}\!=\!q_{\cZ}(q_{\cZ;I}(\wt\cZ_I^{\circ}))$,
\eref{THamConf_e3}  and~\eref{cZdfn_e} give
$$U_{\cZ;I}=q\big(\wt\cZ_{[N],I}^{\circ}\big)
=q\big(\wt\cZ\!\cap\!(U_I\!\times\!\C^N)\big).$$
Combining this with the first equation in~\eref{UIdfn_e} 
and~\eref{THamConf_e7b} for $\mu_{\T}\!=\!\mu$, 
we~obtain 
\BE{THamConf_e11a}\begin{split}
U_{\cZ;I}&=\big\{[x,z]\!\in\!\cZ\!:(\mu(x))_i\!<\!(\mu(x))_j
~\forall\,i\!\in\!I,~j\!\in\![N]\!-\!I\big\}\\
&=\big\{[x,z]\!\in\!\cZ\!:\,\big(\mu_{\cZ}([x,z])\!\big)_{\!i}\!<\!\big(\mu_{\cZ}([x,z])\!\big)_{\!j}
~\forall\,i\!\in\!I,~j\!\in\![N]\!-\!I\big\}.
\end{split}\EE
By~\eref{wtcZphiiIdfn_e}, the second equation in~\eref{UIdfn_e}, 
and~\eref{THamConf_e7a} for $\phi_{\T}\!=\!\phi$,
\BE{THamConf_e11b}\phi_{\cZ;I}=\phi_{\cZ}\big|_{(S^1)^I_{\bu}\times U_{\cZ;I}}.\EE
By~\eref{wtcZhiIdfn_e}, the third equation in~\eref{UIdfn_e},  
and~\eref{THamConf_e7b} for $\mu_{\T}\!=\!\mu$,
\BE{THamConf_e11c}\mu_{\cZ;I}=r_I\!\circ\!\mu_{\cZ}\big|_{U_{\cZ;I}}.\EE
By~\eref{THamConf_e11a}-\eref{THamConf_e11c}, 
$\sC_{\cZ}\!=\!\sC_{\phi_{\cZ},\mu_{\cZ}}$.
\end{proof}

\noindent
By~\eref{THamConf_e1a} and the first property in~\eref{THamConf_e1b}, 
$$\phi_{\cZ;\T}(\T\!\times\!X_i)=X_i~~\forall\,i\!\in\![N], \quad
\pi\big(\phi_{\cZ;\T}(g;y)\big)=\pi(y).$$
By~\eref{THamConf_e1d} and~\eref{THamConf_e1b}, 
the restriction~\eref{THamConf_e1d} of $F$ induces an isomorphism
$$F_{r,g}\!: \big(X,\om,\phi_{\T},\mu_{\T}\big)\lra  
\big(\pi^{-1}(\la),\om_{\cZ}|_{\pi^{-1}(\la)},
\phi_{\cZ;\T}|_{\T\times \pi^{-1}(\la)},\mu_{\cZ;\T}|_{\T\times \pi^{-1}(\la)}\big)$$
whenever $\la\!\equiv\!r\rho_{\bu}(g)\neq0$.
Since $\sC_{\cZ}\!=\!\sC_{\phi_{\cZ},\mu_{\cZ}}$, 
the $(\phi_{\T},\mu_{\T})\!=\!(\phi,\mu)$ case of this statement implies
that $F_{r,g}$ identifies the cutting configuration~\eref{THamConf_e} for $(X,\om)$
with the restriction of the induced Hamiltonian configuration~$\sC_{\cZ}$
to~$\pi^{-1}(\la)$.

\subsection{A local example}
\label{BasicEg_subs}

\noindent
For $N\!\in\!\Z^+$, let
$$\C^N_0=\bigcup_{i\in[N]}\!\!\C^N_i \qquad\hbox{and}\qquad
 \C^N_{\prt}=\bigcup_{\begin{subarray}{c}I\in\cP^*(N)\\ |I|=2\end{subarray}}\!\!\!\!\!\!\C^N_I$$
be the union of the coordinate hyperplanes and 
the union of  the codimension~2 coordinate subspaces, respectively.
Thus, $\C^N_0$ and~$\C^N_{\prt}$ are the SC symplectic variety and its singular locus
associated to the SC symplectic configuration
\BE{BasicSCC_e} 
\X_{\C^N}\equiv \big((\C^N_I)_{I\in\cP^*(N)},(\om_{\C^N}|_{\C^N_i})_{i\in[N]}\big)\EE
as in~\eref{Xesetdfn_e} and~\eref{Xprtdfn_e}. 
Let 
\BE{BasicSCCsm_e} 
\pi_{\C^N}\!: \C^N\lra\C, \qquad 
\pi_{\C^N}(x_1,\ldots,x_N)= x_1\!\ldots\!x_N\,.\EE
The tuple $(\C^N,\om_{\C^N},\pi_{\C^N})$ is then a one-parameter family of smoothings
of the SC symplectic variety~$\C^N_0$ in the sense of
 the sentence after Definition~\ref{SimpFibr_dfn}.\\

\noindent
Under the identifications~\eref{ftNnu_e}, the restriction homomorphism
$$r_{\bu}\!:\ft_N^*\lra\ft_{N;\bu}^*
=\R^N\big/\big\{(a,\ldots,a)\!\in\!\R^N\!:\,a\!\in\!\R\big\},
\qquad \eta\lra\eta|_{\ft_{N;\bu}},$$
is the quotient map.
Let  $(\phi_{\C^N},\mu_{\C^N})$ be the standard Hamiltonian pair for $(\C^N,\om_{\C^N})$
given by~\eref{phimustan_e}.
The maximal $N$-fold Hamiltonian configuration 
\BE{SympCutDfnCN_e}
\sC_{\C^N}\equiv\sC_{\phi_{\C^N}|_{(S^1)^N_{\bu}\times\C^N},r_{\bu}\circ\mu_{\C^N}}\equiv 
\big(U_{\C^N;I},\phi_{\C^N;I},\mu_{\C^N;I} \big)_{I\in\cP^*(N)}\EE 
determined by this pair via~\eref{UIdfn_e} is  given~by
\begin{gather*} 
U_{\C^N;I}=\big\{(x_1,\ldots,x_N)\!\in\!\C^N\!:\,|x_i|\!<\!|x_j|~\forall\,i\!\in\!I,\,
j\!\in\![N]\!-\!I\big\},\\ 
\phi_{\C^N;I}\!: 
(S^1)^I_{\bu}\!\times\!U_I\lra U_I,\quad
\mu_{\C^N;I}\big((x_i)_{i\in[N]}\big)=\frac12
\big[\big(|x_i|^2\big)_{i\in I}\big]\in\ft_{I;\bu}^*\,.
\end{gather*}
The associated $(S^1)$-actions and Hamiltonians in~\eref{phiijhij_e} are
\begin{alignat*}{2}
\phi_{\C^N;ij}\!: S^1\!\times\!\C^N&\lra\C^N, &\quad
\big(\phi_{\C^N;ij}(\ne^{\fI\th};x_1,\ldots,x_N)\big)_k&=
\begin{cases}
\ne^{-\fs_{ij}\fI\th}x_k,&\hbox{if}~k\!=\!i;\\
\ne^{\fs_{ij}\fI\th}x_k,&\hbox{if}~k\!=\!j;\\
x_k,&\hbox{otherwise};
\end{cases}\\
h_{\C^N;ij}\!: \C^N&\lra\R, &\quad 
h_{\C^N;ij}(x_1,\ldots,x_N)&=\frac{\fs_{ij}}{2}\big(|x_j|^2\!-\!|x_i|^2\big).
\end{alignat*}
The Hamiltonian configuration~\eref{SympCutDfnCN_e} is not a cutting configuration,
as it does not satisfy the  additional conditions of Definition~\ref{SympCut_dfn1} if $N\!\ge\!2$.
On the other hand, the restriction of~$\sC_{\C^N}$ to $\C^N\!-\!\C^N_{\prt}$ satisfies
these  conditions and thus is an $N$-fold cutting configuration.\\

\noindent
Fix  $\la\!\in\!\C^*$. Let
\BE{Xladfn_e}
(S^1)^N_{\la}=\rho_{\bu}^{-1}\big(\la/|\la|\big), \quad
X=\big\{(x_1,\ldots,x_N)\!\in\!\C^N\!:\,x_1\!\ldots\!x_N\!=\!\la\big\},
\quad \om=\om_{\C^N}\big|_X.\EE
Since $X$ is preserved by the restriction of the $\phi_{\C^N}$-action
to~$(S^1)^N_{\bu}$, 
\BE{BasicEgConf_e}\sC\equiv \sC_{\C^N}\big|_X
\equiv \sC_{\phi_{\C^N}|_{(S^1)^N_{\bu}\times X},r_{\bu}\circ\mu_{\C^N}|_X}
\equiv \big(U_I,\phi_I,\mu_I\big)_{I\in\cP^*(N)}\EE
is an $N$-fold Hamiltonian configuration.
Since the full $\phi_{\C^N}$-action  does not preserve~$X$,
the present situation is a special case of the setting of 
Lemma~\ref{THamConf_lmm} and not of Lemma~\ref{THamConf_lmm2}.

\begin{lmm}\label{BasicEg_lmm}
The tuple~$\sC$ in~\eref{BasicEgConf_e} is a maximal $N$-fold cutting configuration
for~$(X,\om)$.
\end{lmm}

\begin{proof}
By the first statement of Lemma~\ref{THamConf_lmm},
\eref{BasicEgConf_e} is a maximal $N$-fold Hamiltonian configuration for~$(X,\om)$.
Since the $\phi_{\C^N}$-action is free on $\C^N\!-\!\C^N_0$,
the $\phi_I$-action is free on~$X$ for each $I\!\in\!\cP^*(N)$.
By the second statement of Lemma~\ref{THamConf_lmm},
\eref{BasicEgConf_e} is thus a cutting configuration.
\end{proof}

\begin{prp}\label{BasicEg_prp}
Suppose $(\cZ,\om_{\cZ})$ is the symplectic manifold
determined by~\eref{BasicEgConf_e} via the construction of Section~\ref{Thm12Pf_sec}
and $\pi$, $F$, $\phi_{\cZ}$, and~$\mu_{\cZ}$ are the associated maps provided
by Proposition~\ref{THamConf_prp}.
There is a natural smooth~map
\BE{BasicEg_e1} 
\wch{F}\!: (S^1)^N_{\la}\!\times\!\C^N\lra\cZ\EE
such that 
\BE{BasicEg_e1d}\wch{F}_g\!: \big(\C^N,\om_{\C^N}\big)\lra \big(\cZ,\om_{\cZ}\big),
\qquad x\lra \wch{F}(g,x),\EE
is a symplectomorphism for every $g\!\in\!(S^1)^N_{\la}$, 
\begin{gather}
\label{BasicEg_e1a}
\wch{F}\big((S^1)^N_{\la}\!\times\!\C_i^N\big)=X_i~\forall\,i\!\in\![N], ~~
\pi\big(\wch{F}(g,z)\big)=\pi_{\C^N}(z), ~~
\wch{F}|_{(S^1)^N_{\la}\times X}=F|_{\{|\la|\}\times(S^1)^N_{\la}\times X},\\
\label{BasicEg_e1b}
\phi_{\cZ}\big(g';\wch{F}(g,z)\big)=\wch{F}\big(g,\phi_{\C^N}(g';z)\big), \qquad
\mu_{\cZ}\big(\wch{F}(g,z)\big)=r_{\bu}\big(\mu_{\C^N}(z)\big).
\end{gather}
\end{prp}

\begin{proof}
The action~\eref{THamConf_e5a} in this case reduces~to
\BE{whphiieg_e}
\wt\phi\!:
(S^1)^N_{\bu}\times \big(X\!\times\!\C^N\big) \lra X\!\times\!\C^N, 
\quad \wt\phi(g;x,z)=\big(\phi_{\C^N}(g;x),\phi_{\C^N}(g^{-1};z)\big).\EE
The symplectic manifold $(\cZ,\om_{\cZ})$ of Corollary~\ref{cZtopol_crl} is
described by~\eref{THamConf_e5b} and~\eref{THamConf_e6}, which~become
\begin{gather}
\notag
\wt\cZ=\big\{\!\big((x_j)_{j\in[N]},(z_j)_{j\in[N]}\big)\!\in\!\C^N\!\times\!\C^N\!\!:\!
\prod_{i\in[N]}\!\!x_i\!=\!\la,\,|x_i|^2\!-\!|z_i|^2\!=\!|x_j|^2\!-\!|z_j|^2
~\forall\,i,j\!\in\![N]\big\},\\
\label{cZdfneg_e}
\cZ=\wt\cZ/\wt\phi, \qquad 
q^*\om_{\cZ}=\big(\pi_1^*\om_{\C^N}\!+\!\pi_2^*\om_{\C^N}\big)\big|_{\wt\cZ}\,,
\end{gather}
where $q\!:\wt\cZ\!\lra\!\cZ$ is the quotient map.
The manifolds $(X_i,\om_i)$  and their submanifolds~$(X_I,\om_I)$ are again 
the symplectic submanifolds of~$\cZ$ with $z_i\!=\!0$ and $z_j\!=\!0$ for all $j\!\in\!I$, 
respectively.
They are the images under~$q$ of the subspaces~$\wt{X}_i$
and~$\wt{X}_I$ of~$\wt\cZ$ described in the same way.\\

\noindent
Fix $i\!\in\![N]$ and define
\begin{gather*}
\vr_i\!:\C^N\lra\R^{\ge0}, \qquad
\vr_i(z)=\max\!\big\{\!-\!2\big(\mu_{\C^N}(z)\!\big)_{\!ij}\!:j\!\in\![N]\big\},\\
\hbox{where}\qquad
\big(\mu_{\C^N}(z)\!\big)_{\!ij}=\frac{1}{2}\big(|z_j|^2\!-\!|z_i|^2\big)
~~\forall~z\!\equiv\!(z_j)_{j\in[N]}\in\C^N\,.
\end{gather*}
For each $z\!\in\!\C^N$, let $\vr_{\la;i}(z)\!\in\!(\vr_i(z),\i)$ be 
the unique solution of the equation
\BE{BasicEg_e18}\prod_{j\in[N]}\!\!\!\big(2(\mu_{\C^N}(z))_{ij}\!+\!\vr\big)=|\la|^2\EE
in~$\vr$; this is a special case of the equation~\eref{THamConf_e18}.
Define~$\wch{F}$ in~\eref{BasicEg_e1}~by
\BE{BasicEg_e19}
\wch{F}\big((\ne^{\fI\th_j})_{j\in[N]},(z_j)_{j\in[N]}\big)=
q\Big(\big(\ne^{\fI\th_j}\!
\sqrt{|z_j|^2\!-\!|z_i|^2\!+\!\vr_{\la;i}((z_k)_{k\in[N]})}\big)_{\!j\in[N]},
(z_j)_{j\in[N]}\Big).\EE
This function is independent of the choice of $i\!\in\![N]$.\\

\noindent
By the reasoning in the proof of Corollary~\ref{SympNB_crl2} with
the $x$ and~$z$ components of~$\cZ$ interchanged,
each map~\eref{BasicEg_e1d} is a symplectomorphism.
Since $X_i\!=\!q(\wt{X}_i)$, $\wch{F}$ satisfies the first property in~\eref{BasicEg_e1a}.
By~\eref{THamConf_e15b} and~\eref{BasicSCCsm_e},
it also satisfies the second property in~\eref{BasicEg_e1a}.
By~\eref{THamConf_e7b} for \hbox{$\mu_{\T}\!=\!r_{\bu}\!\circ\!\mu_{\C^N}$}, \eref{BasicEg_e19},
and~\eref{phimustan_e},  
\begin{equation*}\begin{split}
&\mu_{\cZ}\big(\wch{F}\big((\ne^{\fI\th_j})_{j\in[N]},(z_j)_{j\in[N]}\big)\big)
=\big\{r_{\bu}\!\circ\!\mu_{\C^N}\big\}\Big(\!\big(\ne^{\fI\th_j}\!
\sqrt{|z_j|^2\!-\!|z_i|^2\!+\!\vr_{\la;i}((z_k)_{k\in[N]})}\big)_{\!j\in[N]}\Big)\\
&\quad=\frac12 r_{\bu}\Big(\!\big(|z_j|^2\!-\!|z_i|^2\!+\!\vr_{\la;i}((z_k)_{k\in[N]})\big)_{\!j\in[N]}
\Big)
=\frac12 r_{\bu}\big(\big(|z_j|^2\big)_{\!j\in[N]}\big)
=r_{\bu}\big(\mu_{\C^N}\big((z_j)_{j\in[N]}\big)\big).
\end{split}\end{equation*}
This establishes the second claim in~\eref{BasicEg_e1b}.
Since each map~\eref{BasicEg_e1d} is a symplectomorphism, 
the second claim in~\eref{BasicEg_e1b} implies the first one.\\

\noindent
Suppose $(\ne^{\fI\th_j})_{j\in[N]}\!\in\!(S^1)^N_{\la}$ and 
$(z_j)_{j\in[N]}\!\in\!X$. Let
$$\big(\ne^{\fI\th_j'}\big)_{j\in[N]} =
\big(\ne^{-\fI\th_j}z_j/|z_j|\big)_{j\in[N]} \in (S^1)^N_{\bu}\,.$$
By the uniqueness of the solution of~\eref{BasicEg_e18},
\BE{BasicEg_e15} |z_j|^2-|z_i|^2+\vr_{\la;i}\big((z_j\big)_{j\in[N]})
=|z_j|^2\qquad\forall\,j\!\in\![N].\EE
Combining this with~\eref{BasicEg_e19}, \eref{cZdfneg_e}, and~\eref{whphiieg_e},
we~obtain
\begin{equation*}\begin{split}
\wch{F}\big((\ne^{\fI\th_j})_{j\in[N]},(z_j)_{j\in[N]}\big)=
q\Big(\!\big(\ne^{\fI\th_j'}\ne^{\fI\th_j}\!|z_j|\big)_{\!j\in[N]},
\big(\ne^{-\fI\th_j'}z_j\big)_{j\in[N]}\Big)
=q\Big(\!(z_j)_{\!j\in[N]},\big(\ne^{\fI\th_j}|z_j|\big)_{j\in[N]}\Big).
\end{split}\end{equation*}
The last property in~\eref{BasicEg_e1a} now follows 
from~\eref{BasicEg_e15} and~\eref{THamConf_e19}.
\end{proof}

\noindent
By Proposition~\ref{BasicEg_prp}, the restriction~\eref{BasicEg_e1d} of $\wch{F}$ induces 
an isomorphism
$$\wch{F}_g\!: \big(\C^N,\om_{\C^N},\X(\sC),X_{\eset},\pi_{\C^N},
\phi_{\C^N}|_{(S^1)^N_{\bu}\times\C^N},r_{\bu}\!\circ\!\mu_{\C^N}\big)\lra  
\big(\cZ,\om_{\cZ},\X_{\C^N},\C^N_0,\pi,\phi_{\cZ},\mu_{\cZ}\big),$$
where $\X(\sC)$ is the SC symplectic configuration determined by~\eref{BasicEgConf_e}
and $X_{\eset}$ is the associated SC symplectic divisor.
By~\eref{SympCutDfnCN_e} and the last statement of Proposition~\ref{THamConf_prp}, 
this implies that $\wch{F}_g$ identifies the Hamiltonian configuration~\eref{SympCutDfnCN_e} 
for $(\C^N,\om_{\C^N})$ and the induced Hamiltonian configuration~$\sC_{\cZ}$
for~$(\cZ,\om_{\cZ})$.

\subsection{Further refinements}
\label{CmptcZ_subs}

\noindent
For a Hamiltonian $(S^1)^N$-manifold $(X,\om,\phi,\mu)$,
the global quotient description of the main constructions of this paper
provided by the proof of Proposition~\ref{THamConf_prp} gives rise 
to a Hamiltonian $S^1$-action on~$(\cZ,\om_{\cZ})$ which is free 
on the complement of the SC symplectic divisor $X_{\eset}\!\subset\!\cZ$;
see Lemma~\ref{THamConf_lmm2}.
If in addition~$X$ is compact, this action can be used to ``cut" off a precompact
neighborhood of~$X_{\eset}$ in~$\cZ$ to form a one-parameter family of 
smoothings of~$X_{\eset}$ with compact total space~$(\wh\cZ_a,\om_{\wh\cZ;a})$
and projection map~\eref{whpiwhcZa_e} 
which is equivariant with respect to the induced $S^1$-action on~$\wh\cZ_a$ and
the standard $S^1$-action on $\P^1$ given~by
\BE{S1P1dfn_e}
\phi_{\P^1}\!:S^1\!\times\!\P^1\lra\P^1, \quad 
\phi_{\P^1}\big(\ne^{\fI\th};[w,z]\big)=\big[w,\ne^{\fI\th}z\big].\EE
For the purposes of Corollary~\ref{CmptHam_crl} below, let
\BE{whP10dfn_e}
\wh\P_0^1=\Big(\!\big[0,1)\!\times\!S^1\big)\!\sqcup\!\big(\P^1\!-\!\big\{[1,0]\big\}\big)\!\!\Big)
\!\Big/\!\!\!\sim, \quad
(0,1)\!\times\!S^1\!\ni\!\big(r,\ne^{\fI\th}\big)\sim
\big[1,r\ne^{\fI\th}\big]\in \P^1\!-\!\big\{[1,0]\big\}\,;\EE
this is the disk obtained  from~$\P^1$ by replacing $[1,0]$ with~$S^1$.
Let
$$q_{\P^1}\!: \wh\P_0^1\lra\P^1 $$
be the natural projection map restricting to the inclusion of $\P^1\!-\!\{[1,0]\}$.

\begin{lmm}\label{THamConf_lmm2}
Suppose $N\!\in\!\Z^+$,  $(X,\om,\phi,\mu)$ is a regular Hamiltonian $(S^1)^N$-manifold,
$(\cZ,\om_{\cZ})$ is the symplectic manifold
determined by~\eref{THamConf_e} via the construction of Section~\ref{Thm12Pf_sec},
and $\pi$ is the associated map provided 
by Proposition~\ref{THamConf_prp}.
There is then a natural Hamiltonian $S^1$-pair $(\phi_{\cZ;S^1},\mu_{\cZ;S^1})$ for 
$(\cZ,\om_{\cZ})$ such~that
\BE{THamConfC_e1}
\phi_{\cZ;S^1}(S^1\!\times\!X_i)=X_i~~\forall\,i\!\in\![N], \qquad
\pi\big(\phi_{\cZ;S^1}(\ne^{\fI\th};y)\big)=\ne^{\fI\th}\pi(y).\EE
If $(\phi_{\T},\mu_{\T})$ is a Hamiltonian $\T$-pair for $(X,\om)$
compatible with~$(\phi,\mu)$, 
then the associated Hamiltonian  $\T$-pair $(\phi_{\cZ;\T},\mu_{\cZ;\T})$
for~$(\cZ,\om_{\cZ})$ provided  by Proposition~\ref{THamConf_prp}
is compatible with~$(\phi_{\cZ;S^1},\mu_{\cZ;S^1})$.
\end{lmm}

\begin{proof}
We continue with the notation and setup of Proposition~\ref{THamConf_prp}.
In this case, the components $(\mu(x))_i\!\in\!\R$ of $\mu(x)\!\in\!\ft^*$ are well-defined.
For $i\!\in\![N]$, let 
$$\io_i\!: S^1\lra(S^1)^N$$
be the inclusion as the $i$-th component.
Define a Hamiltonian $S^1$-pair for $(X\!\times\!\C^N,\pi_1^*\om\!+\!\pi_2^*\om_{\C^N})$~by 
\begin{alignat}{2}
\label{THamConf_e27a} 
\wt\phi_{\cZ;S^1;i}\!:S^1\!\times\!\big(X\!\times\!\C^N\big)&\lra X\!\times\!\C^N, 
&\quad \wt\phi_{\cZ;S^1;i}\big(\ne^{\fI\th};x,z\big)&=
\big(\phi\big(\io_i(\ne^{-\fI\th});x\big),\phi_{\C^N}\big(\io_i(\ne^{\fI\th});z\big)\!\big),\\
\label{THamConf_e27b} 
\wt\mu_{\cZ;S^1;i}\!:X\!\times\!\C^N&\lra\R, &\quad
\wt\mu_{\cZ;S^1;i}(x,z)&=-\big(\mu(x)\!\big)_{\!i}\!+\!\frac12|z_i|^2\,.
\end{alignat}
By the same reasoning as below~\eref{THamConf_e7b}, 
this pair descends to a Hamiltonian $S^1$-pair
\BE{THamConf_e29} \phi_{\cZ;S^1}\!: S^1\!\times\!\cZ\lra \cZ, \qquad
\mu_{\cZ;S^1}\!: \cZ\lra \R,\EE
for~$(\cZ,\om_{\cZ})$. 
By~\eref{THamConf_e5b}, the restriction of~\eref{THamConf_e27b} to~$\wt\cZ$ is independent of
the choice of $i\!\in\![N]$.
Thus, so is the pair~\eref{THamConf_e29}.
Since \eref{THamConf_e27a} preserves the subspace~$\wt{X}_i$ of~$\wt\cZ$, 
$\phi_{\cZ;S^1}$ preserves the subspace~$X_i$ of~$\cZ$.
By~\eref{THamConf_e15b} and~\eref{THamConf_e27a}, 
\eref{THamConf_e29} satisfies the second property in~\eref{THamConfC_e1} as well.
If  $(\phi_{\T},\mu_{\T})$ is a Hamiltonian $\T$-pair for $(X,\om)$
compatible with~$(\phi,\mu)$,  then the actions of~$\wt\phi_{\cZ;S^1;i}$ and~$\wt\phi_{\cZ;\T}$ commute,
$\wt\mu_{\cZ;\T}$ is $\wt\phi_{\cZ;S^1;i}$-invariant, and  
$\wt\mu_{\cZ;S^1;i}$ is $\wt\phi_{\T;\cZ}$-invariant.
This implies the last claim.
\end{proof}

\noindent
In the setting of Proposition~\ref{THamConf_prp}, 
an extra Hamiltonian $S^1$-action on~$(\cZ,\om_{\cZ})$ can be obtained
via the diagonal action of~$S^1$ on the $\C^N$-component in~\eref{THamConf_e5b}.
However, the action of $(S^1)^{N+1}$ on~$\cZ$ obtained by combining this action
with~$\phi_{\cZ}$ is not effective even if the $(S^1)^N$-action~$\phi$
is effective.
In contrast,  the action of $(S^1)^{N+1}$ on~$\cZ$ obtained by 
combining the $S^1$-action of Lemma~\ref{THamConf_lmm2} with 
with~$\phi_{\cZ}$ is effective if  the $\phi$-action is effective;
see Theorem~\ref{toriccut_thm}.

\begin{crl}\label{CmptHam_crl}
Suppose $N\!\in\!\Z^+$,
$(X,\om,\phi,\mu)$ is a compact regular Hamiltonian $(S^1)^N$-manifold,
and $(\cZ,\om_{\cZ})$,  $X_{\eset}$, $q_{\eset}$, and $(\phi_{\cZ;S^1},\mu_{\cZ;S^1})$ 
are as in Proposition~\ref{THamConf_prp} and Lemma~\ref{THamConf_lmm2}.
For every $a\!\in\!\R$ sufficiently large, 
there exist a natural compact symplectic manifold $(\wh\cZ_a,\om_{\wh\cZ;a})$,
an open neighborhood~$\cZ_a'$ of $X_{\eset}$ in~$\cZ$,  
continuous surjections
\BE{CmptHam_e0} \wh\pi\!:\wh\cZ_a\lra\P^1, \qquad
f_a\!:\ov\cZ_a'\lra\wh\cZ_a, \qquad\hbox{and}\quad
\wh{F}_a\!: \wh\P_0^1\!\times\!X\lra\wh\cZ_a,\EE
and a Hamiltonian $S^1$-pair $(\phi_{\wh\cZ;S^1;a},\mu_{\wh\cZ;S^1;a})$ 
for $(\wh\cZ_a,\om_{\wh\cZ;a})$ so~that
$\wh\pi$ is a one-parameter family of smoothings of
the SC symplectic variety $\wh\cZ_{a;0}\!\equiv\!\wh\pi^{-1}([1,0])$,
$\wh{F}_a$ is smooth outside of~$\wh{F}_a^{-1}(f_a(X_{\prt}))$,
\BE{CmptHam_e1a} 
\wh{F}_a(0,1,x)=f_a\big(q_{\eset}(x)\big), \quad
\wh\pi\big(\wh{F}_a(w,x)\big)=q_{\P^1}(w), \quad
\wh\pi\big(\phi_{\wh\cZ;S^1;a}(\ne^{\fI\th};y)\big)=\ne^{\fI\th}\wh\pi(y),\EE
and the maps
\begin{gather}
\label{CmptHam_e1e}
f_a\!: \big(\cZ_a',\om_{\cZ},X_{\eset},\phi_{\cZ;S^1},\mu_{\cZ;S^1}\big)
\lra  \big(\wh\cZ_a\!-\!\wh\pi^{-1}([0,1]),\om_{\wh\cZ;a},\wh\cZ_{a;0},
\phi_{\wh\cZ;S^1;a},\mu_{\wh\cZ;S^1;a}\big),\\
\label{CmptHam_e1d}
\wh{F}_{a;w}\!: (X,\om)\lra
\big(\wh\pi^{-1}(w),\om_{\wh\cZ;a}\big|_{\wh\pi^{-1}(w)}\big),
~~x\lra\wh{F}_a(w,x),\quad\forall\,w\!\in\!\P^1\!-\!\{[1,0]\},
\end{gather}
are isomorphisms.
Every Hamiltonian $\T$-pair $(\phi_{\T},\mu_{\T})$ for  $(X,\om)$
compatible with~$(\phi,\mu)$
determines a Hamiltonian $\T$-pair $(\phi_{\wh\cZ;\T},\mu_{\wh\cZ;\T})$ 
for~$(\wh\cZ_a,\om_{\wh\cZ;a})$ so that 
$(\phi_{\wh\cZ;\T},\mu_{\wh\cZ;\T})$ is compatible with the pair
$(\phi_{\wh\cZ;S^1;a},\mu_{\wh\cZ;S^1;a})$ and
with the  Hamiltonian $(S^1)^N_{\bu}$-pair $(\phi_{\wh\cZ},\mu_{\wh\cZ})$
for $(\wh\cZ_a,\om_{\wh\cZ;a})$ determined by~$(\phi,\mu)$
and is intertwined with $(\phi_{\T},\mu_{\T})$ by~$\wh{F}_a$ 
and with the pair $(\phi_{\cZ;\T},\mu_{\cZ;\T})$ of Proposition~\ref{THamConf_prp}
by~$f_a$.
\end{crl}

\begin{proof} We split $(\cZ,\om_{\cZ})$ into a compact piece 
$(\wh\cZ,\om_{\wh\cZ})$ containing a neighborhood~$\cZ_a'$ of~$X_{\eset}$
and the ``infinite remainder".
We use the symplectic cut construction of~\cite{L} 
with the $S^1$-action~$\phi_{\cZ;S^1}$,
moment map~$\mu_{\cZ;S^1}$, and its value
\BE{CmptHam_e7} a>-\min_{i\in[N]}\min_{x\in X}\big(\mu(x)\!\big)_{\!i}\,.\EE
This corresponds to the construction of Proposition~\ref{THamConf_prp}
with $N\!=\!2$, 
\begin{alignat*}{2}
\phi\!=\!\phi_{S^1;\bu}:(S^1)^2_{\bu}\!\times\!\cZ&\lra\cZ,  &\qquad 
\phi_{S^1;\bu}\big(\big(\ne^{\fI\th},\ne^{-\fI\th}\big);y\big)&=\phi_{\cZ;S^1}\big(\ne^{\fI\th};y\big),\\
\mu\!=\!\mu_{S^1;\bu}\!:\cZ&\lra\ft_{2;\bu}^*,  &\qquad 
\mu_{S^1;\bu}(y)&=r_{\bu}\big(\mu_{\cZ;S^1}(y),a\!\big).
\end{alignat*}
By~\eref{THamConfC_e1}, the restriction of the $\phi_{S^1;\bu}$-action to $\cZ\!-\!X_{\eset}$
is free.
By~\eref{THamConf_e27b} and~\eref{CmptHam_e7}, $\mu_{S^1;\bu}^{-1}(0)$ is disjoint from~$X_{\eset}$
and so the restriction of the $\phi_{S^1;\bu}$-action to~$\mu_{S^1;\bu}^{-1}(0)$ is~free.
Thus,
$(\cZ,\om,\phi_{S^1;\bu},\mu_{S^1;\bu})$ is a regular Hamiltonian $(S^1)^2_{\bu}$-manifold.
By Lemma~\ref{THamConf_lmm} and Corollary~\ref{SympNB_crl2}, 
$(\cZ,\om,\phi_{S^1;\bu},\mu_{S^1;\bu})$ thus determines  symplectic manifolds
$(\wh\cZ_1,\om_{\wh\cZ;1})$ and $(\wh\cZ_2,\om_{\wh\cZ;2})$
with a common smooth symplectic divisor~$\wh\cZ_{12}$.
We denote $(\wh\cZ_1,\om_{\wh\cZ;1})$ and~$\wh\cZ_{12}$ by $(\wh\cZ_a,\om_{\wh\cZ;a})$
and~$\wh\cZ_{a;\i}$, respectively.\\

\noindent
By the proof of Proposition~\ref{THamConf_prp},
\begin{gather}
\notag
\wh\cZ_a=\wch\cZ_a\big/\!\!\sim, \quad
\wh\cZ_{a;\i}=\wch\cZ_{a;\i}\big/\!\!\sim \,\,\subset \wh\cZ_a, \qquad\hbox{where}\\
\label{CmptHam_e2a}
\wch\cZ_a=\big\{(y,w)\!\in\!\cZ\!\times\!\C\!:\mu_{\cZ;S^1}(y)\!=\!a\!-\!\frac12|w|^2\big\},
 \qquad
\wch\cZ_{a;\i}=\wch\cZ_a\!\cap\!\big(\cZ\!\times\!\{0\}\big), \\
\notag
(y,w)\sim\big(\phi_{\cZ;S^1}(\ne^{\fI\th};y),\ne^{\fI\th}w\big)\,.
\end{gather}
Let $\wh{q}\!:\wch\cZ_a\!\lra\!\wh\cZ_a$ be the quotient map.
The map~$f_a$ in~\eref{CmptHam_e0} is the restriction of the collapsing map~$q_{\eset}$
for this symplectic cut to the preimage of~$\wh\cZ_a$,
\BE{CmptHam_e2e}
f_a\!:\ov\cZ_a'\!\equiv\!\big\{y\!\in\!\cZ\!:\mu_{\cZ;S^1}(y)\!\le\!a\big\}\lra \wh\cZ_a, \qquad
f_a(y)=\wh{q}\Big(y,\sqrt{2\big(a\!-\!\mu_{\cZ;S^1}(y)\big)}\Big).\EE
By Corollary~\ref{SympNB_crl2}, its restriction 
$$f_a\!:\cZ_a'\!\equiv\!\big\{y\!\in\!\cZ\!:\mu_{\cZ;S^1}(y)\!<\!a\big\}
\lra \wh\cZ_a\!-\!\wh\cZ_{a;\i}$$
is a symplectomorphism with respect to the symplectic forms~$\om_{\cZ}$ and~$\om_{\wh\cZ;a}$.\\

\noindent
Let $\pi\!:\cZ\!\lra\!\C$ be as in Proposition~\ref{THamConf_prp} and define 
\BE{whpidfn_e}\wh\pi\!:\wh\cZ_a\lra\P^1, \qquad 
\wh\pi\big(\wh{q}(y,w)\big)=\big[w,\pi(y)\big].\EE
Since $\wh\cZ_{a;\i}$ is disjoint from $f_a(X_{\eset})\!=\!f_a(\pi^{-1}(0))$,
this map is well-defined.
Furthermore, 
\BE{THamConf_e4}\wh\cZ_{a;0}\equiv \wh\pi^{-1}\big([1,0]\big)=f_a\big(\pi^{-1}(0)\big)=f_a(X_{\eset}).\EE
Thus, $f_a$ identifies~$X_{\eset}$ with~$\wh\cZ_{a;0}$.
The facts that $\wh\pi$ is a submersion outside of~$\wh\cZ_{a;0}$ and  
$\om_{\wh\cZ;a}|_{\wh\pi^{-1}(\la)}$ is a symplectic form 
whenever $\la\!\!\neq\![1,0]$ follow from
the first statement in~\eref{CmptHam_e1a} and~\eref{CmptHam_e1d}.\\ 

\noindent
With the notation as in the proof of Lemma~\ref{THamConf_lmm2}, 
 define an $(S^1)^N$-action on $X\!\times\!\C^{N+1}$~by
$$\wt\phi_{N+1}\!:(S^1)^N\!\times\!\big(X\!\times\!\C^{N+1}\big)\lra X\!\times\!\C^{N+1},
\quad \wt\phi_{N+1}\big(\io_i(\ne^{\fI\th});x,z,w\big)
=\big(\wt\phi_{\cZ;S^1;i}(\ne^{\fI\th};x,z),\ne^{\fI\th}w\big).$$
By~\eref{THamConf_e5a}, \eref{THamConf_e5b}, and~\eref{CmptHam_e2a}, 
\BE{CmptHam_e8}\begin{split}
\wh\cZ_a&=\wt\cZ_{N+1;a}\big/\wt\phi_{N+1}, \qquad\hbox{where}\\
\wt\cZ_{N+1;a}=
\bigg\{\big(x,(z_i)_{i\in[N]},w\big)\!&\in\!X\!\times\!\C^{N+1}\!:
-\big(\mu(x)\!\big)_i\!+\!\frac12|z_i|^2=a\!-\!\frac12|w|^2~\forall\,i\!\in\![N]\bigg\}.
\end{split}\EE
The symplectic form $\om_{\wh\cZ;a}$ on~$\wh\cZ_a$
is determined by the condition
\BE{CmptHam_e6}\wh{q}_{N+1;a}^{\,*}\om_{\wh\cZ;a} = 
\big(\pi_1^*\om_X\!+\!\pi_2^*\om_{\C^{N+1}}\big)
\big|_{\wt\cZ_{N+1;a}}\,,\EE
where $\wh{q}_{N+1;a}\!:\wt\cZ_{N+1;a}\!\lra\!\wh\cZ_a$ is the quotient map.
By~\eref{THamConf_e4},
$$\wh\cZ_{a;0}=\wt\cZ_{N+1;a;0}\big/\wt\phi_{N+1}, \quad\hbox{where}~~
\wt\cZ_{N+1;a;0}=\bigcup_{i\in[N]}\!\!\!\big\{
\big(x,(z_j)_{j\in[N]},w\big)\!\in\!\wt\cZ_{N+1;a}\!:z_i\!=\!0\big\}.$$
In particular,
$$\wt\cZ_{N+1;a}^>\equiv \wt\cZ_{N+1;a}\cap\big(X\!\times\!(\R^+)^N\!\times\!\C\big)$$
is a slice for the $\wt\phi_{N+1}$-action on $\wt\cZ_{N+1;a}\!-\!\wt\cZ_{N+1;a;0}$
identified by~$\wh{q}_{N+1;a}$ with $\wh\cZ_a\!-\!\wh\cZ_{a;0}$.\\

\noindent
For each $x\!\in\!\C\!\times\!X$, let
\BE{CmptHam_e21a}\vr(x)=2\min\!\big\{a\!+\!(\mu(x))_i\!:i\!\in\![N]\big\}
\in\R^+;\EE
see~\eref{CmptHam_e7}.
For each $(w,x)\!\in\!\C\!\times\!X$, let $\vr(w,x)\!\in\![0,\vr(x)]$
be the unique solution of the~equation 
\BE{CmptHam_e21b}
|w|^2\!\!\prod_{i\in[N]}\!\!\!\big(2\big(a\!+\!(\mu(x))_i\big)\!-\!\vr\big)=\vr\,\EE
in~$\vr$.
As~$\vr$ increases from~0 to~$\vr(x)$,
the left-hand side of~\eref{CmptHam_e21b} decreases from a positive value to~0
if $w\!\neq\!0$.
Thus, $\vr(w,x)$ is well-defined, depends smoothly on~$(w,x)$, and
\BE{CmptHam_e21c}\sqrt{\vr(w,x)}=|w|\,g(w,x)\EE
for some smooth $\R^+$-valued function~$g$ on~$\C\!\times\!X$.
Define~$\wh{F}_a$ in~\eref{CmptHam_e0}~by
\BE{CmptHam_e23}
\wh{F}_a([w,1],x\big)=\wh{q}_{N+1;a}
\Big(x,\big(\sqrt{2\big(a\!+\!(\mu(x))_i\big)\!-\!\vr(w,x)}\big)_{i\in[N]},
wg(w,x)\Big).\EE
By~\eref{whpidfn_e}, \eref{CmptHam_e21b}, and~\eref{CmptHam_e21c},
$\wh{F}_a$ satisfies the second property in~\eref{CmptHam_e1a}.
By the same reasoning as in the proof of Corollary~\ref{SympNB_crl2} and~\eref{CmptHam_e6},
each map~\eref{CmptHam_e1d} is a symplectomorphism.\\

\noindent
Under the identification in~\eref{whP10dfn_e}, \eref{CmptHam_e23} becomes
\BE{CmptHam_e23b}
\wh{F}_a(r,\ne^{\fI\th},x\big)=\wh{q}_{N+1;a}
\Big(x,\big(\sqrt{2\big(a\!+\!(\mu(x))_i\big)\!-\!\rho(r,x)}\big)_{i\in[N]},
\ne^{-\fI\th}\!\sqrt{\rho(r,x)}\Big),\EE
where $\rho\!=\!\rho(r,x)$ is the unique solution of
$$\prod_{i\in[N]}\!\!\!\big(2\big(a\!+\!(\mu(x))_i\big)\!-\!\rho\big)=r^2\rho$$
in $[0,\vr(x)]$.
It extends continuously over $r\!=\!0$ as $\rho(0,x)\!=\!\vr(x)$.
Thus, \eref{CmptHam_e23b} extends continuously over $\{0\}\!\times\!S^1\!\times\!X$.
By~\eref{UIgeTor_e} and~\eref{CmptHam_e2e},  
this extension satisfies the first property in~\eref{CmptHam_e1a}.
By the continuity of both sides, it also satisfies the second property in~\eref{CmptHam_e1a}.
The~functions 
$$(r,x)\lra \sqrt{2\big(a\!+\!(\mu(x))_i\big)\!-\!\rho(r,x)}, \qquad i\!\in\![N],$$
are smooth at $(0,x)$ if the minimum in~\eref{CmptHam_e21a}
is reached at a unique $i\!\in\![N]$.
Thus, the function~\eref{CmptHam_e23b} is smooth outside of the preimage of~$f_a(X_{\prt})$.\\

\noindent
By Proposition~\ref{THamConf_prp}, the Hamiltonian $S^1$-pair 
$(\phi_{\cZ;S^1},\mu_{\cZ;S^1})$ for~$(\cZ,\om_{\cZ})$ determines 
a Hamiltonian $S^1$-pair $(\phi_{\wh\cZ;S^1;a},\mu_{\wh\cZ;S^1;a})$ 
for $(\wh\cZ_a,\om_{\wh\cZ;a})$ such~that 
$$\phi_{\wh\cZ;S^1;a}\big(\ne^{\fI\th};f_a(y)\big)
=f_a\big(\phi_{\cZ;S^1}(\ne^{\fI\th};y)\big), \quad
\mu_{\wh\cZ;S^1;a}\big(f_a(y)\big)=\mu_{\cZ;S^1}(y) \qquad
\forall\,\ne^{\fI\th}\!\in\!S^1,\,y\!\in\!\cZ_a';$$
this establishes~\eref{CmptHam_e1e}.
By~\eref{whpidfn_e}, the second property in~\eref{THamConfC_e1}, and~\eref{S1P1dfn_e},
$\phi_{\wh\cZ;S^1;a}$ also satisfies the third property in~\eref{CmptHam_e1a}.\\

\noindent
Let $(\phi_{\T},\mu_{\T})$ be a Hamiltonian $\T$-pair for $(X,\om)$
compatible with~$(\phi,\mu)$.
By Proposition~\ref{THamConf_prp} applied to $(\phi_{\cZ;\T},\mu_{\cZ;\T})$,
$(\phi_{\T},\mu_{\T})$ determines a Hamiltonian pair $(\phi_{\wh\cZ;\T},\mu_{\wh\cZ;\T})$
for $(\wh\cZ_a,\om_{\wh\cZ;a})$.
By this proposition and Lemma~\ref{THamConf_lmm2}, 
 $(\phi_{\wh\cZ;\T},\mu_{\wh\cZ;\T})$ is compatible with
$(\phi_{\wh\cZ;S^1;a},\mu_{\wh\cZ;S^1;a})$ and~$(\phi_{\wh\cZ},\mu_{\wh\cZ})$
and is intertwined   with $(\phi_{\wh\cZ;\T},\mu_{\wh\cZ;\T})$ by~$f_a$.
Since $(\phi_{\T},\mu_{\T})$ is compatible with~$(\phi,\mu)$,
the solution $\vr(w,x)$ of~\eref{CmptHam_e21b} is $\phi_{\T}$-invariant.
By~\eref{CmptHam_e23} and the construction of the induced Hamiltonian pair, 
\BE{CmptConf_e1b}
\phi_{\wh\cZ;\T}\big(g;\wh{F}_a(w,x)\!\big)=\wh{F}_a\big(w,\phi_{\T}(g;x)\!\big), \qquad
\mu_{\wh\cZ;\T}\big(\wh{F}_a(w,x)\!\big)=\mu_{\T}(x),\EE
i.e.~$\wh{F}_a$ thus intertwines $(\phi_{\T},\mu_{\T})$
and~$(\phi_{\wh\cZ;\T},\mu_{\wh\cZ;\T})$.
\end{proof}

\noindent
By the second property in~\eref{CmptHam_e1a} and 
the first property in~\eref{CmptConf_e1b},
$$\wh\pi\big(\phi_{\wh\cZ;\T}(g;y)\!\big)=\wh\pi(y).$$
An analogue~$\sC_{\wh\cZ}$ of the cutting configuration~\eref{cZSympCutDfn_e} 
for the symplectic manifold $(\wh\cZ_a,\om_{\wh\cZ;a})$ of Corollary~\ref{CmptHam_crl}
can be constructed from the initial cutting configuration~\eref{THamConf_e}
via~\eref{CmptHam_e2a}  
as in Section~\ref{cZSympCutConf_subs}.
By the same reasoning as in the proof of Proposition~\ref{THamConf_prp},
$\sC_{\wh\cZ}$ is then the cutting configuration determined by 
the Hamiltonian pair $(\phi_{\wh\cZ},\mu_{\wh\cZ})$ for  $(\wh\cZ_a,\om_{\wh\cZ;a})$.
By~\eref{CmptHam_e1d} and~\eref{CmptConf_e1b},   the restriction~\eref{CmptHam_e1d} of $\wh{F}_a$ 
induces an isomorphism
$$\wh{F}_{a;w}\!: \big(X,\om,\phi_{\T},\mu_{\T}\big)\lra  
\big(\wh\pi^{-1}(\la),\om_{\wh\cZ;a}|_{\wh\pi^{-1}(\la)},
\phi_{\wh\cZ;\T}|_{\T\times\wh\pi^{-1}(\la)},\mu_{\wh\cZ;\T}|_{\wh\pi^{-1}(\la)}\big)$$
whenever $\la\!\equiv\!q_{\P^1}(w)\!\neq\![1,0]$.
Since $\sC_{\wh\cZ}\!=\!\sC_{\phi_{\wh\cZ},\mu_{\wh\cZ}}$,
the $(\phi_{\T},\mu_{\T})\!=\!(\phi,\mu)$ case of this statement implies
that $\wh{F}_{a;w}$ identifies the cutting configuration~\eref{THamConf_e} for $(X,\om)$
with the restriction of~$\sC_{\wh\cZ}$ to~$\wh\pi^{-1}(\la)$.

\subsection{Degenerations and moment polytopes}
\label{MomPolytThm_subs}

\noindent
Suppose $\T\!\approx\!(S^1)^k$ is a $k$-torus
and $(X,\om,\phi_{\T},\mu_{\T})$ is a compact connected Hamiltonian $\T$-manifold.
By the Atiyah-Guillemin-Sternberg Convexity Theorem 
\cite[Theorem~27.1]{daSilva}, 
$$\De\equiv\mu_{\T}(X)\subset\ft^*$$ 
is then a convex  polytope.
The Hamiltonian $\T$-pair $(\phi_{\T},\mu_{\T})$ for~$(X,\om)$ gives rise to 
Hamiltonian $(S^1)^N$-pairs for $(X,\om)$;
the latter in turn determine Hamiltonian configurations for~$(X,\om)$
as in~\eref{UIdfn_e}.
In this section, we describe the effect of the constructions of this paper
with cutting configurations arising in this way on 
the moment polytope~$\De$.\\

\noindent
For $\xi\!\in\!\ft$, define
$$L_{\xi}\!:\ft^*\lra\R,\quad L_{\xi}(\eta)=\lr{\eta,\xi}, \qquad
h_{\xi}\!\equiv\!L_{\xi}\!\circ\!\mu_{\T}\!: X\lra\R, \quad
h_{\xi}(x)=L_{\xi}\big(\mu_{\T}(x)\big).$$
A vector $\xi\!\in\!\ft$ is called \sf{integral} if 
its time $2\pi$-flow in~$\T$ generates a circle subgroup $S^1_{\xi}\!\subset\!\T$.
An~integral vector~$\xi$ determines a homomorphism and an $S^1$-action,
$$ \vr_{\xi}\!: S^1\lra S^1_{\xi}
\qquad\hbox{and}\qquad 
\phi_{\xi}\!: S^1\!\times\!X\lra X,$$
respectively;
the latter is the composition of the $\T$-action~$\phi_{\T}$ with the former.
The action~$\phi_{\xi}$ commutes with the $\T$-action and has Hamiltonian~$h_{\xi}$.
We denote by $\La_{\ft}\!\subset\!\ft$ the lattice of integral vectors and by
$\La_{\ft}^*\!\subset\!\ft^*$ the dual lattice.
An integral vector $\xi\!\in\!\ft$ is called \sf{primitive} if 
$\xi/m$ is not integral for any integer $m\!>\!1$.
The homomorphism~$\vr_{\xi}$ is injective if and only~if $\xi$ is primitive.\\

\noindent
Suppose $\xi\!\in\!\La_{\ft}$ is primitive, 
$\ep\!\in\!\R$ is a regular value of~$h_{\xi}$, and 
the $S^1_{\xi}$-action on the hypersurface
$$\wt{V}_{\ep}\equiv \big\{x\!\in\!X\!:\, h_\xi(x)\!=\!\ep\big\}\subset X$$
is free.
By \cite[Remark~1.5]{L}, the 2-fold symplectic cut construction with 
the $S^1_{\xi}$-action and Hamiltonian $h_{\xi}\!-\!\ep$ then
cuts~$X$ along~$\wt{V}_\ep$ into two symplectic manifolds, $(X_-,\om_-)$ and $(X_+,\om_+)$,
with Hamiltonian $\T$-actions and moment polytopes
$$\De_-=\big\{\eta\!\in\!\De\!: L_{\xi}(\eta)\!\le\!\ep\big\}
\qquad\hbox{and}\qquad
 \De_+=\big\{\eta\!\in\!\De\!: L_{\xi}(\eta)\!\ge\!\ep\big\},$$
respectively.\\

\noindent
Our $N$-fold symplectic cut construction produces more complicated subdivisions of 
the moment polytope~$\De$.
Fix a tuple $(\xi_i,\ep_i)_{i\in[N]}$ in $(\La_{\ft}\!\times\!\R)^N$.
For $i\!\in\![N]$ and $a\!\in\!\R$, define
\BE{wtDeadfn_e}\begin{split}
L_i=L_{\xi_i}\!-\!\ep_i\!:\ft^*\lra\R, \quad
\De_i&\equiv\!\De_i\big((\xi_j,\ep_j)_{j\in[N]}\big)
=\big\{\eta\!\in\!\De\!:\,L_i(\eta)\!\le\!L_j(\eta)~\forall\,j\!\in\![N]\big\},\\
\wh\De_a\equiv\!\wh\De_a\big((\xi_j,\ep_j)_{j\in[N]}\big)
&=\big\{(\eta,u)\!\in\!\De\!\times\!\R\!:\,-\min_{i\in[N]}L_i(\eta)\le u\le a\big\}.
\end{split}\EE
For $I\!\in\!\cP^*(N)$, let
$$\De_I\!\equiv\!\De_I\big((\xi_i,\ep_i)_{i\in[N]}\big)
=\big\{\eta\!\in\!\De\!:L_i(\eta)\!\le\!L_j(\eta)~\forall\,i\!\in\!I,\,j\!\in\![N]\big\}.$$
For our purposes, a pair $(\xi_i,\ep_i)$ such that the polytope~$\De_i$ is empty
can be dropped from the consideration (thus reducing the value of~$N$).\\

\noindent
The polytopes $\De_i$ with $i\!\in\![N]$ subdivide~$\De$;
the first diagram in Figure~\ref{P1P1_fig} 
shows such a subdivision for the data of Example~\ref{P1P1_eg}.
For generic choices of $\ep_i\!\in\!\R$, the intersection of each~$\De_I$ with a facet of~$\De$ 
is a polytope of codimension $|I|\!-\!1$ in the facet.
By Proposition~\ref{CombCond_prp} in the next section, this property needs to hold for 
the $N$-fold Hamiltonian configuration~\eref{toriccut_e} for~$(X,\om)$
determined by the tuple $(\xi_i,\ep_i)_{i\in[N]}$ to be a cutting configuration.
By Theorem~\ref{toriccut_thm} below,
the polytopes~$\De_i$ are then the moment polytopes of the Hamiltonian $\T$-manifolds
$(X_i,\om_i,\phi_{\T;i},\mu_{\T;i})$ determined by the construction of
Section~\ref{Thm12Pf_sec}.
If 
\BE{abnd_e2} a>-\min_{i\in[N]}\min_{\eta\in\De}L_i(\eta)\,,\EE
then $\wh\De_a$ is the polytope rising from the graph of the function
$$\De\lra\R, \qquad \eta\lra -\min_{i\in[N]}\!\big(L_i(\eta)\!\big)\,,$$
to the ``horizontal" hyperplane $\ft^*\!\times\!\{a\}$ in $\ft^*\!\times\!\R$.
This graph consists of $N$~polytopes~$\De_i'$ with $\pi_{\ft^*}(\De_i')\!=\!\De_i$,
where
$$\pi_{\ft^*}\!: \ft^*\!\oplus\!\R\lra \ft^*$$
is the projection map.
The second diagram in Figure~\ref{P1P1_fig} shows such a polytope~$\wh\De_a$ 
for the data of Example~\ref{P1P1_eg}.
If~\eref{toriccut_e} is a cutting configuration and~\eref{abnd_e2} holds, then 
$\wh\De_a$ is the moment polytope of the symplectic manifold $(\wh\cZ_a,\om_{\wh\cZ;a})$
of Corollary~\ref{CmptHam_crl} with the Hamiltonian $(\T\!\times\!S^1)$-pair
obtained by combing the pairs $(\phi_{\wh\cZ;\T},\mu_{\wh\cZ;\T})$
and $(\phi_{\wh\cZ;S^1;a},\mu_{\wh\cZ;S^1;a})$.\\

\noindent
With $(\xi_i,\ep_i)_{i\in[N]}$ as above and $\io_i$ as in the proof of Lemma~\ref{THamConf_lmm2},
we define a Hamiltonian $(S^1)^N$-pair
for $(X,\om)$~by 
\BE{Ttophimu_e}\begin{aligned}
\phi\!:(S^1)^N\!\times\!X&\lra X, &\qquad
\phi\big(\io_i(\ne^{\fI\th});x\big)&=\phi_{\xi_i}(\ne^{\fI\th};x\big), \\
\mu\!:X &\lra\R^N, &\qquad
\mu(x)&=\big(L_i(\mu_{\T}(x))\!\big)_{\!i\in[N]}.
\end{aligned}\EE
Let 
\BE{toriccut_e}
\sC_{\phi,\mu}\equiv
\sC_{\phi_{\T},\mu_{\T}}\big((\xi_i,\ep_i)_{i\in[N]}\big)
\equiv \big(U_I,\phi_I,\mu_I\big)_{I\in\cP^*(N)}\EE
be the associated maximal $N$-fold Hamiltonian configuration for~$(X,\om)$ as
in~\eref{THamConf_e}.
In particular,
\BE{UIdfneg2_e}
U_I^{\le}\equiv\big\{x\!\in\!X\!:L_i(\mu_{\T}(x))\!\le\!L_j(\mu_{\T}(x))\,
\forall\,i\!\in\!I,\,j\!\in\![N]\big\}
=\mu_{\T}^{-1}\big(\De_I\big) \quad\forall\,I\!\in\!\cP^*(N).\EE
For an arbitrary Hamiltonian $(S^1)^N$-pair for~$(X,\om)$,
the configuration~\eref{THamConf_e} is the case of~\eref{toriccut_e}
with $(\phi_{\T},\mu_{\T})\!=\!(\phi,\mu)$, $\xi_i\!\in\!\Z^N$ being
the standard $i$-th coordinate vector, and $\ep_i\!=\!0$.\\

\noindent
We call a tuple $(\xi_i,\ep_i)_{i\in [N]}$ as above \sf{regular} 
if the associated $(S^1)^N$-Hamiltonian manifold $(X,\om,\phi,\mu)$
is regular as defined above Lemma~\ref{THamConf_lmm}.
Proposition~\ref{CombCond_prp} in Section~\ref{AdmissDecomp_subs} provides 
a geometric interpretation of 
this criterion in terms of the Delzant condition of \cite[Definition~28.1]{daSilva}.

\begin{numthm}\label{toriccut_thm}
Suppose $\T$ is a $k$-torus, $(X,\om,\phi_{\T},\mu_{\T})$ 
is a compact connected Hamiltonian $\T$-space, 
and $(\xi_i,\ep_i)_{i\in [N]}$ is a regular tuple in $(\La_{\ft}\!\times\!\R)^N$.
The tuple~\eref{toriccut_e} is then a maximal $N$-fold cutting configuration for~$(X,\om)$. 
For every $a\!\in\!\R$ sufficiently large, this tuple determines
\begin{enumerate}[label=(\arabic*),leftmargin=*]

\item a compact symplectic manifold $(\wh\cZ_a,\om_{\wh\cZ;a})$ containing 
the tuple $(X_i,\om_i)_{i\in[N]}$ of  the cut symplectic manifolds of
Corollary~\ref{SympNB_crl2} as an SC symplectic divisor~$X_{\eset}$ and 

\item a Hamiltonian $(\T\!\times\!S^1)$-pair $(\phi_{\wh\cZ;a},\mu_{\wh\cZ;a})$
such~that 
\BE{toriccutprp_e1}
\phi_{\wh\cZ;a}\big(\T\!\times\!S^1\!\times\!X_i\big)=X_i
~~\forall\,i\!\in\![N],\qquad
\mu_{\wh\cZ;a}\big(\wh\cZ_a\big)=\wh\De_a\,.\EE
\end{enumerate}
For all $i\!\in\!I\!\subset\![N]$, the preimage of~$\De_I$ 
under $\pi_{\ft^*}\!\circ\!\mu_{\wh\cZ;a}|_{X_i}$
 is  the symplectic submanifold $(X_I,\om_I)$ of~$(X_i,\om_i)$ determined by~\eref{toriccut_e}.
If the $\phi_{\T}$-action is effective, then so is the $\phi_{\wh\cZ;a}$-action.
The deformation equivalence class of smoothings of~$X_{\eset}$
determined by~\eref{toriccut_e} is represented by
a nearly regular symplectic fibration  $\wh\pi\!:\wh\cZ_a\!\lra\!\P^1$
equivariant with respect to the projection $\T\!\times\!S^1\!\lra\!S^1$.
\end{numthm}

\begin{proof}
Since the tuple $(\xi_i,\ep_i)_{i\in [N]}$ is regular, 
\eref{toriccut_e} is a maximal $N$-fold cutting configuration for~$(X,\om)$
by Lemma~\ref{THamConf_lmm}.
By Theorem~\ref{SympCut_thm12}, it thus determines a symplectic manifold~$(\cZ,\om_{\cZ})$
and a~tuple $(X_i,\om_i)_{i\in[N]}$ of symplectic manifolds contained in~$(\cZ,\om_{\cZ})$
as an SC symplectic divisor~$X_{\eset}$. 
For  $a\!\in\!\R$ sufficiently large, let $(\wh\cZ_a,\om_{\wh\cZ;a})$,
$$\wh\pi\!:\wh\cZ_a\lra\P^1, \qquad f_a\!:\cZ_a'\lra\wh\cZ_a,$$
$(\phi_{\wh\cZ;S^1;a},\mu_{\wh\cZ;S^1;a})$, and $(\phi_{\wh\cZ;\T},\mu_{\wh\cZ;\T})$
be the associated objects provided by  Corollary~\ref{CmptHam_crl}.
In particular, $f_a$ embeds~$X_{\eset}$ as an SC symplectic divisor 
into the compact symplectic manifold $(\wh\cZ_a,\om_{\wh\cZ;a})$.
The required condition~\eref{CmptHam_e7} on $a\!\in\!\R$ in this case becomes~\eref{abnd_e2}.\\

\noindent
Since  the Hamiltonian $\T$-pair $(\phi_{\wh\cZ;\T},\mu_{\wh\cZ;\T})$ and 
the Hamiltonian $S^1$-pair $(\phi_{\wh\cZ;S^1;a},\mu_{\wh\cZ;S^1;a})$ are compatible,
they determine a Hamiltonian $(\T\!\times\!S^1)$-pair $(\phi_{\wh\cZ;a},\mu_{\wh\cZ;a})$
with 
$$\mu_{\wh\cZ;a}\!=\!\big(\mu_{\wh\cZ;\T},\mu_{\wh\cZ;S^1;a}\big)\!:
\wh\cZ_a\lra \ft^*\!\oplus\!\R.$$
Since $\wh\pi$ is $S^1$-equivariant, it is 
equivariant with respect to the projection \hbox{$\T\!\times\!S^1\!\lra\!S^1$}. 
Since the $\phi_{\wh\cZ;\T}$- and $\phi_{\wh\cZ;S^1;a}$-actions preserve $X_i\!\subset\!\wh\cZ_a$,
so does the $\phi_{\wh\cZ;a}$-action.
By Corollary~\ref{CmptHam_crl} and  Proposition~\ref{THamConf_prp},
\BE{toriccut_e9}\mu_{\cZ;\T}=\mu_{\wh\cZ;\T}\!\circ\!f_a \qquad\hbox{and}\qquad
\mu_{\T}=\mu_{\cZ;\T}\!\circ\!q_{\eset}\,,\EE
respectively.
Combining these statements with~\eref{UIdfneg2_e} and Corollary~\ref{SympNB_crl2}, we obtain
\begin{equation*}\begin{split}
\big\{\pi_{\ft^*}\!\circ\!\mu_{\wh\cZ;a}\big\}^{-1}(\De_I)\cap X_{\eset}
&\equiv\mu_{\wh\cZ;\T}^{\,-1}(\De_I) \cap f_a(X_{\eset})
=f_a\big(\mu_{\cZ;\T}^{-1}(\De_I)\!\cap\!X_{\eset}\big)
=f_a\big(\mu_{\cZ;\T}^{-1}(\De_I)\!\cap\!q_{\eset}(X)\big)\\
&=f_a\big(q_{\eset}\big(\mu_{\T}^{-1}(\De_I)\big)\!\big)
=f_a\big(q_{\eset}(U_I^{\le})\big)=f_a(X_I)\equiv X_I 
\end{split}\end{equation*}
for every $I\!\in\!\cP^*(N)$.\\

\noindent
With the notation as in~\eref{CmptHam_e8}, define
\begin{alignat}{2}
\label{toriccut_e15a}
\wt\phi_{\wh\cZ;a}\!: \T\!\times\!S^1\!\times\!\wt\cZ_{N+1;a}&\lra \wt\cZ_{N+1;a},
&\quad \wt\phi_{\wh\cZ;a}\big(g,\ne^{\fI\th};x,z,w\big)&=\big(\phi_{\T}(g;x),z,\ne^{-\fI\th}w\big),\\
\label{toriccut_e15b}
\wt\mu_{\wh\cZ;a}\!:\wt\cZ_{N+1;a}&\lra\ft^*\!\oplus\!\R, &\qquad
\wt\mu_{\wh\cZ;a}(x,z,w)&=\Big(\mu_{\T}(x),a\!-\!\frac12|w|^2\Big).
\end{alignat}
By construction, the pair $(\wt\phi_{\wh\cZ;a},\wt\mu_{\wh\cZ;a})$ descends to 
the pair $(\phi_{\wh\cZ;a},\mu_{\wh\cZ;a})$ on the quotient in~\eref{CmptHam_e8}.
By~\eref{toriccut_e15b} and~\eref{toriccut_e9}, $\pi_{\ft^*}(\mu_{\wh\cZ;a}(\wh\cZ_a))\!=\!\De$.
Since \hbox{$(\mu(x))_i\!=\!L_i(\mu_{\T}(x))$},
\eref{CmptHam_e8} and~\eref{toriccut_e15b} give
$$\big\{u\!\in\!\R\!: (\eta,u)\!\in\!\mu_{\wh\cZ;a}(\wh\cZ_a)\big\}
=\big\{u\!\in\!\R\!:  -L_i(\eta)\!\le\!u\!\le\!a~\forall\,i\!\in\![N]\big\}
\quad\forall\,\eta\!\in\!\De.$$
This establishes the second statement in~\eref{toriccutprp_e1}.\\
 
\noindent
Suppose the $\phi_{\wh\cZ;a}$-action of $(g,\ne^{\fI\th})\!\in\!\T\!\times\!S^1$ 
on~$\wh\cZ_a$ is trivial.
Let
$$z\!=\!(z_i)_{i\in[N]}\in\C^N, \qquad
\wt{y}\equiv\big(x,z,w\big)\in\wt\cZ_{N+1;a}, \qquad\hbox{and}\quad
\big(\ne^{\fI\th_i}\big)_{i\in[N]}\in(S^1)^N$$
be such that
\BE{toriccut_e25}z_i\neq0~~\forall\,i\!\in\![N],\quad w\neq0, \quad
\wt\phi_{\wh\cZ;a}\big(g,\ne^{\fI\th};x,z,w\big)=
\wt\phi_{N+1}\big((\ne^{\fI\th_i})_{i\in[N]};x,z,w\big).\EE
By~\eref{toriccut_e15a} and the first and last assumptions in~\eref{toriccut_e25}, 
$\ne^{\fI\th_i}\!=\!1$
for every $i\!\in\![N]$.
By the second assumption  in~\eref{toriccut_e25}, this in turn implies that
$\ne^{\fI\th}\!=\!1$.
Since the projection 
$$\wt\cZ_{N+1;a}\cap\big(X\!\times\!(\C^{N+1}\!-\!\C^{N+1}_0)\big)\lra X$$
is surjective, it follows that  the $\phi_{\T}$-action of~$g$ on~$X$ is trivial.
If the~$\phi_{\T}$-action is effective, the $\phi_{\wh\cZ;a}$-action is thus effective as well.
\end{proof}

\noindent
The compatible Hamiltonian pairs $(\phi_{\cZ;\T},\mu_{\cZ;\T})$ and 
$(\phi_{\cZ;S^1},\mu_{\cZ;S^1})$  for $(\cZ,\om_{\cZ})$
provided by Proposition~\ref{THamConf_prp} and Lemma~\ref{THamConf_lmm2}
give rise to a Hamiltonian $(\T\!\times\!S^1)$-pair 
$(\phi_{\cZ;\i},\mu_{\cZ;\i})$
for the symplectic manifold $(\cZ,\om_{\cZ})$ determined by~\eref{toriccut_e}
if the tuple $(\xi_i,\ep_i)_{i\in[N]}$ is regular.
The corresponding moment ``polytope"~is
$$\mu_{\cZ;\i}(\cZ)=\wh\De_{\i}\!\equiv\!
\big\{(\eta,u)\!\in\!\De\!\times\!\R\!:\,-\min_{i\in[N]}L_i(\eta)\le u\big\}.$$
The moment polytope of  $(\wh\cZ_a,\om_{\wh\cZ;a},\phi_{\wh\cZ;a},\mu_{\wh\cZ;a})$
 is  obtained by cutting off this infinite ``polytope"  
at the level $u\!=\!a$ of the moment map for the  $S^1$-action~$\phi_{\cZ;S^1}$,
as expected from the proof of Corollary~\ref{CmptHam_crl} and \cite[Remark~1.5]{L}.\\

\noindent
By Corollary~\ref{CmptHam_crl}, the fibers~$(\wh\cZ_{a;\la},\om_{\wh\cZ;a;\la})$ of 
the map~$\wh\pi$ in Theorem~\ref{toriccut_thm}
over $\P^1\!-\!\{[1,0]\}$ with the Hamiltonian $\T$-pair $(\phi_{\wh\cZ;\T},\mu_{\wh\cZ;\T})$
are canonically isomorphic to~$(X,\om)$ with the pair~$(\phi_{\T},\mu_{\T})$.
The fibers over~$[1,0]$ and~$[0,1]$ are preserved by the full $\T\!\times\!S^1$-action
$\phi_{\wh\cZ;a}$ on~$\wh\cZ_a$.
The restriction of the $S^1$-action to the latter is in fact trivial;
this is reflected in the ``top" face of the polytope~$\wh\De_a$ in~\eref{wtDeadfn_e} being ``horizontal".
The fibers over $\C^*\!\subset\!\P^1$ are not preserved by the $S^1$-action.
Their images \hbox{$\mu_{\wh\cZ;a}(\wh\cZ_{a;\la})\!\subset\!\wh\De_a$}  
do not depend the angular component of~$\la$.
The restriction of $\pi_{\ft^*}$ to $\mu_{\wh\cZ;a}(\wh\cZ_{a;\la})$ is surjective onto~$\De$
and has one-dimensional fibers.\\
 
\noindent
A compact connected Hamiltonian $\T$-manifold $(X,\om,\phi_{\T},\mu_{\T})$ 
is a \sf{symplectic toric manifold}
if the $\T$-action~$\phi_{\T}$ is effective and $\dim_{\R}\!X\!=\!2\dim_{\R}\!\T$.
By Theorem~\ref{toriccut_thm}, $(\wh\cZ_a,\om_{\wh\cZ;a},\phi_{\wh\cZ;a},\mu_{\wh\cZ;a})$
is a symplectic toric manifold if $(X,\om,\phi_{\T},\mu_{\T})$~is.
The projection 
\hbox{$\ft\!\oplus\!\R\!\lra\!\R$}
then induces a projection from the toric fan of~$\wh\cZ_a$ to the toric fan of~$\P^1$. 
The projection~$\wh\pi$ in Theorem~\ref{toriccut_thm} is
the projective morphism induced by~$\pi_{\R}$;
see \cite[Proposition VII.1.16]{Audin}.
We can think of~$\wh\pi$ as a one-parameter family 
of K\"ahler manifolds smoothing $\wh\cZ_{a;0}\!=\!X_{\eset}$ into $\wh\cZ_{a;\i}\!\approx\!X$. 
The vertical edges of the polytope~$\wh\De_a$ in~\eref{wtDeadfn_e} correspond to 
holomorphic sections of~$\wh\pi$.\\

\noindent
The configuration~\eref{toriccut_e}, 
which determines the output of Theorems~\ref{SympCut_thm12} and~\ref{SympCut_thm3}, 
depends only on the restriction of 
the Hamiltonian $(S^1)^N$-pair~$(\phi,\mu)$ defined by~\eref{Ttophimu_e} to
the subtorus $(S^1)^N_{\bu}\!\subset\!(S^1)^N$.
The latter is determined by a tuple $(\xi_{ij},\ep_{ij})_{i,j\in[N]}$ in
 $(\La_{\ft}\!\times\!\R)^{N\times N}$ satisfying the one-cocycle condition
$$\big(\xi_{ij},\ep_{ij}\big)+\big(\xi_{jk},\ep_{jk}\big)=
\big(\xi_{ik},\ep_{ik}\big) \qquad\forall~i,j,k\!\in\![N].$$
A tuple $(\xi_i,\ep_i)_{i\in[N]}$ in $(\La_{\ft}\!\times\!\R)^N$ 
cobounding $(\xi_{ij},\ep_{ij})_{i,j\in[N]}$, i.e.
\BE{coboundcond_e}\big(\xi_{ij},\ep_{ij}\big)=\big(\xi_j,\ep_j\big)-\big(\xi_i,\ep_i\big)
\qquad\forall\,i,j\!\in\![N],\EE
determines an extension of the Hamiltonian $(S^1)^N_{\bu}$-pair
 determined by $(\xi_{ij},\ep_{ij})_{i,j\in[N]}$ to an \hbox{$(S^1)^N$-pair}.
If $(\xi_i',\ep_i')_{i\in[N]}$ is another tuple satisfying~\eref{coboundcond_e}, then
there exists~$(\xi,\ep)$ in $\La_{\ft}\!\times\!\R$ such~that
$$\big(\xi_i',\ep_i'\big)=\big(\xi_i,\ep_i\big)+(\xi,\ep) \qquad\forall\,i\!\in\![N].$$
Along with~\eref{THamConf_e27a} and~\eref{THamConf_e27b}, this implies that
$$\phi_{\cZ;S^1}'\big(\ne^{\fI\th};y\big)=
\phi_{\cZ;\T}\big(\rho_{\xi}(\ne^{-\fI\th});\phi_{\cZ;S^1}(\ne^{\fI\th};y)\big), \quad
\mu_{\cZ;S^1}'(y)=\mu_{\cZ;S^1}(y)\!-\!L_{\xi}\big(\mu_{\cZ;\T}(y)\big)\!+\!\ep\,.$$
As indicated by Example~\ref{P1cutN1_eg},  the deformation equivalence class of
the symplectic manifold $(\wh\cZ_a,\om_{\wh\cZ;a})$ of Theorem~\ref{toriccut_thm} in general 
depends on the choice of such a coboundary.

\begin{eg}\label{P1cutN1_eg}
Let $\om_{\P^1}$ denote the doubled Fubini-Study symplectic form on~$\P^1$, i.e.
$$\om_{\P^1}\big|_z=\frac{2\,\om_{\C}}{(1\!+\!|z|^2)^2} \qquad\forall~z\!\in\!\C\!\subset\!\P^1.$$
We take $(X,\om)\!=\!(\P^1,\om_{\P^1})$, $\T\!=\!S^1$, 
$\phi_{\T}\!=\!\phi_{\P^1}$ as in~\eref{S1P1dfn_e},  
$$\mu_{\T}\!\equiv\!\mu_{\P^1}:\P^1\lra\R, ~~ 
\mu_{\P^1}\big([w,z]\big)=\frac{|z|^2}{|w|^2\!+\!|z|^2}\,,
\quad N=1,  \quad \xi_1=m\in\Z\!=\!\La_{\ft}, \quad \ep_1=0.$$
Thus, $\cZ\!=\!\P^1\!\times\!\C$, $\wh\De_a$ is the polytope in the first diagram of 
Figure~\ref{P1cut_fig}, and
$$\wh\pi\!:\wh\cZ_a\!=\!\bF_m\lra\P^1$$
is the $m$-th Hirzebruch surface with the canonical projection; see \cite[Homework~22.3]{daSilva}. 
We conclude that $(\wh\cZ_a,\om_{\wh\cZ;a})$ depends on the parity of~$m$ in this case.
\end{eg}

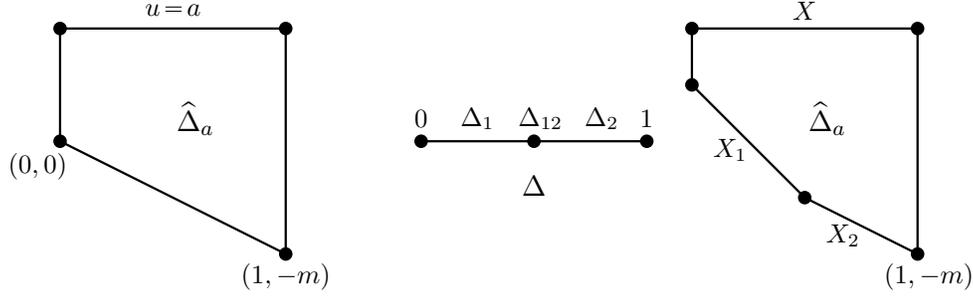
\begin{figure}
\begin{pspicture}(-3.5,-.3)(10,3)
\psset{unit=.3cm}
\psline[linewidth=.1](-3,10)(7,10)(7,0)(-3,5)(-3,10)
\rput(2,10.8){\sm{$u\!=\!a$}}
\pscircle*(-3,10){.3}\pscircle*(7,10){.3}\pscircle*(7,0){.3}\pscircle*(-3,5){.3}
\rput(3,6){$\wh\De_a$}\rput(7,-1){\sm{$(1,-m)$}}\rput(-4,3.9){\sm{$(0,0)$}}
\psline[linewidth=.1](13,5)(23,5)
\pscircle*(13,5){.3}\pscircle*(18,5){.3}\pscircle*(23,5){.3}
\rput(13,6){\sm{$0$}}\rput(23,6){\sm{$1$}}
\rput(15.5,6){\sm{$\De_1$}}\rput(21,6){\sm{$\De_2$}}\rput(18.3,6){\sm{$\De_{12}$}}
\rput(18,3){$\De$}
\psline[linewidth=.1](35,10)(25,10)(25,7.5)(30,2.5)(35,0)(35,10)
\rput(30,10.8){\sm{$X$}}
\pscircle*(35,0){.3}\pscircle*(30,2.5){.3}\pscircle*(25,7.5){.3}
\pscircle*(35,10){.3}\pscircle*(25,10){.3}
\rput(31,6){$\wh\De_a$}\rput(35.5,-1){\sm{$(1,-m)$}}
\rput(26.7,4.6){\sm{$X_1$}}\rput(31.7,.8){\sm{$X_2$}}
\end{pspicture}
\caption{The polytope~$\wh\De_a$ for the one-parameter ``degeneration" of symplectic manifolds
for the data of Example~\ref{P1cutN1_eg},
the subdivision of~$\De$ into two polytopes for the data of Example~\ref{P1cut_eg},
and the polytope~$\wh\De_a$ for the associated one-parameter family 
of symplectic manifolds degenerating $X\!=\!\P^1$
into $X_{\eset}\!=\!\P^1\!\vee\!\P^1$.}
\label{P1cut_fig}
\end{figure}

\begin{eg}\label{P1cut_eg}
With $(X,\om,\phi_{\T},\mu_{\T})$ as in Example~\ref{P1cutN1_eg} and $m\!\in\!\Z$, 
we now~take 
$$N=2, \quad \xi_1=m\!+\!1, \quad \xi_2=m, \quad \ep_1=\frac12, \quad \ep_2=0.$$
The tuple $(\xi_i,\ep_i)_{i\in[2]}$ is then regular.
The polytopes~$\De$ for $(\phi_{\T},\mu_{\T})$,
$\De_i$ for the cut symplectic manifolds~$(X_i,\om_i)$ with $i\!=\!1,2$,
and $\wh\De_a$ for the symplectic manifold $(\wh\cZ_a,\om_{\wh\cZ;a})$
containing $X_1\!\cup\!X_2$ as an SC symplectic divisor are shown in
the second and third diagrams of Figure~\ref{P1cut_fig}.
By \cite[Homework~22.2]{daSilva}, $(\wh\cZ_a,\om_{\wh\cZ;a},\phi_{\wh\cZ;a},\mu_{\wh\cZ;a})$
is a symplectic manifold that can be constructed as either 
a toric blowup of~$\bF_m$ or a toric blowup of~$\bF_{m+1}$.
By Delzant's Classification Theorem \cite[Theorem~28.2]{daSilva},
this quadruple  depends on the choice of~$m$.
However, the deformation equivalence class of the symplectic manifold $(\wh\cZ_a,\om_{\wh\cZ;a})$
is independent of this choice.
\end{eg}

\begin{eg}\label{P1P1_eg}
With the notation as in Example~\ref{P1cutN1_eg}, we now take
\begin{gather*}
X=\P^1\!\times\!\P^1, \qquad \om=3\pi_1^*\om_{\P^1}\!+\!2\pi_2^*\om_{\P^1},
\qquad \T\!=\!(S^1)^2, \\
\phi_{\T}\!=\!\phi_{\P^1}\!\times\!\phi_{\P^1}\!:\T\!\times\!X\lra X,
\qquad 
\mu_{\T}=\big(3\mu_{\P^1}\!\circ\!\pi_1\!-\!2,2\mu_{\P^1}\!\circ\!\pi_2\!-\!1\big)\!:
X\lra\R^2,\\
\xi_1=(0,0), \qquad \xi_2=(1,0), \qquad \xi_3=(0,1), 
\qquad \ep_1,\ep_2,\ep_3\!=\!0.
\end{gather*}
The associated subdivision of~$\De\!=\![-2,1]\!\times\![-1,1]$ is  
shown in the first diagram of Figure~\ref{P1P1_fig}.
It corresponds to a 3-fold symplectic cut 
of $\P^1\!\times\!\P^1$ into  $X_1\!=\!\P^1\!\times\!\P^1$, 
a one-point blowup~$X_2$ of $\P^1\!\times\!\P^1$, and $X_3\!=\!\bF_1$. 
The smooth divisors $X_{12},X_{13}\!\subset\!X_1$ are one of the horizontal lines and 
one of the vertical lines and thus have normal bundles of degree~0.
The smooth divisors $X_{12},X_{23}\!\subset\!X_2$ are the proper transform of a ruling 
of~$\P^1\!\times\!\P^1$ through the blowup point and the exceptional divisor, respectively,
and thus have normal bundles of degree~$-1$.
The smooth divisors $X_{13},X_{23}\!\subset\!X_3$ are the exceptional divisor and a fiber
and thus have normal bundles of degrees~$-1$ and~$0$, respectively.
Since $X_{123}$ is a single point in this case, the restrictions
of the line bundle~$\cO_{X_{\prt}}(X_{\eset})$ in~\eref{PsiDfn_e} to $X_{12},X_{13},X_{23}$ 
are thus of degree~0.
Since $X_{12},X_{13},X_{23}\!\approx\!\P^1$ are simply connected, 
this line bundle has a unique homotopy class of trivializations.
The second diagram in Figure~\ref{P1P1_fig} shows the polytope~\eref{wtDeadfn_e}
for the associated symplectic toric $\T\!\times\!S^1$-manifold
$(\wh\cZ_a,\om_{\wh\cZ;a},\phi_{\wh\cZ;a},\mu_{\wh\cZ;a})$. 
\end{eg}

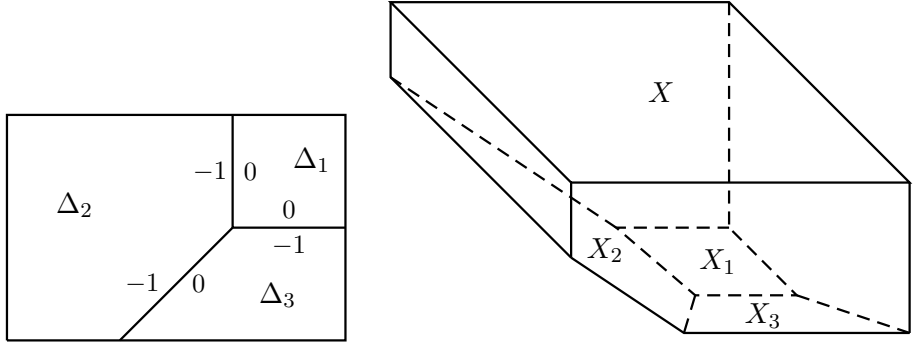
\begin{figure}
\begin{pspicture}(1.2,0.5)(10,5)
\psset{unit=.3cm}
\psline[linewidth=.1](10,2)(25,2)(25,12)(10,12)(10,2)
\psline[linewidth=.1](20,12)(20,7)(25,7)
\psline[linewidth=.1](20,7)(15,2)
\rput(23.5,10){$\De_1$}\rput(13,8){$\De_2$}\rput(22,4){$\De_3$}
\rput(20.8,9.5){\sm{$0$}}\rput(22.5,7.8){\sm{$0$}}
\rput(16,4.5){\sm{$-1$}}\rput(19,9.5){\sm{$-1$}}
\rput(22.5,6.2){\sm{$-1$}}\rput(18.5,4.5){\sm{$0$}}
\psline[linewidth=.1](35,9)(50,9)(42,17)(27,17)(35,9)
\psline[linewidth=.1,linestyle=dashed](42,17)(42,7)(45,4)(40.5,4)(37,7)(42,7)
\psline[linewidth=.1,linestyle=dashed](45,4)(50,2.33)
\psline[linewidth=.1](50,2.33)(50,9)
\psline[linewidth=.1,linestyle=dashed](40.5,4)(40,2.33)
\psline[linewidth=.1,linestyle=dashed](37,7)(27,13.67)
\psline[linewidth=.1](35,5.67)(27,13.67)(27,17)
\psline[linewidth=.1](35,5.67)(35,9)
\psline[linewidth=.1](50,2.33)(40,2.33)(35,5.67)
\rput(41.5,5.5){$X_1$}\rput(36.5,6){$X_2$}\rput(43.5,3.1){$X_3$}
\rput(39,13){$X$}
\end{pspicture}
\caption{The subdivision of $\De$ into three polytopes for the data of Example~\ref{P1P1_eg}
and the polytope~$\wh\De_a$ for the associated one-parameter family of symplectic manifolds 
degenerating~$X$ (the top face) into the SC symplectic variety $X_1\!\cup\!X_2\!\cup\!X_3$.
The numbers next to each edge $\De_{ij}\!=\!\De_i\!\cap\!\De_j$ in the first diagram
are the degrees of the normal bundles of the symplectic submanifold
$(X_{ij},\om_{ij})$ in the symplectic manifolds $(X_i,\om_i)$.}
\label{P1P1_fig}
\end{figure}

\subsection{Admissible decompositions of polytopes}
\label{AdmissDecomp_subs}

\noindent
Suppose $\T$ is a $k$-torus  and $\De\!\subset\!\ft^*$ is a convex polytope
as before.
For each $v\!\in\!\De$, let $\fd_v(\De)\!\in\!\Z^{\ge0}$ be the dimension
of the minimal facet~$\De_v$ of~$\De$ containing~$v$.
If  $v\!\in\!\De$ is a vertex,
let $\E_v(\De)$ be the set of edges of~$\De$ containing~$v$. 
For each $e\!\in\!\E_v(\De)$, we denote by $e/v\!\in\!\De$
the vertex of~$e$ other than~$v$.
A vertex~$v$ of a polytope $\De\!\subset\!\ft^*$ is called~\sf{smooth}
if there exist a $\Z$-basis $\{u_e\}_{e\in\E_v(\De)}$ for~$\La_{\ft}^*$ 
and a~tuple $(t_{e,v})_{e\in\E_v(\De)}$ in $\R^{\E_v(\De)}$ so that 
$$e/v=v\!+\!t_{e,v}u_e  \in \ft^* \qquad\forall~e\!\in\!\E_v(\De).$$
The smoothness of a vertex $v\!\in\!\De$ implies that $|\E_v(\De)|\!=\!k$
and the slope of each edge $e\!\in\!\E_v(\De)$ is rational with respect to 
the lattice~$\La_{\ft}^*$.
A polytope $\De\!\subset\!\ft^*$ is \sf{Delzant} if every vertex $v\!\in\!\De$ is smooth.\\

\noindent
Let $(X,\om,\phi_{\T},\mu_{\T})$ be a compact connected Hamiltonian $\T$-manifold
with moment polytope~$\De$.
For a tuple  $(\xi_i,\ep_i)_{i\in [N]}$  in $(\La_{\ft}\!\times\!\R)^N$,
let $\De_i\!\subset\!\De$ be as in~\eref{wtDeadfn_e}.
For each $v\!\in\!\De$, define
$$I_v=\{i\!\in\![N]\!:\,v\!\in\!\De_i\big\}\,.$$
We call $(\xi_i,\ep_i)_{i\in [N]}$ a \sf{quasi-regular} tuple  
 if every polytope $\De_i\!\subset\!\De$ is Delzant,
$\xi_i\!-\!\xi_j$ is a primitive element of~$\La_{\ft}$ 
for all $i,j\!\in\![N]$ distinct such that $\De_{ij}\!\neq\!\eset$, and
\BE{DeTrans_e}
|I_v|\le\fd_v(\De)\!+\!1 \qquad\forall~v\!\in\!\De.\EE
The last condition is equivalent to the same condition 
for all vertices~$v$ of the polytopes~$\De_i$ with $i\!\in\![N]$.
By dimensional considerations, the inequality in~\eref{DeTrans_e} for any given vertex $v\!\in\!\De_i$
is equivalent to the equality for the same vertex.
An example of a subdivision for a quasi-regular tuple is shown in 
the first diagram in Figure~\ref{P1P1_fig}.\\

\noindent
By~\eref{DeTrans_e},  the combinatorics of the subdivision of~$\De$ determined
by a quasi-regular tuple  $(\xi_i,\ep_i)_{i\in[N]}$ does not change under
sufficiently small changes in the values of~$\ep_i$.
The next proposition relates the combinatorial notion of quasi-regularity for 
 $(\xi_i,\ep_i)_{i\in [N]}$ to the geometric notion of regularity for
the induced Hamiltonian configuration~\eref{toriccut_e}.

\begin{prp}\label{CombCond_prp}
Suppose $N\!\in\!\Z^+$,
$(X,\om,\phi_{\T},\mu_{\T})$ is a compact connected Hamiltonian $\T$-manifold, 
and $(\xi_i,\ep_i)_{i\in [N]}$ is a tuple in $(\La_{\ft}\!\times\!\R)^N$.
If this tuple is regular, then it is quasi-regular.
If it satisfies~\eref{DeTrans_e} and $(\ep_i)_{i\in [N]}$ is generic, then
 the $\phi_I$-action of $(S^1)^I_{\bu}$ on $\mu_I^{-1}(0)$ has at most finite
stabilizers for every $I\!\in\!\cP^*(N)$.
If $(X,\om,\phi_{\T},\mu_{\T})$ is a toric symplectic manifold 
and the tuple $(\xi_i,\ep_i)_{i\in [N]}$ 
is quasi-regular, then it is regular.
\end{prp}

\noindent
For $i,j\!\in\![N]$, let $\xi_{ij}\!=\!\xi_i\!-\!\xi_j$.
For each $v\!\in\!\ft^*$, let
$$\ft_v\equiv\big\{\xi\!\in\!\ft\!:\,L_{\xi}(\eta\!-\!v)\!=\!0~\forall\,\eta\!\in\!\De_v\big\}$$
be the annihilator of the vectors contained in~$\De_v$.
By \cite[Corollary~IV.4.13]{Audin}, $\ft_v$ is the $\R$-span of a sublattice of $\La_{\ft}$ and
thus generates a subtorus~$\T_v\subset\T$.
For $I\!\in\!\cP^*(N)$, define
\begin{gather*}
L_{\ft^*;I}\!:\ft^*\lra\ft_{I;\bu}^*=\R^I\big/\big\{(a,\ldots,a)\!\in\!\R^I\!:\,a\!\in\!\R\big\},
\quad \quad
L_{\ft^*;I}(\eta)=\big[\big(L_i(\eta)\big)_{\!i\in I}\big],\\
\cK_{\ft^*;I}=\big\{\eta\!\in\!\ft^*\!:\,
L_{\xi_i}(\eta)\!=\!L_{\xi_j}(\eta)~\forall\,i,j\!\in\!I\big\}.
\end{gather*}
Thus,
\begin{gather}
\label{CombCond_e2b}
\Ann\big(\cK_{\ft^*;I}\big)\equiv
\big\{\xi\!\in\!\ft\!:\,L_{\xi}(\eta)\!=\!0~\forall\,\eta\!\in\!\cK_{\ft^*;I}\big\}
=\Span_{\R}\big\{\xi_{ij}\!:i,j\!\in\!I\big\},\\
\notag
\ker\!\big(\nd_vL_{\ft^*;I}\!:T_v\ft^*\!\lra\!\ft_{I;\bu}^*\big)
=\cK_{\ft^*;I}\subset\ft^*
\qquad\forall~v\!\in\!\ft^*.
\end{gather}
Since the dimension of the space on the right-hand side of~\eref{CombCond_e2b} is 
at most $|I|\!-\!1$,
\BE{CombCond_e2a}
\dim\!\big(T_v\De_v\!\cap\!\cK_{\ft^*;I_v}\big)
\ge \fd_v(\De) -\dim\!\big(\Span_{\R}\big\{\xi_{ij}\!:i,j\!\in\!I_v\big\}\!\big)
\ge \fd_v(\De)\!+\!1\!-\!|I_v| \quad\forall~v\!\in\!\ft^*\,.\EE
For $x\!\in\!X$, let
\BE{Txftx_e}\T_x=\big\{g\!\in\!\T\!:\phi_{\T}(g;x)\!=\!x\big\} \quad\hbox{and}\quad
\ft_x=\Ann\big(\Im\big(\nd_x\mu_{\T}\!:T_xX\lra T_v\ft^*\big)\big)\EE
be the stabilizer of~$x$ and the Lie algebra of~$\T_x$, respectively.

\begin{lmm}\label{CombCond_lmm}
Let $N\!\in\!\Z^+$ and $(\xi_j,\ep_j)_{j\in [N]}\!\in\!(\La_{\ft}\!\times\!\R)^N$.
The  conditions~\eref{DeTrans_e},
the homomorphism
\BE{CombCond_e7} \nd_vL_{\ft^*;I_v}\!: T_v\De_v\lra\ft_{I;\bu}^*\EE
is surjective for all $v\!\in\!\De$,
and the kernel of the homomorphism
\BE{CombCond_e11a}
\T_v\!\times\!(S^1)^{I_v}_{\bu}\lra\T, \qquad
\big(g,(\ne^{\fI\th_i})_{i\in I_v}\big)\lra 
g\prod_{i\in I_v}\!\vr_{\xi_i}\!\big(\ne^{\fI\th_i}\big),\EE
is finite for all $v\!\in\!\De$ are equivalent.
\end{lmm}

\begin{proof}
We show that each of the three conditions of Lemma~\ref{CombCond_lmm} is equivalent~to
\BE{CombCond_e4a}
\dim\!\big(\Span_{\R}\big\{\xi_{ij}\!:i,j\!\in\!I_v\big\}\!\big)=|I_v|\!-\!1, ~~
\ft_v\cap \Span_{\R}\big\{\xi_{ij}\!:i,j\!\in\!I_v\big\}=\big\{0\big\}
\quad\forall\,v\!\in\!\De\,.\EE
Suppose $v\!\in\!\De$. 
The first equality in~\eref{CombCond_e4a} is equivalent to the second inequality
in~\eref{CombCond_e2a} being an equality; 
the second equality in~\eref{CombCond_e4a} is equivalent to the first inequality
in~\eref{CombCond_e2a} being an equality.
Thus, the two equalities in~\eref{CombCond_e4a} for a fixed $v\!\in\!\De$ are equivalent~to
\BE{CombCond_e6b}\dim\!\big(T_v\De_v\!\cap\!\cK_{\ft^*;I_v}\big)=\fd_v(\De)\!+\!1\!-\!|I_v|\EE
and thus to the surjectivity of the homomorphism~\eref{CombCond_e7}.\\

\noindent
The kernel of the homomorphism~\eref{CombCond_e11a}
is finite if and only if the homomorphism
$$\ft_v\oplus\ft_{I_v;\bu}\lra\ft, \qquad
\big(\xi,(r_i)_{i\in I_v}\big)\lra \xi+\sum_{i\in I_v}\!r_i\xi_i\,,$$
is injective.
By~\eref{ftNnu_e}, the latter is the case if and only if~\eref{CombCond_e4a} holds.\\

\noindent
If $v$ is a vertex of~$\De_i$ for some $i\!\in\![N]$, then  
$$T_v\De_v\!\cap\!\cK_{\ft^*;I_v}=\{0\} \qquad\hbox{and}\qquad
\fd_v(\De)\le |I_v|\!-\!1.$$
The inequality~\eref{DeTrans_e} for a given vertex $v\!\in\!\De_i$ is thus equivalent 
to~\eref{CombCond_e6b} and so to the two equalities in~\eref{CombCond_e4a} for~$v$.
If $v\!\in\!\De_i$ is arbitrary and
$v'\!\in\!(\De_i)_v$ is a vertex of the minimal facet of~$\De_i$ containing~$v$,
then 
$$I_{v'}\supset I_v, \quad   
\Span_{\R}\big\{\xi_{ij}\!:i,j\!\in\!I_{v'}\big\}\supset
\Span_{\R}\big\{\xi_{ij}\!:i,j\!\in\!I_v\big\}, \quad
\ft_{v'}\supset\ft_v.$$
The two equalities in~\eref{CombCond_e4a} with~$v$ replaced by~$v'$ thus 
imply the two equalities in~\eref{CombCond_e4a} themselves.
Since the same is also the case for the inequality in~\eref{DeTrans_e}, 
we conclude that
the conditions~\eref{DeTrans_e} and~\eref{CombCond_e4a} are equivalent.
\end{proof}

\begin{lmm}\label{CombCond_lmm2}
Let $N\!\in\!\Z^+$ and $(\xi_j,\ep_j)_{j\in [N]}\!\in\!(\La_{\ft}\!\times\!\R)^N$.
The tuple $(\xi_j,\ep_j)_{j\in [N]}$ is quasi-regular if and only if 
the homomorphism~\eref{CombCond_e11a} is injective
for all $v\!\in\!\De$.
\end{lmm}

\begin{proof}
Let $v\!\in\!\De$ and $i\!\in\!I_v$.
By~\eref{ftNnu_e}, the injectivity of the homomorphism~\eref{CombCond_e11a}
is equivalent~to
\BE{CombCond_e19}
\La_{\ft}\cap\big(\ft_v\!+\!\Span_{\R}\big\{\xi_{ij}\!:j\!\in\!I_v\big\}\big)
=\La_{\ft_v}\oplus\bigoplus_{j\in I_v-i}\!\!\!\Z\xi_{ij}\,.\EE
If \eref{CombCond_e19} holds, then $\xi_{ij}\!\in\!\La_t$ is a primitive element for~every  
$j\!\in\!I_v\!-\!i$.
If in addition $v$ is a vertex of~$\De_i$, then \eref{CombCond_e19} implies that 
it is smooth.
Along with Lemma~\ref{CombCond_lmm}, this implies that 
the tuple $(\xi_j,\ep_j)_{j\in [N]}$ is quasi-regular.\\

\noindent
Suppose the tuple $(\xi_j,\ep_j)_{j\in [N]}$ is quasi-regular.
Let $i\!\in\![N]$ and  $v\!\in\!\De_i$.
By the proof of Lemma~\ref{CombCond_lmm}, the two equalities in~\eref{CombCond_e4a} hold.
Thus, \eref{CombCond_e19} holds when tensored with~$\R$.
Since $\xi_{ij}\!\in\!\La_t$ is primitive  for~every  $j\!\in\!I_v\!-\!i$
and a vertex $v'\!\in\!(\De_i)_v$ is smooth,
\eref{CombCond_e19} itself holds as~well.
\end{proof}

\begin{lmm}\label{CombCond_lmm3}
Suppose $N\!\in\!\Z^+$, $(\xi_j,\ep_j)_{j\in [N]}\!\in\!(\La_{\ft}\!\times\!\R)^N$
 satisfies~\eref{DeTrans_e},  $v\!\in\!\De$, and  $x\!\in\!\mu^{-1}(v)$.
The homomorphism
\BE{CombCond_e9a}\nd_x\mu_{I_v}\!:T_xX\lra\ft_{I_v;\bu}^*\EE
is surjective if and only~if the kernel of the homomorphism
\BE{CombCond_e11}\T_x\!\times\!(S^1)^{I_v}_{\bu}\lra\T, \qquad
\big(g,(\ne^{\fI\th_i})_{i\in I_v}\big)\lra 
g\prod_{i\in I_v}\!\vr_{\xi_i}\!\big(\ne^{\fI\th_i}\big),\EE
is finite.
\end{lmm}

\begin{proof}
By the definitions of~$I_v$ and~$\De_i$, $(\mu(x))_i\!\le\!(\mu(x))_j$ for all
$i\!\in\!I_v$ and $j\!\in\![N]$ and the equality holds if and only if $j\!\in\!I_v$.
Thus, $x\!\in\!U_{I_v}$ and 
the homomorphism~\eref{CombCond_e9a} is well-defined. 
We show that each condition of Lemma~\ref{CombCond_lmm3} is equivalent~to
\BE{CombCond_e4d}
\ft_x\cap \Span_{\R}\big\{\xi_{ij}\!:i,j\!\in\!I_v\big\}=\big\{0\big\}\,.\EE
By Lemma~\ref{CombCond_lmm} and its proof,
the homomorphism~\eref{CombCond_e7} is surjective and 
both equalities in~\eref{CombCond_e4a} hold in the present case.\\

\noindent
Since $\mu_{I_v}\!=\!L_{\ft^*;I_v}\!\circ\!\mu_{\T}$,
\BE{CombCond_e8d}\nd_x\mu_{I_v}\!=\!\nd_vL_{\ft^*;I_v}\!\circ\!\nd_x\mu_{\T}\!: 
T_xX\lra T_v\ft^*\lra\ft_{I_v;\bu}^*\,.\EE
By the definition of~$\De_v$,
\BE{CombCond_e3a}  
\Im\big(\nd_x\mu_{\T}\!:T_xX\lra T_v\ft^*\big)\subset T_v\De_v\,.\EE
By the surjectivity of the homomorphism~\eref{CombCond_e7},  
the surjectivity of the homomorphism~\eref{CombCond_e9a} is thus equivalent~to
$$T_v\De_v\subset \Im\big(\nd_x\mu_{\T}\big)+\cK_{\ft^*;I_v} \,.$$
By~\eref{CombCond_e2b} and~\eref{Txftx_e}, this condition is in turn equivalent~to
$$\ft_v\supset \ft_x\cap \Span_{\R}\big\{\xi_{ij}\!:i,j\!\in\!I_v\big\}\,.$$
In light of the second equality in~\eref{CombCond_e4a},  
the last condition is equivalent to~\eref{CombCond_e4d}.\\

\noindent
The finiteness of the kernel of~\eref{CombCond_e11}
is equivalent to the injectivity of the~homomorphism
$$\ft_x\oplus\ft_{I_v;\bu}\lra\ft, \qquad
\big(\xi,(r_i)_{i\in I_v}\big)\lra \xi+\sum_{i\in I_v}\!r_i\xi_i.$$
By~\eref{ftNnu_e} and the first equality in~\eref{CombCond_e4a},  
the latter is equivalent to~\eref{CombCond_e4d}.
\end{proof}

\begin{proof}[{\bf{\emph{Proof of Proposition \ref{CombCond_prp}}}}]
Let  $v\!\in\!\De$ and $x\!\in\!\mu^{-1}(v)$.
Thus, $x\!\in\!\mu_{I_v}^{-1}(0)$ and \hbox{$\T_x\!\supset\!\T_v$}.
Furthermore, the projection 
$$\T_x\!\times\!(S^1)^{I_v}_{\bu}\lra (S^1)^{I_v}_{\bu}$$
restricts to an isomorphism from 
the kernel of the homomorphism~\eref{CombCond_e11}
to the stabilizer of the $\phi_{I_v}$-action on~$x$.
If this stabilizer is trivial, then 
the homomorphisms~\eref{CombCond_e11} and~\eref{CombCond_e11a} are injective.
Since $\mu^{-1}(v)\!\neq\!\eset$ for every $v\!\in\!\De$,
Lemma~\ref{CombCond_lmm2} thus implies the first claim of Proposition~\ref{CombCond_prp}.\\

\noindent
Suppose the tuple $(\xi_j,\ep_j)_{j\in [N]}$ satisfies~\eref{DeTrans_e}
and the homomorphism~\eref{CombCond_e9a} is surjective.
By Lemma~\ref{CombCond_lmm3}, the homomorphism~\eref{CombCond_e11} then has finite kernel and so
the stabilizer of the $\phi_{I_v}$-action on~$x$ is finite.
The homomorphism~\eref{CombCond_e9a} is surjective if 
\hbox{$[(\ve_i)_{i\in I}]\!\in\!\ft_{I;\bu}^*$} 
is a regular value of the smooth map~$\mu_I$ on~$U_I$ for all $I\!\in\!\cP^*(N)$.
By Sard's Theorem \cite[p10]{Milnor}, this is the case 
if the tuple $(\ve_i)_{i\in[N]}$ is generic.
This establishes the second claim of Proposition~\ref{CombCond_prp}.\\  

\noindent
By \eref{CombCond_e8d} and Lemma~\ref{CombCond_lmm}, 
the homomorphism~\eref{CombCond_e9a} is surjective if
the tuple $(\xi_i,\ep_i)_{i\in [N]}$ satisfies~\eref{DeTrans_e}
and the inclusion~\eref{CombCond_e3a} is an  equality.
Suppose  $(X,\om,\phi_{\T},\mu_{\T})$ is a toric symplectic manifold.
By \cite[Corollary~IV.4.14]{Audin}, the inclusion~\eref{CombCond_e3a} is then an  equality.
Furthermore, $\T_x\!=\!\T_v$.
Since the inclusion~\eref{CombCond_e3a} is an  equality, 
this statement is equivalent to the connectedness of~$\T_x$. 
The latter is implied by each toric symplectic manifold being 
the quotient of a subset~$\wt{X}_M^{\tau}$ of~$\C^{k'}$ 
by the restriction of the standard $(S^1)^{k'}$-action 
to an action of a $(k'\!-\!k)$-subtorus $\T'\!\subset\!(S^1)^{k'}$
with $\T\!=\!(S^1)^{k'}/\T'$ and $\phi_{\T}$ being the induced action;
see \cite[Section~2.1]{Popa} for the relevant details. 
Thus, the stabilizer~$\T_x$ of the point $x\!\in\!X$ determined by 
a point $\wt{x}\!\in\!\wt{X}_M^{\tau}$ is the quotient of
the stabilizer $(S^1)^{k'}_{\wt{x}}\!\subset\!(S^1)^{k'}$ of~$\wt{x}$
by its intersection with~$\T'$.
Since  $(S^1)^{k'}_{\wt{x}}$ is connected, so is~$\T_x$.\\

\noindent
Suppose $(X,\om,\phi_{\T},\mu_{\T})$ is a toric symplectic manifold and the tuple 
$(\xi_i,\ep_i)_{i\in [N]}$ is quasi-regular.
By the previous paragraph, the homomorphism~\eref{CombCond_e9a} is then surjective. 
Since $\T_x\!=\!\T_v$,
the triviality of the stabilizer of the $\phi_{I_v}$-action on~$x$ 
is equivalent to the injectivity of the homomorphism~\eref{CombCond_e11a}.
The latter is the case by Lemma~\ref{CombCond_lmm2}.
This establishes the last claim of Proposition~\ref{CombCond_prp}.
\end{proof}

\noindent
The example below provides decompositions of moment polytopes for 
Hamiltonian $S^1$-manifolds $(X,\om,\phi_{\T},\mu_{\T})$ with effective actions.
These decompositions arise from quasi-regular tuples $(\xi_i,\ep_i)_{i\in[2]}$
that are not regular.
Example~\ref{notsmooth_eg} illustrates that 
assuming that the action is Hamiltonian and 
the tuple $(\xi_i,\ep_i)_{i\in[N]}$ is quasi-regular
is not sufficient to overcome either of the two deficiencies of the second statement 
of Proposition~\ref{CombCond_prp} as compared to the third.
On the other hand, our degeneration and symplectic cut construction can be applied whenever 
 the $\phi_I$-action on $\mu_I^{-1}(0)$ has at most finite
stabilizers for all $I\!\in\!\cP^*(N)$.
It would then produce symplectic orbifolds.
Thus, the combinatorics of polytope decompositions in the context of 
Theorem~\ref{toriccut_thm} fit most naturally 
with the category of symplectic orbifolds (rather than just manifolds).

\begin{eg}\label{notsmooth_eg}
Let $\om_{\P^1}$ and $(\phi_{\P^1},\mu_{\P^1})$ be as in Example~\ref{P1cutN1_eg}.
Let
$$(X,\om)=\big(\P^1\!\times\!\P^1,\pi_1^*\om_{\P^1}\!+\!\pi_2^*\om_{\P^1}\big)
\qquad\hbox{and}\qquad \T=S^1.$$
Fix $m_1,m_2\!\in\!\Z^+$ and define
\begin{alignat*}{2}
\phi_{\T}\!:S^1\!\times\!X&\lra X, &\qquad 
\phi_{\T}\big(\ne^{\fI\th};z_1,z_2\big)&=
\big(\phi_{\P^1}(\ne^{\fI m_1\th};z_1\big),\phi_{\P^1}(\ne^{\fI m_2\th};z_2)\big),\\
\mu_{\T}\!:X&\lra\R, &\qquad 
\mu_{\T}(z_1,z_2)&=m_1\mu_{\P^1}(z_1)\!+\!m_2\mu_{\P^1}(z_2)\,.
\end{alignat*}
The tuple $(X,\om,\phi_{\T},\mu_{\T})$ is then  a Hamiltonian $S^1$-manifold.
The associated moment polytope~$\De$ is the interval $[0,m_1\!+\!m_2]$.
Two of the four $\phi_{\T}$-fixed points, $P_{\i0}$ and~$P_{0\i}$, are mapped
to the interior points~$m_1$ and~$m_2$.
A setting for the usual $N\!=\!2$ symplectic cut configuration of~\cite{L}
in this case is obtained by taking
$$\xi\equiv\xi_2\!-\!\xi_1=1\in\Z\!=\!\La_{\R}
\qquad\hbox{and}\qquad \ep=
\ep_1\!-\!\ep_2\in(0,m_1\!+\!m_2\big).$$
The associated decomposition breaks~$\De$ into the intervals 
$\De_1\!=\![0,\ep]$ and $\De_2\!=\![\ep,m_1\!+\!m_2]$.
In this case, $\phi_{12}\!=\!\phi_{\T}$ under a suitable identification $(S^1)^2_{\bu}\!=\!S^1$.
Since $\nd_x\mu_{\T}\!=\!0$  if and only if $x$ is a $\phi_{\T}$-fixed point,
the $\phi_{12}$-action on $\mu_{12}^{-1}(0)\!=\!\mu_{\T}^{-1}(\ep)$ is non-trivial 
if and only~if $\ep\!\neq\!m_1,m_2$.
This illustrates the necessity of the ``generic" assumption in Proposition~\ref{CombCond_prp}.
Since the subgroups $\Z_{m_1},\Z_{m_2}\!\subset\!S^1$ act trivially on $\P^1\!\times\!\{0,\i\}$
and $\{0,\i\}\!\times\!\P^1$, respectively,
the $\phi_{12}$-action on $\mu_{12}^{-1}(0)$ is not free for any 
$\ep\!\in\!(0,m_1\!+\!m_2)$ unless $m_1,m_2\!=\!1$.
However, the $\phi_{\T}$-action on~$X$ is effective if~$m_1$ and~$m_2$ are relatively prime.
\end{eg}

\vspace{.1in}

\vbox{
\noindent
{\it Simons Center for Geometry and Physics, Stony Brook University, Stony Brook, NY 11794\\
mtehrani@scgp.stonybrook.edu}\\

\noindent
{\it Department of Mathematics, Stony Brook University, Stony Brook, NY 11794\\
azinger@math.stonybrook.edu}}


\end{document}